\documentclass{amsart}
\usepackage{basic-preamble}
\allowdisplaybreaks[1]
\AtEndDocument{\vspace*{\stretch{1}}}

\begin{document}
\title{Group Gradings on Classical Lie Superalgebras}

\author{Caio De Naday Hornhardt}
\address{Department of Mathematics and Statistics, Memorial University of Newfoundland, St. John's, NL, A1C5S7, Canada}
\email{cdnh22@mun.ca}
\thanks{This paper is based on the first author's Ph.D. thesis, written under the supervision of the second author. The first author acknowledges the support of Memorial University and of the Natural Sciences and Engineering Research Council (NSERC) of Canada.}

\author{Mikhail Kochetov}
\address{Department of Mathematics and Statistics, Memorial University of Newfoundland, St. John's, NL, A1C5S7, Canada}
\email{mikhail@mun.ca}
\thanks{The second author is supported by Discovery Grant 2018-04883 of NSERC}

\begin{abstract}
    We classify, up to isomorphism, the group gradings on the non-exceptional classical simple Lie superalgebras, except for type~$A(1,1)$, over an algebraically closed field of characteristic zero.
    To this end, we study graded-simple and graded-superinvolution-simple associative superalgebras satisfying the descending chain condition on graded left superideals, which allows us to classify abelian group gradings on finite-dimensional simple and superinvolution-simple associative superalgebras over an algebraically closed field of characteristic different from $2$.
\end{abstract}



\maketitle

\tableofcontents

\section{Introduction}

The use of group gradings in Lie theory traces back to 1888 (see \cite{MR1510529}), when W.~Killing introduced the root space decomposition of complex semisimple Lie algebras.
Later, $\ZZ_2$-gradings appeared in E.~Cartan's work on real semisimple Lie algebras (see \cite{Cartan-1914}). 
Interest in gradings intensified in the 1960s through works of J.~Tits, I.L.~Kantor, and M.~Koecher (see \cite{Tit62,Kan64,Koe67}). 
Systematic classification of group gradings on Lie algebras began with \cite{PZ} and remains an active research area in Lie theory and representation theory.
Building on the work of several authors, including \cite{HPP,BSZ01,BZ02,BZ03,BSZ05,BZ06,BZ07}, all group gradings for classical simple Lie algebras (series $A$, $B$, $C$, $D$) over an algebraically closed field of characteristic different from $2$ were classified up to isomorphism in \cite{BK10}, while fine gradings were classified up to equivalence in \cite{Eld10} (see also \cite{livromicha} and references therein).
These classifications are also known over the field of real numbers \cite{paper-adrian,ELDUQUE202261}.

The first appearances of Lie superalgebras were related to cohomology \cite{FN56,Gen63,Gen64,MM65}, but they were independently introduced in physics in relation to supersymmetries \cite{GN64,Miy68,Mic69}. 
V.~Kac classified the finite-dimensional simple Lie superalgebras over an algebraically closed field of characteristic $0$ \cite{artigokac}, dividing them into two types: classical and Cartan type (see also \cite{livrosuperalgebra}).
The classical Lie superalgebras are divided into the series $A(m,n)$, $B(m,n)$, $C(n)$, $D(m,n)$, $P(n)$, and $Q(n)$, as well as three exceptional cases: $F(4)$, $G(3)$, and the family $D(2,1,\alpha)$ with a scalar parameter $\alpha \neq 0, -1$.  
The Lie superalgebras of Cartan type are divided into the series $W(n)$, $S(n)$, $\tilde S(n)$ and $H(n)$.

The $\mathbb{Z}$-gradings on classical Lie superalgebras where classified in \cite{kacZ}. 
In \cite{serganova}, gradings by finite cyclic groups were considered and the corresponding twisted loop superalgebras were classified. 
Fine gradings on the exceptional classical simple Lie superalgebras were classified up to equivalence in \cite{artigoelduque}.

Classifications of all group gradings up to isomorphism on Lie superalgebras were first obtained in \cite{paper-Qn} for the series $Q$, in \cite{paper-MAP} for the series $P$ and partially $A$ (what we call Type I gradings), and in \cite{Helens_thesis} for the series $B$.

In the present work, our main goal is to classify, up to isomorphism, the group gradings on non-exceptional classical Lie superalgebras.
To this end, we study graded-simple and graded-superinvolution-simple associative superalgebras that are (left) graded-Artinian (\ie, satisfy the descending chain condition on graded left superideals).
Our method is a combination of \cite{livromicha}, where graded-simple Artinian associative algebras and their anti-automorphisms were studied, and of \cite{racine}, which classified primitive associative superalgebras with superinvolution that have a minimal one-sided superideal.
In particular, this method allows us to classify abelian group gradings on finite-dimensional simple and superinvolution-simple associative superalgebras over an algebraically closed field of characteristic different from $2$.
Taking advantage of the fact that the nonexceptional classical Lie superalgebras in characteristic $0$ of type different from $A(1,1)$ can be realized in terms of skew elements in superinvolution-simple associative superalgebras in such a way that the automorphism groups match, we can then transfer the classification of gradings from the associative to Lie superalgebras.
This gives a uniform classification of gradings for all types other than $A(1,1)$.
This latter has an exceptionally large automorphism group, so our transfer approach does not apply.
Group gradings on exceptional classical Lie superalgebras as well as $A(1,1)$ were classified up to isomorphism in \cite{gradingsExceptional}.

The paper is structured as follows.
In \cref{sec:preliminaries}, we establish basic definitions and notation, as well as some preliminary results. 
\Cref{sec:grdd-simple-ass} develops the theory of graded-simple graded-Artinian associative superalgebras over a field $\FF$, concluding with the classification of abelian group gradings on finite-dimensional simple associative superalgebras for algebraically closed $\FF$, $\Char \FF \neq 2$. 
For superalgebras with superinvolution, \cref{sec:grdd-sinv-simple} presents the general theory of graded-superinvolution-simple graded-Artinian superalgebras, which is applied in \Cref{sec:gradings-on-vphi-simple} to finite-dimensional superinvolution-simple superalgebras for algebraically closed $\FF$, $\Char \FF \neq 2$. 
Finally, \Cref{sec:Lie} transfers these classifications to non-exceptional classical Lie superalgebras for algebraically closed $\FF$, $\Char \FF = 0$.

\section{Preliminaries}\label{sec:preliminaries}

\subsection{Gradings and Superalgebras}\label{subsec:grds}

To discuss gradings on (super)algebras, we start with gradings on vector spaces:

\begin{defi}\label{defi:grading}
    Let $G$ be a group. 
	A $G$-\emph{grading on a vector space} $V$ is a direct sum decomposition
	\[\label{eq:grading}
	    \Gamma\from V= \bigoplus_{g \in G} V_g,
	\]
	indexed by the elements of $G$.
    If $\Gamma$ is fixed, we say that $V$ is \emph{$G$-graded} or simply \emph{graded} if $G$ is understood from the context. 
    The support, denoted either by $\supp \Gamma$ or $\supp V$, is the subset $\{g\in G \mid V_g \neq 0\} \subseteq G$.
    The subspaces $V_g$ are called \emph{homogeneous components} and their elements are called \emph{homogeneous elements}.
    If $0 \neq v \in V_g$, we say that the \emph{degree} of $v$ is $g$ and write $\deg v = g$.
\end{defi}

Given two $G$-graded vector spaces $V = \bigoplus_{g \in G} V_g$ and $W = \bigoplus_{g \in G} W_g$, a linear map $f\from V \to W$ is said to be a \emph{homomorphism of $G$-graded vector spaces} if it is degree-preserving, \ie, if $f(V_g) \subseteq W_g$ for all $g\in G$.
A subspace $U \subseteq V$ is said to be a \emph{graded subspace} if
\[
    U = \bigoplus_{g\in G} U_g, \text{ where } U_g \coloneqq U \cap V_g\,.
\]
Then the decomposition above is a $G$-grading on $U$ such that the inclusion map $U \hookrightarrow V$ is a homomorphism of $G$-graded spaces and, also, $V/U$ becomes a $G$-graded vector space in such way that the natural map $V \to V/U$ is a homomorphism of $G$-graded vector spaces.
Furthermore, we consider $V \times W$ and $V \tensor W$ as $G$-graded vector spaces by defining $(V\times W)_g \coloneqq V_g \times W_g$ and by setting, for all  non-zero homogeneous elements $v\in V$ and $w\in W$,
\[
    \deg (v \tensor w) \coloneqq (\deg v) (\deg w)\,.
\]

\begin{defi}\label{defi:grdd-algebra}
    A $G$-\emph{grading on an algebra} $A$ is a grading $\Gamma\from A= \bigoplus_{g \in G} A_g$ on its underlying vector space with the extra condition that
    \[
        \forall g,h\in G, \quad A_gA_h \subseteq A_{gh}\,,
    \]
    \ie, the multiplication map $A\tensor A \to A$ is a homomorphism of $G$-graded vector spaces.
\end{defi}

A {homomorphism of $G$-graded algebras} is a degree-preserving homomorphism of algebras.
A \emph{graded subalgebra} (respectively, \emph{graded ideal}) is a subalgebra (respectively, ideal) that is a graded subspace.
A $G$-graded algebra $A$ is said to be \emph{graded-simple} if $AA \neq 0$ and the only graded ideals are $0$ and $A$.

In the study of superalgebras, there is a canonical grading by $\ZZ_2$, which we will distinguish from other gradings by using superscripts instead of subscripts; there is also special terminology associated with this grading.
Thus, a \emph{vector superspace}, or simply \emph{superspace}, is a $\ZZ_2$-graded vector space, $U = U\even \oplus U\odd$, where the elements of $U\even$ and $U\odd$ are said to be, respectively, \emph{even} and \emph{odd}, so the degree for this grading is called \emph{parity}.
We will write $|u| = i$ if $0 \neq u\in U^i$.

\begin{defi}\label{defi:superspace}
    A \emph{$G$-grading on a superspace} $U$ is a grading $\Gamma\from U = \bigoplus_{g\in G} U_g$ on its underlying vector space that is compatible with the canonical $\ZZ_2$-grading in the sense that every homogeneous component is a $\ZZ_2$-graded subspace, \ie,
    \[
        \forall g\in G,\quad U_g = (U_g\cap U\even)\oplus(U_g\cap U\odd)\,,
    \]
    or, equivalently, $U\even$ and $U\odd$ are $G$-graded subspaces.
    A \emph{superalgebra} is a $\ZZ_2$-graded algebra $A = A\even \oplus A\odd$ and a \emph{$G$-grading on a superalgebra} $A$ is one that makes $A$ both a $G$-graded algebra and a $G$-graded superspace.
\end{defi}

\begin{remark}\label{rmk:G-sharp}
    Letting $G^\# \coloneqq G\times \ZZ_2$, we note that a $G$-graded superspace (superalgebra) is the same as $G^\#$-graded vector space (algebra), with $U_{(g,i)} \coloneqq U_g\cap U^i$ for all $g\in G$ and $i \in \ZZ_2$.
\end{remark}

We define \emph{homomorphisms of ($G$-graded) superalgebras} as homomorphisms of $(G\times)\ZZ_2$-graded algebras.
Similarly, we define \emph{($G$-graded) subsuperalgebras} (respectively, \emph{superideals}) as $(G\times)\ZZ_2$-graded subalgebras (respectively, ideals) and \emph{simple ($G$-graded) superalgebras} as simple $(G\times)\ZZ_2$-graded algebras.

The above definitions do not address a crucial aspect of ``super'' structures, known as the \emph{sign rule}, which is relevant to the identities imposed on multiplication to define classes of superalgebras such as Lie superalgebras.
Many definitions in algebra (\eg, the definition of anti-automorphism), involve interchanging the position of elements ($\vphi(ab) = \vphi(b)\vphi(a)$).
In the ``super'' context, a minus sign is introduced if odd elements are interchanged with each other.
This procedure can be formalized through the concept of \emph{symmetric monoidal category} (see, \eg, \cite[Chapter 3]{MR2069561}), but in this work it will be sufficient to apply the sign rule on a case-by-case basis, as in \cref{defi:super-anti-auto,defi:Lie-SA} below.

\begin{defi}\label{defi:super-anti-auto}
    Let $R$ and $S$ be ($G$-graded) superalgebras.
    A \emph{super-anti-isomorphism} is an isomorphism of ($G$-graded) superspaces $\vphi\from R \to S$ such that
    \[
        \forall a,b \in R\even \cup R\odd,\quad \vphi(ab) = (-1)^{|a||b|}\vphi(b)\vphi(a)\,.
    \]
    If $R = S$, we say that $\vphi$ is a \emph{super-anti-automorphism}.
    If, further, $\vphi^2 = \id$, we say that $\vphi$ is a \emph{superinvolution}.
\end{defi}

We note that, if $\Char \FF = 2$, then \cref{defi:super-anti-auto} reduces to the definition of anti-isomorphism, anti-automorphism, and involution.
Thus, we are mainly interested in the case $\Char \FF \neq 2$.

A \emph{$G$-grading on a superalgebra with super-anti-automorphism $(R, \vphi)$} is a $G$-grading on $R$ for which $\vphi$ is degree-preserving.
A \emph{homomorphism between ($G$-graded) superalgebras with super-anti-automorphism} is a homomorphism preserving parities, (degrees) and the super-anti-automorphisms.
A ($G$-graded) superalgebra with superinvolution $(R, \vphi)$ is said to be \emph{(graded-)superinvolution-simple} if $RR \neq 0$ and the only $\vphi$-invariant (graded) superideals are $0$ and $R$.

\begin{defi}\label{defi:Lie-SA}
	A superalgebra $L=L\even\oplus L\odd$ with product $[\cdot, \cdot]\from L\times L \to L$ is said to be a \emph{Lie superalgebra} if, for all nonzero homogeneous elements $a, b, c \in L$, we have\footnote{Additional conditions are often required if $\Char \FF = 2$ or $3$ (\eg, \cite[\S 1.2]{MR1192546}).}:
	\begin{enumerate}
		\item $[ a, b ] = - \sign{a}{b} [b, a]$ (\emph{super-anti-commutativity});
		\item $[a,[b,c]] = [[a,b],c] + \sign{a}{b} [b, [a,c]]$ (\emph{super Jacoby identity}).
	\end{enumerate}
\end{defi}

Given an associative superalgebra $R$, we define the Lie superalgebra $R^{(-)}$ to be the same superspace as $R$ but with product given by the \emph{supercommutator}:
\[
    \forall a,b\in R\even\cup R\odd, \quad [a,b] \coloneqq ab - \sign{a}{b} ba\,.
\]
If $\vphi$ is a super-anti-automorphism on $R$, then we also define the Lie subsuperalgebra $\Skew(R, \vphi) \coloneqq \{ r\in R \mid \vphi(r) = - r \} \subseteq R^{(-)}$.

For the main results in this work, we assume that the grading group $G$ is abelian.  
In our context, this is without loss of generality because the group generated by the support is abelian for graded-simple Lie algebras (see \cite[Lemma 2.1]{BZ06}, \cite[Proposition 1]{MR2257580}, or [Propositions 1.12]\cite{livromicha}) and for graded-involution-simple associative algebras (see \cite[Theorem~1]{BSZ05} or [Proposition 2.49]\cite{livromicha}), and the same proofs work in the superalgebra setting.  

To close this section, we define isomorphism of gradings:

\begin{defi}\label{defi:isomorphic-grds}
    Let $V$ be an (ungraded) object in the category of vector spaces, superspaces, algebras, superalgebras or superalgebras with super-anti-isomorphisms, and let $\Gamma$ and $\Delta$ be $G$-gradings on $V$.
    We say that that $\Gamma$ is \emph{isomorphic} to $\Delta$, denoted by $\Gamma \iso \Delta$ if $(V, \Gamma) \iso (V, \Delta)$ as $G$-graded objects.
\end{defi}

\subsection{Homogeneous maps, elementary gradings and graded modules}\label{sec:grd-modules}

Homomorphisms are not the only maps that interest us.
Some maps that do not preserve degrees are of fundamental importance when dealing with gradings:

\begin{defi}\label{defi:homogeneous-map}
    Let $V$ and $W$ be $G$-graded vector spaces and let $f\from V \to W$ be a linear map.
    We say that $f$ is \emph{homogeneous of degree $g \in G$}, and write $f\in \Hom(V, W)_g$, if 
    \[
    \forall h\in G, \quad f(V_h) \subseteq W_{gh}\,.
    \]
    The subspace $\bigoplus_{g\in G} \Hom(V, W)_g$ of $\Hom(V, W)$ is $G$-graded and will be denoted by $\Hom^{\text{gr}}(V, W)$.
    As before, we will use the word \emph{parity} instead of \emph{degree} in the case of canonical $\ZZ_2$-gradings of superspaces.
\end{defi}

Note that, for non-abelian $G$, \cref{defi:homogeneous-map} depends on the convention of writing functions on the left of their arguments.
The evaluation map $\Hom^{\text{gr}}(V, W) \tensor V \to W$ becomes a (degree-preserving) homomorphism of $G$-graded vector spaces.
If $V$ is finite-dimensional, it is easy to check that $\Hom^{\text{gr}}(V, W) = \Hom(V, W)$.

\phantomsection\label{para:elementary-grading}

In the case $V=W$, the associative algebra $\End^{\text{gr}}(V)$ becomes a $G$-graded algebra. 
When $V$ is finite-dimensional, the $G$-gradings on $\End(V)$ that arise from $G$-gradings on $V$ in this way are called \emph{elementary}.

A basis of $V$ composed of homogeneous elements (with potentially distinct degrees) is called a \emph{graded basis}.  
Given a graded basis $\mathcal{B} = \{v_1, \ldots, v_n\}$ and defining $g_i \coloneqq \deg v_i$ for all $1 \leq i \leq n$, we may identify $\End(V)$ with the matrix algebra $M_n(\FF)$.  
The elementary grading on $M_n(\FF)$ is determined by assigning degree $g_i g_j^{-1}$ to the matrix unit $E_{ij}$ (\ie, the matrix with $1$ in the position $(i,j)$ and $0$ elsewhere) for all $1 \leq i, j \leq n$.  
We say that this is the \emph{elementary grading associated to the tuple $(g_1, \ldots, g_n) \in G^n$}.  
Clearly, for any tuple in $G^n$, we could define a grading on $M_n(\FF)$ via this construction.

\begin{defi}\label{defi:left-grdd-module}
    Let $R = \bigoplus_{g\in G} R_g$ be a $G$-graded associative (super)algebra and let $V = \bigoplus_{g\in G} V_g$ be a $G$-graded vector (super)space.
    We say that $V$ is a \emph{$G$-graded left module $V$} if it is a left $R$-module and
    \[
    \forall g,h\in G,\quad R_g \cdot M_h \subseteq M_{gh}\,.
    \]
    Similarly, we say that $V$ is a \emph{$G$-graded right module $V$} if it is a right $R$-module and
    \[
    \forall g,h\in G,\quad M_h \cdot R_g \subseteq M_{hg}\,.
    \]
    A \emph{homomorphism of $G$-graded left (right) $R$-modules} is a degree-preserving $R$-linear map.
    If $S$ is a $G$-graded (super)algebra and $V$ is a $(R,S)$-bimodule, we say that $V$ is a \emph{$G$-graded bimodule} if it is $G$-graded as both a left $R$-module and a right $S$-module.
\end{defi}

\phantomsection\label{para:Hom-R-gr}

Note that, if $V$ and $W$ are $G$-graded right $R$-modules, then $\Hom_R^{\text{gr}} (V,W) \coloneqq \Hom_R (V,W) \cap \Hom^{\text{gr}} (V,W)$ is a graded subspace of $\Hom^{\text{gr}}(V, W)$.
To make this work for left modules, one should write homomorphisms on the right and define their degree accordingly.

\begin{defi}\label{defi:shift}
    Let $\Gamma\from V = \bigoplus_{h\in G} V_h$ be a grading on a vector space $V$. 
    Given an element $g\in G$, the \emph{right shift of $\Gamma$ by $g$}, denoted $\Gamma^{[g]}$, is obtained by declaring elements of degree $h \in G$ to have degree $hg$, \ie, $\Gamma^{[g]}\from V = \bigoplus_{h\in G} V_h'$ where $V_{h}' \coloneqq V_{h g\inv}$, for all $h\in G$. 
    If $\Gamma$ is fixed, we denote by $V^{[g]}$ the $G$-graded vector space $V$ endowed with $\Gamma^{[g]}$.
    The \emph{left shift of $\Gamma$ by $g$} is defined similarly. 
\end{defi}

Note that, for every $g\in G$, $\End^{\text{gr}}(V) = \End^{\text{gr}}(V^{[g]})$ as $G$-graded algebras.

\phantomsection\label{phsec:definitions-of-M-and-Q-associative}

\begin{defi}\label{defi:M(m,n)}
    Let $m,n\in \ZZ_{\geq 0}$.
    We denote by $\FF^{m|n}$ the superspace whose even and odd components are $\FF^m$ and $\FF^n$, respectively.
    Note that the canonical basis of $\FF^{m|n} \iso \FF^{m+n}$ is graded. 
    We denote by $\M(m,n)$ or $\M_{m|n} (\FF)$ the matrix superalgebra, which is $\M_{m+n}(\FF)$ with the corresponding elementary grading.
    In block notation:
    \begin{align}
    \M(m,n)\even &\coloneqq \left\{ \left(\begin{array}{c|c}
        A & 0\\
        \hline
        0 & D
    \end{array}\right)
    \mid A\in M_{m}(\FF),\, D\in M_n(\FF) \right\},\\
    \M(m,n)\odd &\coloneqq \left\{
    \left(\begin{array}{c|c}
        0 & B\\
        \hline
        C & 0
    \end{array}\right)
    \mid B \in M_{m\times n}(\FF), \, D\in M_{n\times m}(\FF) \right\}.
\end{align}
\end{defi}

\phantomsection\label{elementary-grading-SA}

Given $\gamma_\bz \in G^m$ and $\gamma_\bo \in G^n$, we can endow $M(m,n)$ with an elementary $G$-grading compatible with the structural $\ZZ_2$-grading by taking $\gamma \in G^{m+n}$ to be the concatenation of $\gamma_\bz$ and $\gamma_\bo$ and consider on $M_{m+n}(\FF)$ the elementary grading associated to $\gamma$.
We will refer to this grading as the \emph{elementary grading on $M(m,n)$ associated to $(\gamma_\bz, \gamma_\bo)$}.

We use the opportunity to define an important subsuperalgebra (\ie, a subalgebra that is $\ZZ_2$-graded subspace) of $M(n,n)$:

\begin{defi}\label{defi:associative-Q(n)}
    The \emph{queer associative superalgebra} $Q(n)$ is defined by:
\[
    Q(n) \coloneqq \left\{ \left(\begin{array}{c|c}
        A & B\\
        \hline
        B & A
    \end{array}\right)
    \mid A,B\in M_{n}(\FF)
    \right\}  \subseteq M(n,n).
\]
\end{defi}

\phantomsection\label{phsec:change-side-of-map}

We will follow the convention of writing $R$-linear functions between left (right) $R$-modules on the right (left).
This is especially important in the case $R$ is a superalgebra, because of the rule of signs: if we change the positions of a function and its argument, we may have to change the sign.
We emphasize that this is not an issue for homomorphisms between $R$-modules, since those are parity-preserving by definition. 

\begin{defi}\label{def:change-map-to-the-left}
	Let $R$ be a ($G$-graded) superalgebra, let $\U$ and $\V$ be ($G$-graded) left $R$-modules and let $\psi\from \U \to \V$ be a non-zero homogeneous $R$-linear map (hence written on the right). 
	We define $\psi^\circ\from \U \to  \V$ to be the following map, written on the left:
	\[
		\forall v\in \V\even \cup \V\odd, \quad \psi^\circ(v) = \sign{\psi}{v} (v)\psi\,.
	\]
\end{defi}

It is easy to see that $\psi^\circ$ is not necessarily $R$-linear, but instead
\[\label{lemma:change-of-side-properties}
    \forall r\in R\even \cup R\odd, v\in \V,\quad \psi^\circ (rv) = \sign{\psi}{r} r \psi^\circ (v)\,.
\]
Further, given another homogeneous $R$-linear map $\tau\from \V \to \mc W$, we have 
\[\label{eq:change-of-side-composition}
    (\psi\tau)^\circ = \sign{\psi}{\tau} \tau^\circ\psi^\circ\,.
\]

\subsection{Super-anti-automorphisms and supersymmetric bilinear forms}

We have defined super-anti-automorphisms in \cref{subsec:grds}, here is an example:

\begin{defi}\label{def:supertranspose}
    The \emph{supertranspose} of a matrix in $M(m,n)$ is given by
    \[
        \left(\begin{array}{c|c}
            A & B\\
            \hline
            C & D
        \end{array}\right)\stransp \coloneqq
        \left(\begin{array}{c|c}
            A\transp & -C\transp\\
            \hline
            \rule{0pt}{2.5ex}
            B\transp & \phantom{-}D\transp
        \end{array}\right)\,.
    \]
\end{defi}

Note that it is a super-anti-automorphism of order $4$ if $mn \neq 0$.

\begin{defi}
    Let $U$ be a superspace and let $\langle \cdot, \cdot \rangle \from U\times U \to \FF$ be a bilinear form which is homogeneous if seen as a map $U\tensor U \to \FF$.
    We say that $\langle \cdot, \cdot \rangle$ is \emph{supersymmetric} if
    \[
        \forall u,v \in U\even \cup U\odd, \quad \langle u, v \rangle 
        = \sign{u}{v} \langle v, u \rangle\,,
    \]
    and \emph{super-skew-symmetric} if
    \[
        \forall u,v \in U\even \cup U\odd, \quad \langle u, v \rangle 
        = -\sign{u}{v} \langle v, u \rangle\,.
    \]
    If $U$ is finite-dimensional, then any nondegenerate $\langle \cdot, \cdot \rangle$ defines a unique linear map $\vphi\from \End(U) \to \End(U)$ that sends any $T\in \End(U)\even \cup \End(U)\odd$ to $\vphi(T)$ defined by
    \[
        \forall u,v \in U\even \cup U\odd,\,
        \langle T(u), v \rangle = \sign{T}{u} \langle u, \vphi(T)(v) \rangle\,. 
    \]
    We will call this map $\vphi$ the \emph{superadjunction} with respect to $\langle \cdot, \cdot \rangle$.
\end{defi}

\phantomsection\label{phsec:parity-reversed}

The superadjunction is always a super-anti-automorphism, and it is a superinvolution \IFF $\langle \cdot, \cdot \rangle$ is supersymmetric or super-skew-symmetric. 
Since a super-skew-symmetric bilinear form on $U$ is the same as a supersymmetric bilinear form on the parity-reversed superspace $U^{[\barr 1]}$ (\cref{defi:shift}), we can restrict our attention to the supersymmetric ones.

\begin{defi}\label{def:superopposite}
    Let $R$ be a superalgebra. 
    We define the \emph{superopposite superalgebra $R\sop$} to be $R$ as a superspace but with a different product. 
    When an element $r\in R$ is regarded as an element of $R\sop$, we will denote it by $\bar r$. 
    The product on $R\sop$ is defined by:
    \[
    \forall r, s \in R\even \cup R\odd,\quad \bar r \, \bar {s} = \sign{r}{s} \, \overline{sr}\,. 
    \]
\end{defi}

Note that if $G$ is abelian and $R$ is $G$-graded, if we consider $R\sop$ as the same $G$-graded superspace, then $R\sop$ is also a $G$-graded superalgebra.
In particular, the construction below is a source of $G$-graded superalgebras with superinvolution:

\begin{defi}\label{defi:exchange-sinv}
    Let $S$ be a superalgebra and set $R \coloneqq S \times S\sop$.
    We define the \emph{exchange superinvolution} on $R$ to be the map $\vphi\from R \to R$ given by
    \[
        \forall s_1,s_2\in S, \quad \varphi (s_1, \bar s_2) = (s_2, \bar s_1)\,.
    \]
    Unless stated otherwise, we will always consider $R = S \times S\sop$ to be endowed with exchange superinvolution.
\end{defi}

It is straightforward to verify that $S \times S\sop$ is graded-superinvolution-simple \IFF $S$ is graded-simple. 

\begin{ex}\label{ex:FxF-iso-FZ2}
	The simplest possible example is to take $S = \FF$. 
	If $\Char \FF \neq 2$, then $S\times S\sop = \FF [\zeta]$ where $\zeta = (1, -1)$ and the exchange superinvolution is given by $\vphi(1) = 1$ and $\vphi(\zeta) = -\zeta$.
	Note that $\FF [\zeta] \iso \FF\ZZ_2$ with the trivial superalgebra structure. 
\end{ex}

\begin{ex}\label{ex:FZ2xFZ2sop-iso-FZ4}
	Consider $S = Q(1)$, so $S\even = \FF 1$ and $S\odd = \FF u$ where $u^2 =1$.
	Note that $S$ is isomorphic to $\FF\ZZ_2$, but this time with the superalgebra structure given by its natural $\ZZ_2$-grading. 
	If $\Char \FF \neq 2$, we claim that $R \coloneqq S\times S\sop$ is isomorphic to $\FF\ZZ_4$.
	Indeed, the element $\omega \coloneqq (u, \bar u) \in S\times S\sop$ has order $4$ and generates $S\times S\sop$: $\omega^2 = (1, - 1)$, $\omega^3 = (u, - \bar u)$ and $\omega^4 = (1, 1)$.
	Hence $R\even = \FF1 \oplus \FF \omega^2$ and $R\odd = \FF \omega \oplus \FF \omega^3$.
	Also, the exchange superinvolution on $R$ is given by $\vphi(1) = 1$, $\vphi(\omega) = \omega$, $\vphi(\omega^2) = -\omega^2$ and $\vphi(\omega^3) = -\omega^3$.
\end{ex}

The following is the graded version of a result in \cite{racine}:

\begin{prop}\label{prop:only-SxSsop-is-simple}
	Let $(R, \vphi)$ be a graded superalgebra with superinvolution. 
	Then $(R, \vphi)$ is graded-superinvolution-simple \IFF either $R$ is a graded-simple or $(R, \vphi)$ is isomorphic to $S\times S\sop$ with the exchange superinvolution, for some graded-simple superalgebra $S$.
\end{prop}

\begin{proof}
	Suppose $(R, \vphi)$ is 
	graded-superinvolution-simple but $R$ is not graded-simple. 
	Let $0 \neq I \subsetneq R$ be a graded superideal.
	Note that $\vphi(I)$ is also a graded superideal, hence $I \cap \vphi(I)$ and $I + \vphi (I)$ are $\vphi$-invariant graded superideals. 
	Since $I \cap \vphi(I) \subseteq I \neq R$, we have $I \cap \vphi(I) = 0$, so we can write $I + \vphi (I) = I \oplus \vphi (I)$. 
	Since $0 \neq I \subseteq I \oplus \vphi (I)$, we conclude that $R = I \oplus \vphi (I)$.
	Clearly, this implies that $(R, \vphi)$ is isomorphic to $I \times I\sop$ with exchange superinvolution.
\end{proof}

\subsection{Center and Supercenter}

In the context of superalgebras, we have not only the usual concept of center but also the concept of supercenter:

\begin{defi}\label{defi:center}
	Let $R$ be an associative superalgebra.
	The \emph{center} of $R$ is the set
	\[
		Z(R) = \{c\in R \mid cr = rc \text{ for all } r\in R \},
	\]
	\ie, the center of $R$ seen as an algebra, and the \emph{supercenter} of $R$ is the set $sZ(R) \coloneqq sZ(R)\even \oplus sZ(R)\odd$, where
	\begin{align*}
		sZ(R)^i = \{c\in R^i \mid cr = (-1)^{i |r|} rc \text{ for all } r\in R\even\cup R\odd \},\quad i \in \ZZ_2.
	\end{align*}
\end{defi}

It is well known that $Z(M_n (\FF)) = \FF 1$, and one can check that $sZ(M_n (\FF)) = \FF 1$.
For the associative superalgebra $Q(n) = M_n (\FF) \oplus u \, M_n(\FF)$, we have that $Z(Q(n))$ is the subspace spanned by $1$ and $u$ (so $Z(Q(n)) \iso Q(1)$), while $sZ(Q(n)) = \FF 1$.

\begin{lemma}\label{lemma:center-is-graded}
	Let $G$ be an abelian group and let $R$ be an associative $G$-graded superalgebra.
	Then the center $Z(R)$ and the supercenter $sZ(R)$ are $G$-graded subsuperalgebras of $R$.
\end{lemma}

\begin{proof}
	We consider $R$ as a $G^\#$-graded algebra.
	Let $c \in Z(R)$ and write $c = \sum_{g \in G^\#} c_g$, where $c_g \in R_g$ for all $g \in G^\#$.
	For every homogeneous $r \in R$, we have
	\begin{align*}
		\big(\sum_{g\in G^\#} c_g\big)r = r \big(\sum_{g\in G^\#} c_g\big).
	\end{align*}
	Comparing the components of degree $gh = hg$, where $h = \deg r$, we conclude that $rc_g = c_g r$ for all $g \in G^\#$.
	By linearity, $r c_g = c_g r$ for all $r\in R$, hence $c_g \in Z(R)$.

	The same argument works to show that $sZ(R)\even$ is graded and, with straightforward modifications, to show that $sZ(R)\odd$ is graded.
\end{proof}

\begin{lemma}
	Let $(R, \vphi)$ be a superalgebra with super-anti-automorphism.
	Then $Z(R)$ and $sZ(R)$ are $\vphi$-invariant.
\end{lemma}

\begin{proof}
	We will prove only that $Z(R)$ is $\vphi$-invariant, since the proof for $sZ(R)$ is similar.
    By \cref{lemma:center-is-graded}, we have $Z(R) = Z(R)\even \oplus Z(R)\odd$, so it is sufficient to show that if $c \in Z(R)\even \cup Z(R)\odd$, then $\vphi(c) \in Z(R)$. 
	Let $r \in R\even \cup R\odd$.
	Since $c\vphi\inv (r) = \vphi\inv (r)c$, we can apply $\vphi$ on both sides and get $\sign{c}{r} r \vphi(c) = \sign{c}{r} \vphi(c) r$ and, hence, $r \vphi(c) = \vphi(c) r$.
\end{proof}

The following result is well known (see, \eg, \cite[Theorem 8.1]{Sh98} and \cite[Theorem 28]{MR2407903}):

\begin{cor}\label{cor:Q-no-sinv-center}
	If $\Char \FF \neq 2$, the associative superalgebra $Q(n)$ does not admit a superinvolution.
\end{cor}

\begin{proof}
	The center of $Q(n)$ is isomorphic to $\FF1 \oplus \FF u$, where $u$ is an odd element with $u^2 = 1$.
	Let $\vphi$ be a super-anti-automorphism on $Q(n)$.
	Since $u$ is odd and central, $\vphi(u)$ is odd and central.
	Hence there is $\lambda \in \FF$ such that $\vphi(u) = \lambda u$.
	Using that $u^2 = 1$, we have $1 = \vphi(1) = \vphi(u^2) = - \vphi(u)^2 = - \lambda^2$.
	But then $\vphi^2 (u) = \lambda^2 u = -u \neq u$, hence $\vphi^2 \neq \id$.
\end{proof}

\subsection{Simple superalgebras}\label{subsec:simple-superalgebras}

We now present the classification of finite-dimensional simple superalgebras in the associative, associative-with-superinvolution, and Lie cases.  
Throughout this subsection, we assume that $\FF$ is algebraically closed and that $\Char \FF \neq 2$.  
For the Lie case, we will further assume $\Char \FF = 0$ and limit ourselves to the non-exceptional classical Lie superalgebras.

\subsubsection{Simple associative superalgebras}

The classification up to isomorphism of simple finite-dimensional associative superalgebras is well known (and also follows from the results in \cref{subsec:artin-wedderburn} by considering trivial $G$-gradings).  
Given $m,n \in \ZZ_{\geq 0}$, the associative superalgebras $M(m,n)$, with $m$ and $n$ not both zero, and $Q(n)$, with $n \neq 0$, are simple.  
Conversely, any finite-dimensional simple associative superalgebra $R$ is isomorphic to either $M(m,n)$ or $Q(n)$.  
These cases are mutually exclusive, and we say $R$ is of \emph{type $M$} or \emph{type $Q$}, respectively.  
Moreover:
\begin{itemize}
    \item $M(m,n) \iso M(m',n')$ \IFF $(m,n) = (m',n')$ or $(m,n) = (n',m')$;
    \item $Q(n) \iso Q(n')$ \IFF $n = n'$.
\end{itemize}

\subsubsection{Superinvolution-simple associative superalgebras}

The classification of finite-dimensional superinvolution-simple associative superalgebras was obtained in \cite{racine} (and also follows from \cref{thm:vphi-iff-vphi0-and-B,thm:vphi-involution-iff-delta-pm-1,thm:iso-(R-vphi)-with-parameters} with trivial $G$-grading).

If $(R, \vphi)$ is superinvolution-simple but $R$ is not simple as a superalgebra, then it must be isomorphic to $M(m,n) \times M(m,n)\sop$, or $Q(n)\times Q(n)\sop$, which are mutually exclusive cases.
If $R$ is simple as a superalgebra, then $R$ must be of type $M$ and $(R, \vphi) \iso (\End(\FF^{m|n}), \vphi')$, where $\vphi'$ is the superadjunction with respect to a nondegenerate homogeneous supersymmetric bilinear form $\langle \cdot, \cdot \rangle$ on $\FF^{m|n}$.
Moreover, the isomorphism class of $(R, \vphi)$ depends only on $m$, $n$, and the parity of $\langle \cdot, \cdot \rangle$.
This motivates the following:

\begin{defi}\label{defi:M(m-n-p_0)}
    Let $m,n \in \ZZ_{\geq 0}$ (not both zero) and let $\langle \cdot, \cdot \rangle$ be a nondegenerate homogeneous supersymmetric bilinear form on $\FF^{m|n}$ with parity $p_0 \in \ZZ_2$. 
    We define $M^*(m,n,p_0)$ as the superalgebra $M(m,n) \iso \End(\FF^{m|n})$ endowed with the superinvolution given by adjunction with respect to $\langle \cdot, \cdot \rangle$.
    In matrix terms, the superinvolution $\vphi$ is given by:
    \[
        \forall X \in M_{m|n}(\FF), \quad \vphi(X) \coloneqq \Phi^{-1} X^{\stransp} \Phi,
    \]
    where $\Phi$ is the matrix representation of $\langle \cdot, \cdot \rangle$ in the standard basis.
\end{defi}

We note that an even nondegenerate homogeneous supersymmetric bilinear form on $\FF^{m|n}$ exists \IFF $n$ is an even number.
In this case, $\Phi$ can be chosen to be:
\[
\Phi = \left(
    \begin{array}{c|c}
        I_m & 0 \\
        \hline
        0 & \begin{array}{cc}
            0 & I_{\frac{n}{2}} \\
            - I_{\frac{n}{2}} & 0
        \end{array}
    \end{array}
    \right)\,.
\]
An odd nondegenerate homogeneous supersymmetric bilinear form on $\FF^{m|n}$ exists \IFF $m=n$.
In this case, $\Phi$ can be chosen to be:
\[\label{eq:matrix-for-P}
\Phi = \left(
    \begin{array}{c|c}
        0 & I_n \\
        \hline
        I_n & 0
    \end{array}
    \right)\,.
\]

\begin{defi}\label{defi:types-sinv-simple}
    Let $(R, \vphi)$ be a finite-dimensional superinvolution-simple superalgebra. 
    \begin{enumerate}
        \item If $R$ is of type $M$, we say that $(R, \vphi)$ is of type $M$;
        \item If $(R, \vphi) \iso M(m,n) \times M(m,n)\sop$, for some $m,n \geq 0$, we say that $(R, \vphi)$ is of type $M\times M\sop$;
        \item If $R \iso Q(n) \times Q(n)\sop$, for some $n \geq 0$, we say that $(R, \vphi)$ is of type $Q \times Q\sop$.
    \end{enumerate}
\end{defi}

We can distinguish the types using the center:

\begin{lemma}\label{prop:types-of-SA-via-center}
	Let $(R, \vphi)$ be a superalgebra with superinvolution.
	\begin{enumerate}
		\item If $(R, \vphi)$ is of type $M$, then $(Z(R), \vphi) \iso (\FF, \id)$;\label{item:F-id}
		\item If $(R, \vphi)$ is of type $M\times M\sop$, then $(Z(R), \vphi)$ is isomorphic to the superalgebra with superinvolution in \cref{ex:FxF-iso-FZ2};\label{item:FZ2-exchg}
		\item If $(R, \vphi)$ is of type $Q\times Q\sop$, then $(Z(R), \vphi)$ is isomorphic to the superalgebra with superinvolution in \cref{ex:FZ2xFZ2sop-iso-FZ4}.\label{item:FZ4-exchg}
	\end{enumerate}
\end{lemma}

\begin{proof}
    \Cref{item:F-id} follows from the fact that $Z(M_{m+n}(\FF)) \iso \FF$.  
    To verify \cref{item:FZ2-exchg,item:FZ4-exchg}, note that $Z(S \times S\sop) = Z(S) \times Z(S)\sop$ for any superalgebra $S$.  
    Since $Z(M_{m+n}(\FF)) \iso \FF$ and $Z(Q(n)) \iso Q(1)$, the results follow directly.
\end{proof}

Finally, the isomorphism conditions are the following:
\begin{enumerate}
    \item $M^*(m,n, p_0) \iso M^*(m',n', p_0')$ \IFF{} $(m,n) = (m',n')$ and $p_0 = p_0'$;  
    \item $M(m,n) \times M(m,n)\sop \iso M(m',n') \times M(m',n')\sop$ \IFF{} $(m,n) = (m',n')$ or $(m,n) = (n',m')$;  
    \item $Q(n)\times Q(n)\sop \iso Q(n')\times Q(n')\sop$ \IFF{} $n = n'$.  
\end{enumerate}

\subsubsection{Simple Lie superalgebras}\label{subsec:def-Lie-superalgebras}

We now assume that $\Char \FF = 0$. 
In this case, the classification of finite-dimensional simple Lie algebras was established in \cite{artigokac}.  
A simple Lie superalgebra $L = L\even \oplus L\odd$ is said to be \emph{classical} if $L\odd$ is a semisimple $L\even$-module and of \emph{Cartan type} otherwise.  
The classical Lie superalgebras are divided into the series $A(m,n)$, $B(m,n)$, $C(n)$, $D(m,n)$, $P(n)$, and $Q(n)$, as well as three exceptional cases: $F(4)$, $G(3)$, and the family $D(2,1,\alpha)$ with $\alpha \in \FF \setminus \{0, -1\}$. 
In this work, we will define only the non-exceptional classical Lie superalgebras, as these are the ones to which our method of classifying gradings applies.
For the definitions of exceptional classical Lie superalgebras and Cartan-type series, we refer to \cite{artigokac,livrosuperalgebra}.

\phantomsection\label{phsec:series-A}

Let $m,n \in \ZZ_{\geq 0}$, not both zero. 
The \emph{general linear Lie superalgebra} $\gl(m|n)$ is $M(m,n)^{(-)}$.
The \emph{supertrace} of a matrix in $M(m,n)$ is defined by
\[
    \operatorname{str} 
    \left(\begin{array}{c|c}
        A & B\\
        \hline
        C & D
    \end{array}\right)
    \coloneqq \tr A - \tr D.
\] 
The \emph{special linear Lie superalgebra} $\Sl (m|n)$ is the derived superalgebra of $\gl(m|n)$ or, equivalently,
\[
    \Sl(m|n) \coloneqq \left\{
	T \in \gl(m|n)
	\mid \str T = 0
	\right\}.
\]
If $m\neq n$ then $\Sl(m|n)$ is a simple Lie superalgebra. 
However, if $m=n$, then
\[
    Z(\Sl(m|n)) = \FF 1 = \left\{
    \left(\begin{array}{c|c}
        \lambda I & 0\\
        \hline
        0 & \lambda I
    \end{array}\right)
	\mid \lambda \in \FF
	\right\}
\]
is a nontrivial superideal. 
The quotient $\Sl(n|n)/ \FF 1$ is a simple superalgebra \IFF $n > 1$.

In the definition below, we impose $m\geq n$ to avoid repetition:

\begin{defi}\label{defi:A(m-n)}
    For $m \geq n$, the Lie superalgebra $A(m,n)$ is defined to be $\Sl(m+1 \,|\, n+1)$ if $m\neq n$, and $\mathfrak{psl}(n+1 \,|\, n+1) \coloneqq \Sl(n+1 \,|\, n+1)/ \FF 1$ if $m=n$. 
\end{defi} 

\phantomsection\label{phsec:series-B}

The \emph{orthosymplectic Lie superalgebra} $\osp(m|n)$ is $\Skew(M^*(m,n, \bz))$. 
If $m,n > 0$, then $\osp(m|n)$ is a simple Lie superalgebra. 

\begin{defi}\label{defi:B-C-D}
    The orthosymplectic Lie superalgebras are divided into three series:
    \begin{itemize}
    	\item $B(m,n) \coloneqq \osp(2m+1 \,|\, 2n)$, for $m \geq 0$ and $n \geq 1$;
    	\item $C(n) \coloneqq \osp(2 \,|\,  2n - 2)$, for $n\geq 2$;
    	\item $D(m,n) \coloneqq \osp(2m \,|\, 2n)$, for $m\geq 2$ and $n\geq 1$.
    \end{itemize}
\end{defi}

Since $C(2)\iso A(1,0)$, it is sometimes imposed $n\geq 3$ in the $C(n)$ case to avoid repetition.

\phantomsection\label{phsec:series-P}

The \emph{periplectic Lie superalgebra} $\mathfrak{p}(n)$ is defined to be $\Skew(M^*(n,n, \bo))$.
Unlike the orthosymplectic case, $\mathfrak{p}(n)$ is not simple.

\begin{defi}\label{defi:P(n)}
    The superalgebra $P(n)$ is the derived superalgebra of $\mathfrak{p}(n+1)$. 
\end{defi}

$P(n)$ is simple \IFF $n\geq 2$. 

\phantomsection\label{phsec:series-Q-Lie}

Finally, let $R$ denote the associative superalgebra $Q(n+1)$ and consider the Lie superalgebra $R^{(-)}$.
Its derived superalgebra is
\[
    R^{(1)} = \left\{
    \left(\begin{array}{c|c}
        A & B\\
        \hline
        B & A
    \end{array}\right) \in R^{(-)} \mid \tr A = \tr B = 0\right\}\,,
\]
which is not simple since $Z(R^{(1)}) = \FF1$.

\begin{defi}
    The \emph{queer Lie superalgebra}, also denoted by $Q(n)$, is $R^{(1)}/\FF1$.
\end{defi}

The Lie superalgebra $Q(n)$ is simple \IFF $n\geq 2$.
\section{Graded-Simple Associative Superalgebras}\label{sec:grdd-simple-ass}

In this section, we develop the theory of graded-simple associative superalgebras satisfying the descending chain condition on graded left superideals.
In \cref{subsec:artin-wedderburn}, we show, following \cite{livromicha}, that any such superalgebra is of the form $\End_\D(\U)$, where $\D$ is a graded-division superalgebra and $\U$ is a graded right $\D$-module of finite rank.
Then, in \cref{subsection:graded-modules-over-D}, we describe the structure of $\U$ and $\End_\D(\U)$.
The last two subsections focus on finite-dimensional graded-simple superalgebras over an algebraically closed field $\FF$ with $\Char \FF \neq 2$: \cref{subsec:D-alg-closed} describes graded-division superalgebras, while \cref{subsec:grdd-simple-ass-algebraically-closed} classifies graded-simple superalgebras and, in particular, $G$-gradings on simple finite-dimensional associative superalgebras.

\subsection{Graded Wedderburn-Artin Theory}\label{subsec:artin-wedderburn}

A cornerstone in the theory of gradings is the analog of the Wedderburn–Artin Theorem for graded‐simple algebras. Special cases of this result have appeared in several works (\eg, \cite{BSZ01,MR2046303,BZ02}).
Here we state the theorem in the context of superalgebras, which follows directly as a corollary of \cite[Theorem 2.6]{livromicha} (see \cite[page 31]{livromicha} for the converse) by regarding a $G$-graded superalgebra as a $G^\#$-graded algebra.

\begin{defi}\label{def:graded-division-algebra}
    A \emph{$G$-graded-division superalgebra} is a unital $G$-graded superalgebra $\D$ in which every nonzero $G^\#$-homogeneous element (\ie, an element homogeneous with respect to both the $G$-grading and the canonical $\ZZ_2$-grading) is invertible.
    In this context, we refer to the $G$-grading on $\D$ as a \emph{division grading}.
\end{defi}

Division superalgebras are the special case corresponding to trivial $G$ (see \cite{racine,MR1746565}), whereas graded-division algebras are the special case corresponding to trivial $\ZZ_2$-grading, \ie, $\D = \D\even$.

\begin{defi}\label{defi:even-odd-D}
    If $\D = \D\even$, we say that $\D$ is an \emph{even} graded-division superalgebra, and if $\D \neq \D\even$, we say that $\D$ is an \emph{odd} graded-division superalgebra. 
\end{defi}

\begin{remark}\label{lemma:odd-M-m=n}
    Note that if $\D$ is odd, then for any nonzero $d_1 \in \D\odd$, we have $\D\odd = d_1 \D\even$, which implies that $\dim_\FF \D\even = \dim_\FF \D\odd$.
    In particular, if $\D \iso M(m,n)$ as a superalgebra, then $\D$ can be odd only when $m = n$ (since $\dim \D\even = m^2 + n^2$ and $\dim \D\odd = 2mn$).
\end{remark}

A standard argument with Zorn's Lemma shows that any graded $\D$-module $\U$ has a basis consisting of $G^\#$-homogeneous elements.
Moreover, if the basis is finite, then $\End_\D^\text{gr} (\U) = \End_\D (\U)$ (see \cref{para:Hom-R-gr}).

\begin{thm}\label{thm:End-over-D}
    Let $G$ be a group and let $R$ be a $G$-graded associative superalgebra satisfying the descending chain condition on graded left superideals. 
	Then $R$ is graded-simple \IFF there exists a $G$-graded division superalgebra $\D$ and a nonzero graded right $\D$-module $\mc{U}$ of finite rank such that $R \iso \End_{\D} (\mc{U})$ as graded superalgebras. \qed
\end{thm}

\begin{remark}\label{lemma:converse-density-thm}
    The right $\D$-module $\U$ is also a left $R$-module and a $(R, \D)$-bimodule.
    As a left $R$-module, $\U$ is graded-simple, and the representation $\rho\from \D \to \End_R(\U)$ is an isomorphism  (see \cite[Exercise 3 on page 60]{livromicha}). 
\end{remark}

An important observation is that the centers of $\D$ and $R$ can be identified.
This identification will later be used to determine the type of $\D$ as a superalgebra (with or without superinvolution).
The ungraded version of the following result is well known, and the same proof applies to the graded case:

\begin{lemma}\label{prop:R-and-D-have-the-same-center}
    Suppose $G$ is abelian. 
    Then the map $\iota\from Z(\D) \to Z(R)$ given by $\iota (d)(u) \coloneqq ud$, for all $d\in Z(\D)$ and $u\in \U$, is an isomorphism of graded algebras. \qed
\end{lemma}

    

We finish this subsection by presenting a necessary and sufficient condition for $G$-graded superalgebras $\End_\D(\U)$ and $\End_{\D'}(\U')$ to be isomorphic.
To this end, we introduce the following definitions:

\begin{defi}\label{def:twist}
    Let $\psi_0\from \D \to \D'$ be a homomorphism of $G$-graded superalgebras, and let $\U$ be is graded right $\D'$-module.
    We can equip $\U$ with the structure of a graded right $\D$-module by defining $u\cdot d \coloneqq u\,\psi_0 (d)$, for all $u\in \U$ and $d\in \D$.
    We denote the resulting graded superspace by $\U^{\psi_0}$.
    In the case $\psi_0\from \D \to \D$ is an automorphism, $\U^{\psi_0}$ is called the \emph{twist of $\U$ by $\psi_0$}.
\end{defi}

\begin{defi}\label{def:inner-automorphism}
	Let $d\in \D$ be a nonzero homogeneous element.
	We define:
    \begin{enumerate}
        \item The \emph{inner automorphism} $\operatorname{Int}_d\from \D \to \D$ by $\operatorname{Int}_d (c) \coloneqq dcd\inv$, for all $c\in \D$;
        \item The \emph{superinner automorphism} $\operatorname{sInt}_d\from \D \to \D$ to be the linear map determined by $\operatorname{sInt}_d (c) \coloneqq \sign{d}{c}dcd\inv$, for all $c\in \D\even\cup\D\odd$.
    \end{enumerate}
\end{defi}

The next result is \cite[Theorem 2.10]{livromicha} with a slightly different notation:

\begin{thm}\label{thm:iso-abstract}
	Let $R \coloneqq \End_\D(\U)$ and $R' \coloneqq \End_{\D'}(\U')$, where $\D$ and $\D'$ are graded-division superalgebras, and $\U$ and $\U'$ are nonzero right graded modules of finite rank over $\D$ and $\D'$, respectively.
	Given an isomorphism $\psi\from R \to R'$, there is a triple $(g, \psi_0, \psi_1)$, where $g \in G^\#$, $\psi_0\from {}^{[g\inv]}\D^{[g]} \to \D'$ is an isomorphism of graded superalgebras, $\psi_1\from \U^{[g]} \to (\U')^{\psi_0}$ is an isomorphism of graded right $\D$-modules, such that
	\begin{equation}\label{eq:def-iso-algebras}
		\forall r\in R, \quad \psi(r) = \psi_1 \circ r \circ \psi_1\inv.
	\end{equation}
	Conversely, given a triple $(g, \psi_0, \psi_1)$ as above, Equation \eqref{eq:def-iso-algebras} defines an isomorphism of graded superalgebras $\psi\from R \to R'$.
	Another triple $(g', \psi_0', \psi_1')$ defines the same isomorphism $\psi$ \IFF there are $t\in \supp \D' \subseteq G^\#$ and $0 \neq d\in \D'_t$ such that $g'= gt$, $\psi_0' = \mathrm{Int}_{d\inv} \circ \psi_0$ and $\psi_1' (u) = \psi_1 (u) d$ for all $u \in \U$. \qed
\end{thm}

\subsection{Graded-division superalgebras and their graded modules}\label{subsection:graded-modules-over-D}

We start with examples of division gradings on the superalgebras $Q(1)$ and $M(1,1)$.

\begin{ex}\label{ex:Q(1)-as-grd-div-SA}
    Let $G = \langle h \rangle$ be a cyclic group of order $1$ or $2$.
    We can grade $Q(1) = \FF \oplus \FF u$ by declaring the $G$-degree of $u$ to be $h$ and this gives us a division grading. 
    In this case $G^\# = \langle h \rangle \times \ZZ_2$ and the support of the grading is the subgroup $\langle (h, \bar 1) \rangle \iso \ZZ_2$. 
\end{ex}

\begin{ex}\label{ex:Pauli-2x2-super}
    Assume that $\Char \FF \neq 2$ and let $G = \ZZ_2$ (so $G^\# = \ZZ_2 \times \ZZ_2$).
    Then we have a division-grading on $M(1,1)$, sometimes called \emph{Pauli grading}, with $G^\#$-degrees given by:
    \begin{align*}
	\deg \begin{pmatrix}
		\phantom{.}1 & \phantom{-}0\phantom{.} \\
		\phantom{.}0 & \phantom{-}1\phantom{.}
	\end{pmatrix} = (\bar 0, \bar 0),\quad & \deg \begin{pmatrix}
		\phantom{.}0 & \phantom{-}1\phantom{.} \\
		\phantom{.}1 & \phantom{-}0\phantom{.}
	\end{pmatrix} = (\bar 0, \bar 1), \\
	\deg \begin{pmatrix}
		\phantom{.}1 & \phantom{-}0\phantom{.} \\
		\phantom{.}0 & -1\phantom{.}
	\end{pmatrix} = (\bar 1, \bar 0),\quad &
	\deg \begin{pmatrix}
		\phantom{.}0 & -1\phantom{.}           \\
		\phantom{.}1 & \phantom{-}0\phantom{.}
	\end{pmatrix} = (\bar 1, \bar 1).
\end{align*}
\end{ex}

For any graded-division superalgebra $\D$, one readily verifies that $T \coloneqq \supp \D$ is a subgroup of $G^\#$, so $\D$ can be reguarded as a $T$-graded algebra. 
We will use subscripts to refer to this $T$-grading, \ie, $\D_t$ refers to the homogeneous component $\D_g^i$ where $t = (g, i) \in T \subseteq G^\#$. 

The parity is encoded by the group homomorphism $p\from T \subseteq G^\# \to \ZZ_2$ defined by $p (g, i) = i$ for all $(g,i)\in G^\#$.
We will denote the kernel of $p$ by $T^+$, \ie, $T^+ \coloneqq T\cap (G\times \{ \bar 0 \}) = \supp \D\even$. 
Similarly, we define $T^- \coloneqq T\cap (G\times \{ \bar 1 \}) =  \supp \D\odd$. 

\phantomsection\label{para:D-modules}

Now, let $\U = \bigoplus_{g\in G^\#} \U_g$ be an arbitrary graded right $\D$-module.
Observe that the $\D$-submodule generated by a homogeneous component $\U_g$ is $\U_{gT} \coloneqq \bigoplus_{t\in T} \U_{gt}$, so it is convenient to introduce the following:

\begin{defi}
    Given a left coset $x\in G/T$, the corresponding \emph{isotypic component} of $\U$ is defined by
    $
        \U_{x} \coloneqq \bigoplus_{g\in x} \U_{g}.
    $
\end{defi}

\begin{remark}
    The use of the term ``isotypic component'' here is consistent with its common use. 
    It is easy to see that all simple $\D$-modules are of the form ${}^{[g]}\D$ for some $g\in G^\#$, and ${}^{[g]}\D \iso {}^{[g']}\D$ \IFF $g\inv g' \in T$. 
\end{remark}

Clearly, two graded right $\D$-modules $\U$ and $\V$ are isomorphic \IFF the isotypic component $\U_x$ is isomorphic to $\V_x$ for all $x \in G/T$.

\begin{defi}
    A $\D$-basis $\mc B$ of $\U$ is called a \emph{graded basis of $\U$} if all the elements in $\mc B$ are $G^\#$-homogeneous (of various degrees).
\end{defi}

Since $\D$ is a graded-division superalgebra, $\D_e$ is a division algebra.
For all $g\in G^\#$, the graded subspace $\U_g$ is a $\D_e$-submodule, and it is straightforward to check that a $\D_e$-basis of $\U_g$ is a graded $\D$-basis of the isotypic component $\U_{gT}$.
As a consequence, we have:

\begin{lemma}\label{cor:iso-isotypic-components}
\begin{enumerate}
    \item Every graded right $\D$-module $\U$ has a graded basis, and all graded bases of $\U$ have the same cardinality (so $\dim_\D \U$ is well defined).
    \item Let $\U$ and $\V$ be graded right $\D$-modules.
    Then, for any $g\in G^\#$, the graded $\D$ modules $\U_{gT}$ and $\V_{gT}$ are isomorphic \IFF $\dim_{\D_e} \U_g = \dim_{\D_e} \V_g$. \qed
\end{enumerate}
\end{lemma}

It follows that the isomorphism class of any graded $\D$-module $\U$ with finite dimension over $\D$ is determined by the map $\kappa \from G^\#/T \to \ZZ_{\geq 0}$ defined by $\kappa(x) \coloneqq \dim_\D (\U_x)$, which has finite support.
In other words, $\kappa$ is a finite \emph{multiset} in $G^\#/T$, where $\kappa(x)$ is viewed as the multiplicity of the point $x$. 
As usually done for multisets, we define $|\kappa| \coloneqq \sum_{x \in \supp \kappa} \kappa(x)$. 
By definition, $|\kappa| = \dim_\D (\U)$.

\begin{defi}\label{defi:gamma-realizes-kappa}
    Given a map $\kappa\from G^\#/T \to \ZZ_{\geq 0}$ with finite support, we say that a $k$-tuple $\gamma = (g_1, \ldots, g_k) \in (G^\#)^k$, where $k \coloneqq |\kappa|$, \emph{realizes} $\kappa$ if the number of entries $g_i$ with $g_i\in x$ is equal to $\kappa (x)$ for every $x\in G/T$.
\end{defi}

It is clear that, choosing any $k$-tuple $\gamma = (g_1, \ldots, g_k)$ that realizes $\kappa$, we can construct the graded $\D$-module $\U \coloneqq {}^{[g_1]}\D \oplus \cdots \oplus {}^{[g_k]}\D$, which is a representative of the isomorphism class of graded $\D$-modules determined by $\kappa$.

\phantomsection\label{para:End_D(U)-is-M-tensor-D}

Fixing a graded basis $\B$ of a graded right $\D$-module $\U$ with $k \coloneqq \dim_\D \U < \infty$, we can write $\End_\D(\U)$ as $M_k(\D)$.
Given a nonzero homogeneous element $d \in \D$ and $1\leq i,j \leq k$, let $E_{ij}(d)$ be the matrix with $d$ in position $(i,j)$ and $0$ everywhere else.
Writing $g_1, \ldots, g_k$ for the degrees of the elements in $\B$, we have
\[
    \deg E_{ij}(d) = g_i (\deg d) g_j\inv\,.
\]
Assuming $G$ abelian, it follows that $M_k(\D) \iso M_k(\FF) \tensor \D$ as $G$-graded superalgebras, where $M_k(\FF)$ is equipped with the elementary $G^\#$-grading associated to $(g_1, \ldots, g_k)$ (see \cref{para:elementary-grading}).

\phantomsection\label{phsec:simple-R-D-super}

A standard result states that the ideals of $M_k(\D)$ are exactly the sets $M_k(I)$, where $I$ is an ideal of $\D$.  
It is straightforward to check the same is true for superideals.

\begin{prop}\label{prop:simple-R-D-super}
    The superideals of $M_k(\D)$ are precisely the sets $M_k(I)$, where $I$ is a superideal of $\D$.  
    In particular, $\End_\D(\U)$ is simple as a superalgebra if and only if $\D$ is simple as a superalgebra. \qed
\end{prop}




\phantomsection\label{para:module-over-even-or-odd-D}

The distinction between even and odd graded-division superalgebras influences the parameters we can use to describe a graded $\D$-module $\U$ and, consequently, the graded superalgebra $\End_\D (\U)$.

If $\D$ is an even graded-division superalgebra, then both $\U\even$ and $\U\odd$ are graded $\D$-submodules, and we can write 
\begin{align}
\End_\D(\U)\even &=
\begin{pmatrix}
    \End_\D(\U\even) & 0\\
    0 & \End_\D(\U\odd)
\end{pmatrix} \text{ and}\\
\End_\D(\U)\odd &=
\begin{pmatrix}
    0                       & \Hom_\D(\U\odd,\U\even)\\
    \Hom_\D(\U\even,\U\odd) & 0
\end{pmatrix}\,.
\end{align}

\phantomsection\label{phsec:kappas-even-and-odd}

The isomorphism classes of $\U\even$ and $\U\odd$ are determined, respectively, by maps $\kappa_\bz, \kappa_\bo\from G/T \to \ZZ_{\geq 0}$ given by $\kappa_i (x) = \dim \U^i_x$ for all $i\in \ZZ_2$ and $x\in G/T$. 
Note that, in this case, $G^\#/T = G^\#/T^+$ is the disjoint union of $G/T$ and $(e, \bar 1) \cdot G/T$ and, clearly, $\kappa ((g,i)T) = \kappa_i (gT)$, for all $i\in \ZZ_2$ and $g\in G$. 
Selecting $\D$-bases for $\U\even$ and $\U\odd$, we can identify $\End_\D (\U)$ with matrix superalgebra $M_{k_\bz | k_\bo}(\D)$ where $k_i \coloneqq |\kappa_i|$.

If $\D$ is an odd graded-division superalgebra and $\U \neq 0$, then neither $\U\even$ nor $\U\odd$ is a $\D$-submodule.
However, one can flip the parity of elements in any graded basis of $\U$ by multiplying them with a homogeneous element from $\D\odd$.
Consequently, it is always possible to select an \emph{even $\D$-basis} for $\U$, \ie, a graded $\D$-basis consisting solely of even elements.
Using such a basis to write $\End_\D(\U)$ as $M_k(\D) \iso M_k(\FF)\tensor \D$, we see that all the information about the canonical $\ZZ_2$-grading is encoded in $\D$:
\[
    M_k(\D)\even =  M_k(\FF)\tensor \D\even \AND M_k(\D)\odd =  M_k(\FF)\tensor \D\odd\,.
\]

\begin{convention}\label{conv:pick-even-basis}
    If $\D$ is an odd graded-division superalgebra, we always choose the graded basis $\B$ to be an even basis.
\end{convention}

From the point of view of the map $\kappa$, this convention corresponds to choosing $\gamma \in G^{|\kappa|}$ to realize $\kappa$.  
Since the map $\iota\from G/T^+ \to G^\#/T$ given by $\iota (gT^+) = (g, \bar 0) T$ is a bijection, we can work with $\kappa \circ \iota\from G/T^+ \to \ZZ_{\geq 0}$, which we will, by abuse of notation, also denote by $\kappa$. 

\subsection{Graded-division superalgebras over an algebraically closed field}\label{subsec:D-alg-closed}

In this subsection, we will assume that the grading group $G$ is abelian and that $\FF$ is an algebraically closed field with $\Char \FF \neq 2$.
We will parametrize and classify the finite-dimensional graded-division superalgebras and develop a theory regarding this parametrization.
Unless stated otherwise, $\D$ will be assumed to be a fixed finite-dimensional graded-division superalgebra.

\subsubsection{Parametrization of finite-dimensional graded-division superalgebras}\label{subsubsec:param-D}

Picking nonzero elements $X_t \in \D_t$, we have that $\D_t = \D_e X_t$.
Since $G$ is abelian, the elements $X_tX_s$ and $X_sX_t$ are both in $\D_{ts}$.
Since $\FF$ is algebraically closed, we have $\D_e = \FF$ for any finite-dimensional $\D$ and, hence, $X_tX_s$ is a nonzero scalar multiple of $X_sX_t$.
We can then define the map $\beta\from T\times T \to \FF^\times$ by
\[\label{eq:def-beta}
    \forall t,s \in T, \forall X_t\in \D_t, \forall X_s\in \D_s, \quad X_tX_s = \beta(t,s) X_sX_t\,.
\]
Thus, $\beta$ is a measure of how $\D$ fails to be commutative.
In the superalgebra setting, it is convenient to measure how $\D$ fails to be supercommutative, so we define $\tilde\beta\from T \times T \to \FF^\times$ by
\[\label{eq:def-tilde-beta}
    \forall t,s \in T, \forall X_t\in \D_t, \forall X_s\in \D_s, \quad X_tX_s = (-1)^{p(t)p(s)}\tilde\beta(t,s) X_sX_t\,.
\]

\begin{defi}\label{def:bicharacter}
    A map $b\from T\times T \to \FF^\times$ is a \emph{bicharacter} if, for every $t\in T$, both maps $b(t, \cdot)\from T \to \FF^\times$ and $b(\cdot, t)\from T \to \FF^\times$ are characters. 
    We say that a bicharacter $b$ is 
    \begin{enumerate}
        \item \emph{alternating} if $b(t,t) = 1$ for all $t\in T$;
        \item \emph{skew-symmetric} if $b(t,s) = b(s,t)\inv$ for all $t,s\in T$.
    \end{enumerate}
    Clearly, every alternating bicharacter is skew-symmetric. 
    The \emph{radical} of a skew-sym\-metric bicharacter is the subgroup of $T$ defined by
    \[
        \rad b \coloneqq \{ t \in T \mid b(t, T) =1 \}.
    \]
    If $\rad b = \{e\}$, we say that $b$ is \emph{nondegenerate}.
\end{defi}

Straightforward verification shows that both $\beta$ and $\tilde\beta$ are bicharacters, are independent of the choice of the elements $X_t\in \D_t$, and are related by the equation $\tilde\beta(t,s) = (-1)^{p(t)p(s)}\beta(t,s)$.
It follows from \cref{eq:def-beta,eq:def-tilde-beta} that the support of $Z(\D)$ is $\rad \beta$, and the support of $sZ(\D)$ is $\rad \tilde\beta$. 

Note that while $\beta$ is alternating, $\tilde\beta$ is skew-symmetric, but not alternating if $T^- \neq \emptyset$.
It is well known (\cite{MR529734}) and easy to check that for any abelian group $T$ and any skew-symmetric bicharacter $\tilde\beta$ on $T$, there is an alternating bicharacter $\beta$ on $T$ and a group homomorphism $p\from T \to \ZZ_2$ such that $\tilde\beta(t,s) = (-1)^{p(t)p(s)}\beta(t,s)$ for all $t,s\in T$.
Hence, a pair $(T, \tilde\beta)$, where $\tilde\beta$ is a skew-symmetric bicharacter on $T$, carries the same information as a triple $(T, \beta, p)$, where $\beta$ is an alternating bicharacter and $p\from T\to \ZZ_2$ is a group homomorphism.
In \cite{caios_thesis} we decided to use one notation but here we are using the other.

\begin{prop}\label{prop:parametrization-T-beta}
    The pair $(T, \tilde\beta)$ determines the isomorphism class of a finite-dimensional graded-division superalgebra $\D$.
    Moreover, for every finite abelian group $T$ and every skew-symmetric bicharacter $\tilde\beta\from T \times T \to \FF^\times$, there is a graded-division superalgebra associated with $(T,\tilde\beta)$. 
\end{prop}

\begin{proof}
    Since $\tilde\beta$ determines $\beta$ and $p$, this result follows from the fact that graded-division algebras are classified by the pairs $(T,\beta)$ where $T$ is a finite abelian group and $\beta$ is an alternating bicharacter (see \cite[Section 2.2]{EK_d4} or \cite[Subsection 2.2.3]{caios_thesis} for an alternative proof). 
\end{proof}

It is known (see, \eg, \cite[Section 2.2]{livromicha} or \cite[Corollary 2.37]{caios_thesis}) that $\D$ is simple as an algebra \IFF $\beta$ is nondegenerate.
In \cref{prop:simple-tilde-beta}, we will show that $\D$ is simple as a superalgebra \IFF $\tilde\beta$ is nondegenerate.

\subsubsection{Automorphisms of \texorpdfstring{$\D$}{D}}

We will need some information about automorphisms of $\D$ as a graded superalgebra.

\begin{lemma}\label{lemma:Aut(D)-widehat-T}
    An $\FF$-linear map $\psi_0\from \D \to \D$ is an automorphism of $\D$ \IFF there is $\chi \in \widehat T$ such that $\psi_0( X_t ) = \chi(t) X_t$ for all $t\in T$. 
\end{lemma}

\begin{proof}
    Any invertible degree-preserving linear map $\psi\from \D \to \D$ is determined by a map $\chi\from T \to \FF^\times$ such that $\psi(X_t) = \chi(t) X_t$, for all $t\in T$. 
    It is easy to see that $\psi$ is an automorphism \IFF $\chi$ is a group homomorphism, \ie, $\chi \in \widehat T$. 
\end{proof}

\begin{lemma}\label{prop:all-central-automorphisms-of-D-are-superinner}  
    Let $\psi_0\from \D \to \D$ be an automorphism that restricts to the identity on $sZ(\D)$. 
    Then $\psi_0 = \operatorname{sInt}_{d}$ for some nonzero $G^\#$-homogeneous element $d\in \D$. 
\end{lemma}

\begin{proof}
    By \cref{lemma:Aut(D)-widehat-T}, there is a character $\chi\in \widehat{T}$ such that $\psi_0(X_t) = \chi(t) X_t$, for all $t\in T$. 
    Since $\psi_0$ acts trivially on $sZ(\D)$, we have $\chi(t) = 1$ for all $t \in \rad \tilde\beta$. 
    
    Set $\barr T \coloneqq T/\rad \tilde\beta$ and let $\pi \from T \to \barr T$ be the quotient map. 
    The induced bicharacter $b\from \overline{T} \times \overline{T} \to \FF^\times$, defined by $b(\pi(s), \pi(t)) \coloneqq \tilde\beta(s,t)$, is nondegenerate and skew-symmetric. 
    Therefore, the map $\barr T \to \widehat{\barr T}$ given by $t \mapsto b(t, \cdot)$ is a group isomorphism.
    
    Since $\chi$ is trivial on $\rad \tilde\beta$, there exists a $\barr \chi \in \widehat{\barr T}$ such that $\chi = \barr \chi \circ \pi$.
    By nondegeneracy of $b$, there is an element $t\in T$ such that $b(\pi(t), \cdot ) = \barr \chi$. 
    A straightforward computation shows that $\psi_0 = \operatorname{sInt}_{d}$ where $d\in \D$ is any nonzero homogeneous element of degree $t$.
\end{proof}

\subsubsection{Radicals and parity elements}\label{para:radicals-betas}

As we have seen, the center and supercenter of $\D$ correspond, respectively, to the radicals of $\beta$ and $\tilde\beta$. 
From now on, we will also consider another bicharacter: $\beta^+$, which is the restriction of $\beta$ to $T^+ \times T^+$.
The radical of $\beta^+$ corresponds to the center of $\D\even$.
The interplay between these three radicals will be important in our theory.

\begin{lemma}\label{lemma:rad-tilde-beta}
	$\rad \tilde\beta = (\rad \beta)\cap T^+$.
\end{lemma}

\begin{proof}
    First note that if $t_1 \in T^-$, then $t_1 \not\in \rad \tilde\beta$ since $\tilde \beta (t_1,t_1) = (-1)^{|t_1|} = -1$.
    Now, if $t_0\in T^+$, then $\tilde\beta (t_0, \cdot) = \beta (t_0, \cdot)$. Hence
	$\tilde\beta (t_0, T) = 1$ \IFF $\beta (t_0, T) = 1$. Thus $(\rad \tilde\beta)\cap T^+ = (\rad \beta)\cap T^+$, concluding the proof.
\end{proof}

\begin{cor}\label{cor:Z-sZ-division}
    $sZ(\D) = Z(\D)\cap \D\even$. \qed
\end{cor}

\begin{defi}\label{def:parity-element}
    We call an element $t_p \in T$ a \emph{parity element} if $\tilde\beta(t_p, t) = (-1)^{p(t)}$ for all $t\in T$, and denote the set of all parity elements by $T_p$.
\end{defi}

Note that for a nonzero homogeneous element $d\in \D$, its $G^\#$-degree is a parity element $t_p \in T$ \IFF the super inner automorphism $\operatorname{sInt}_d$ is the parity automorphism.
As a consequence of \cref{prop:all-central-automorphisms-of-D-are-superinner,cor:Z-sZ-division}, we have:

\begin{cor}\label{cor:t_p-exists}
    The set $T_p$ is non-empty. \qed
\end{cor}

Given a parity element $t_p$, we observe that $T_p = t_p (\rad \tilde\beta)$.
However, there is another way to describe $T_p$ relating the radicals of $\beta$, $\beta^+$ and $\tilde\beta$.

\begin{lemma}\label{lemma:set-parity-elements}
    If $T^- =\emptyset$, then $T_p = \rad \tilde\beta = \rad \beta = \rad \beta^+$.
    If $T^- \neq \emptyset$, then $T_p \cap T^+ = (\rad \beta^+) \smallsetminus (\rad\beta)$ and $T_p \cap T^- = (\rad \beta) \cap T^-$.
\end{lemma}

\begin{proof}
    The first assertion is trivial, so we assume $T^- \neq \emptyset$.
    
    Let $t_0 \in T^+$ be any element. 
    If $t_0 \in T_p$, then clearly $t_0 \in \rad\beta^+$ but, for any $t_1 \in T^-$, $\beta(t_0,t_1) = \tilde\beta(t_0,t_1) = -1$, so $t_0 \not\in \rad \beta$. 
    Conversely, if $t_0 \in \rad \beta^+ \setminus \rad\beta$, we have that $\beta(t_0,t_1) \neq 1$ for some $t_1 \in T^-$. Since $t_1^2 \in T^+$, $\beta(t_0, t_1)^2 = \beta(t_0, t_1^2) = 1$ and, hence, $\beta(t_0, t_1) = - 1$. Then $\beta(t_0, T^-) = \beta(t_0, t_1 T^+) = \beta(t_0, t_1) = -1$, proving that $t_0$ is a parity element.

    Now, let $t_1 \in T^-$ be any element. Then $t_1$ is a parity element \IFF $\tilde\beta(t_1, t) = (-1)^{p(t)} = (-1)^{p(t_1) p(t)}$ for all $t\in T$ which, by the definition of $\tilde\beta$, is equivalent to $\beta(t_1, t) = 1$ for all $t\in T$. 
    The result follows. 
\end{proof}

\begin{cor}\label{cor:radical-with-parity}
    We have that either $T_p \subseteq T^+$ or $T_p \subseteq T^-$. 
    If $T_p \subseteq T^+$, then $\rad \beta = \rad \tilde\beta$ and $\rad \beta^+ = (\rad \tilde\beta) \cup T_p$. 
    If $T_p \subseteq T^-$, then $\rad \beta = (\rad \tilde\beta) \cup T_p$ and $\rad \beta^+ = \rad \tilde\beta$. 
\end{cor}

\begin{proof}
    Since $T_p$ is a coset of $\rad \tilde \beta$ and $\rad \tilde \beta \subseteq T^+$ by \cref{lemma:rad-tilde-beta}, the first assertion follows directly.
    The rest follows from \cref{lemma:set-parity-elements}.
\end{proof}

Restricting this discussion about radicals and parity elements to the case where $\tilde\beta$ is nondegenerate, we have:

\begin{prop}\label{lemma:beta-deg-beta-tilde-nondeg}
    The bicharacter $\tilde\beta$ is nondegenerate \IFF $|T_p| = 1$. In this case, the unique parity element $t_p$ has order $2$. Furthermore:
    \begin{enumerate}
        \item \label{item:beta-nondegenerate-case} If $\beta$ is nondegenerate, then $\tilde\beta$ is nondegenerate, $t_p \in T^+$ and $\rad\beta^+ = \langle t_p \rangle$;
        \item \label{item:beta-degenerate-case} If $\beta$ is degenerate, then $\tilde\beta$ is nondegenerate \IFF $\beta^+$ is nondegenerate. In this case, $t_p \in T^-$ and $\rad \beta = \langle t_p \rangle$. \qed
    \end{enumerate}
\end{prop}

\subsubsection{Standard realizations of \texorpdfstring{$\D$}{D} that are simple as a superalgebra}\label{subsubsec:standard-realizations}

We will now demonstrate that case \cref{item:beta-nondegenerate-case} in \cref{lemma:beta-deg-beta-tilde-nondeg} corresponds to superalgebras of type $M$ and case \cref{item:beta-degenerate-case} to superalgebras of type $Q$.

Assuming $\beta$ is nondegenerate, we can construct a graded-division superalgebra associated to $(T,\tilde\beta)$ using matrices. 
For that, we follow \cite[Remark 18]{EK15} (see also \cite[Remark 2.16]{livromicha}).

First, we can decompose the group $T$ as $A\times B$, where the restrictions of $\beta$ to each of the subgroups $A$ and $B$ are trivial (see \cite[page 36]{livromicha}) and, hence, $A$ and $B$ are in duality by $\beta$, \ie, the map $A \to \widehat B$ given by $a \mapsto \beta(a, \cdot)$ is an isomorphism of groups (note that, in particular, $|T|$ is a perfect square). 

\begin{defi}\label{def:standard-realization-M}
Let $V$ be the vector space with basis $\{e_b\}_{b\in B}$ (\ie, $V$ is the vector space underlying the group algebra $\FF B$). 
For each $a\in A$, define $X_a\in \End(V)$ by
\[
    \forall b' \in B, \quad X_a (e_{b'}) \coloneqq \beta(a, b')e_{b'},
\]
and, for each $b\in B$, define $X_b\in \End(V)$ by
\[
    \forall b' \in B, \quad X_b (e_{b'}) \coloneqq e_{bb'}.
\]
Finally, we define $X_{ab} \coloneqq X_a X_b$, for all $a\in A$ and $b\in B$.
Then, the set $\{X_t\}_{t\in T}$ is a basis of $\End(V)$, and it defines a division $T$-grading on it.
The resulting graded-division algebra $\D$ corresponds to $(T, \beta)$ and is called \emph{standard realization} associated to $(T,\beta)$.
Note that, as an algebra, $\D \iso M_\ell(\FF)$, where $\ell = \sqrt{|T|}$.
The partition $T = T^+ \cup T^-$ defines a superalgebra structure on $\D$ corresponding to the pair $(T, \tilde\beta)$.
As a superalgebra, $\D$ will be called a \emph{standard realization of type $M$} associated to $(T, \tilde\beta)$.
\end{defi}

Now suppose $\tilde\beta$ is nondegenerate but $\beta$ is degenerate.
In this case, by \cref{lemma:beta-deg-beta-tilde-nondeg}, $\beta^+$ is nondegenerate, so let $\D\even$ be a standard realization of type $M$ associated with $(T^+,\beta^+)$.
Since $t_p \in T^-$, there is an order two element $h\in G$ such that $t_p = (h, \barr 1)\in G^\#$.
We can grade $Q(1)$ with a $\langle h \rangle$-grading as in \cref{ex:Q(1)-as-grd-div-SA}.
It is easy to check that the superalgebra $\D \coloneqq Q(1) \tensor \D\even$, graded by setting $\deg (a \tensor b) \coloneqq (\deg a)(\deg b)$, is associated with $(T, \tilde\beta)$. 

\begin{defi}\label{def:standard-realization-Q}
    The $G$-graded superalgebra $\D \coloneqq Q(1) \tensor \D\even = \D\even \oplus u\D\even$, where we declare the $G$-degree of $u$ to be $h$, will be referred to as a \emph{standard realization of type $Q$} associated to $(T,\tilde\beta)$ or to $(T^+, \beta^+, h)$.
    Note that $\D \iso Q(\ell)$, where $\ell = \sqrt{|T^+|} = \sqrt{|T|/2}$.
\end{defi}

\begin{prop}\label{prop:simple-tilde-beta}
    As a superalgebra, $\D$ is simple \IFF $\tilde\beta$ is nondegenerate.
\end{prop} 

\begin{proof}
    We know that $\D$ is simple as algebra \IFF $\beta$ is nondegenerate. The construction of standard realization of type $Q$ proves that if $\tilde\beta$ is nondegenerate, then $\D$ is simple as superalgebra. It only remains the converse of this last statement.

    If $\D \iso Q(n)$, then $\D\even \iso M(n)$ is a graded-division algebra that is simple as an algebra and, hence, $\beta^+$ is nondegenerate. The result follows from \cref{lemma:beta-deg-beta-tilde-nondeg}.
\end{proof}

\subsection{Graded-simple superalgebras over an algebraically closed field}\label{subsec:grdd-simple-ass-algebraically-closed}

We continue assuming that $\FF$ is algebraically closed and $G$ is abelian.
If $R$ is a finite-dimensional graded-simple superalgebra, it follows from \cref{thm:End-over-D} that $R \iso \End_\D(\U)$, where $\D$ and $\U$ are also finite-dimensional.
Since we know how to parametrize $\D$ and $\U$, we can parametrize $\End_\D(\U)$.
First, we will do it in terms of the group $G^\#$:

\begin{defi}\label{def:E(D,U)}
    Let $\D$ be a finite-dimensional graded-division superalgebra over an algebraically closed field $\FF$ with $\Char \FF \neq 2$, and let $\U$ be a graded right $\D$-module of finite rank. 
	If $\D$ is associated to $(T, \tilde\beta)$ and $\U$ is associated to $\kappa\from G^\#/T \to \ZZ_{\geq 0}$ , we say that $(T, \tilde\beta, \kappa)$ are the \emph{parameters} of the pair $(\D, \U)$.
\end{defi}

Consider the usual $G^\#$-action on the functions $\kappa \from G^\#/T \to \ZZ_{\geq 0}$, \ie, given $g\in G^\#$, set $(g\cdot \kappa) (x) \coloneqq \kappa(g\inv x)$ for all $x\in G^\#/T$. 
Since $G$ is abelian, it is easy to see that, if $\kappa$ is associated to a graded $\D$-module $\U$, then $g \cdot \kappa$ is associated to $\U^{[g]}$. 
Note that if $(g_1, \ldots, g_k)$ is a $k$-tuple realizing $\kappa$, then $(gg_1, \ldots, gg_k)$ is a $k$-tuple realizing $g\cdot \kappa$.

If $\psi_0\from \D \to \D'$ is an isomorphism of graded algebras and $\U'$ is a graded right $\D'$-module associated to $\kappa'$, it is clear that $\dim_{\D'} \U_x' = \dim_\D (\U_x')^{\psi_0}$, for all $x \in G^\#/T$, and, hence, the graded $\D$-module $(\U')^{\psi_0}$ is also associated to $\kappa'$. 
Thus, \cref{thm:iso-abstract} becomes the following:

\begin{thm}\label{thm:iso-End_D-U-with-parameters}
	Let $(\D, \U)$ and $(\D', \U')$ be pairs as in Definition \ref{def:E(D,U)}, and let $(T, \tilde\beta, \kappa)$ and $(T', \tilde\beta', \kappa')$ be their parameters. 
	Then $\End_\D (\U) \iso \End_{\D'} (\U')$ \IFF $T = T'$, $\tilde\beta = \tilde\beta'$, and $\kappa$ and $\kappa'$ belong to the same $G^\#$-orbit. \qed 
\end{thm}

The distinction of graded-division algebras being even or odd allowed us to reparametrize their graded modules (see \cref{para:module-over-even-or-odd-D}).
The same can be done here:

\begin{defi}\label{def:E(D-U)-super}
    Let $\D$ be a finite-dimensional graded-division superalgebra associated to $(T, \tilde\beta)$, and let $\U$ be a graded right $\D$-module of finite rank. 
    We parametrize the pair $(\D, \U)$ by:
    \begin{enumerate}
        \item\label{it:even-param} the quadruple $(T, \beta, \kappa_\bz, \kappa_\bo)$, if $\D$ is even and $\U$ is associated to the maps $\kappa_\bz, \kappa_\bo \from G/T \to \ZZ_{\geq 0}$;
        \item\label{it:odd-param} the quadruple $(T, \tilde\beta, \kappa)$, if $\D$ is odd and $\U$ is associated to the map $\kappa\from G/T^+ \to \ZZ_{\geq 0}$.
    \end{enumerate}
\end{defi}

\phantomsection\label{phsec:odd-param-in-terms-of-G}

Observe that in \cref{it:odd-param}, we do not present a parametrization solely in terms of the group $G$, as $T \not\subseteq G$.  
While such a parametrization can be constructed (see \cite[Section 2.3]{caios_thesis}), it requires intricate technical steps, and we have opted to exclude it here for conciseness.  
Nevertheless, in cases where \( T \) can be naturally described in terms of \( G \), as in \cref{def:Gamma-T-beta-kappa-Q}, we will include it explicitly.  

We can easily adapt \cref{thm:iso-End_D-U-with-parameters}:

\begin{thm}\label{thm:iso-D-even}\label{thm:iso-D-odd}
	Let $(\D, \U)$ and $(\D', \U')$ be pairs as in Definition \ref{def:E(D-U)-super}.
    \begin{enumerate}
        \item If one of $\D$ and $\D'$ is even and the other is odd, then $\End_\D (\U) \not\iso \End_{\D'} (\U')$.
        \item If both $\D$ and $\D'$ are even, let $(T, \beta, \kappa_\bz, \kappa_\bo)$ and $(T', \beta', \kappa_\bz', \kappa_\bo')$ be the parameters of $(\D, \U)$ and $(\D', \U')$, respectively.
        Then $\End_\D (\U) \iso \End_{\D'} (\U')$ \IFF $T=T'$, $\beta=\beta'$, and there is $g\in G$ such that either $g \cdot \kappa_{\bar 0}=\kappa_{\bar 0}'$ and $g \cdot \kappa_{\bar 1}=\kappa_{\bar 1}'$, or $g \cdot \kappa_{\bar 0}=\kappa_{\bar 1}'$ and $g \cdot \kappa_{\bar 1}=\kappa_{\bar 0}'$.
        \item If both $\D$ and $\D'$ are odd, let $(T, \tilde\beta, \kappa)$ and $(T', \tilde\beta', \kappa')$ be the parameters of $(\D, \U)$ and $(\D', \U')$, respectively.
        Then $\End_\D (\U) \iso \End_{\D'} (\U')$ \IFF $T=T'$, $\tilde\beta = \tilde\beta'$, and $\kappa$ and $\kappa'$ belong to the same $G$-orbit. \qed
        \end{enumerate}
\end{thm}

\phantomsection\label{para:gradings-on-M-and-Q}

We now specialize this result for simple superalgebras.
By \cref{prop:simple-tilde-beta}, this is equivalent to considering the case where $\tilde\beta$ is nondegenerate.
For these, we can construct models for the gradings explicitly using elementary gradings and standard realizations (see \cref{para:elementary-grading,subsubsec:standard-realizations}, respectively).

We start with even gradings on $M(m,n)$:

\begin{defi}\label{def:Gamma-T-beta-kappa-even}
    Let $m, n\in \ZZ_{\geq 0}$, not both zero, and consider a finite subgroup $T \subseteq G$, a nondegenerate bicharacter $\beta\from T\times T \to \FF^\times$ and maps $\kappa_\bz, \kappa_\bo \from G/T \to \ZZ_{\geq 0}$ with finite support such that $k_\bz \sqrt{|T|} = m$ and $k_\bo \sqrt{|T|} = n$, where $k_i \coloneqq |\kappa_i|$.
    Choose:
    \begin{enumerate}
        \item $\D$ to be a standard realization of type $M$ associated to $(T,\beta)$;
        \item $\gamma_\bz$ and $\gamma_\bo$ to be tuples realizing $\kappa_\bz$ and $\kappa_\bo$, respectively.
    \end{enumerate}
    Consider the superalgebra $M(k_\bz, k_\bo)$ equipped with the elementary grading associated to $(\gamma_\bz, \gamma_\bo)$.
    The even grading $\Gamma_M(T, \beta, \kappa_\bz, \kappa_\bo)$ on $M(m,n)$ is induced by the identification $M(m,n) \iso M_{k_\bz | k_\bo}(\D) = M(k_\bz, k_\bo) \tensor \D$ via the Kronecker product.
    The resulting graded superalgebra is denoted $M(T, \beta, \kappa_\bz, \kappa_\bo)$. 
\end{defi}

From \cref{thm:iso-D-even}, we have:

\begin{cor}\label{cor:iso-M-even}
    Every even $G$-grading on $M(m,n)$ is isomorphic to one of the form $\Gamma_M(T, \beta, \kappa_\bz, \kappa_\bo)$, as in \cref{def:Gamma-T-beta-kappa-even}.
    Two such gradings $\Gamma_M (T, \beta, \kappa_\bz, \kappa_\bo)$ and $\Gamma_M (T', \beta', \kappa_\bz', \kappa_\bo')$ are isomorphic \IFF $T = T'$, $\beta = \beta'$ and there is $g\in G$ such that either $g \cdot \kappa_{\bar 0}=\kappa_{\bar 0}'$ and $g \cdot \kappa_{\bar 1}=\kappa_{\bar 1}'$, or $g \cdot \kappa_{\bar 0}=\kappa_{\bar 1}'$ and $g \cdot \kappa_{\bar 1}=\kappa_{\bar 0}'$.  \qed
\end{cor}

\begin{remark}\label{rmk:m-different-n-even-grading}
    If $m \neq n$, then $|\kappa_\bz| \neq |\kappa_\bo|$ and, hence, only the case $g \cdot \kappa_{\bar 0}=\kappa_{\bar 0}'$ and $g \cdot \kappa_{\bar 1}=\kappa_{\bar 1}'$ is possible.
\end{remark}

Now, let us consider the odd gradings on $M(m,n)$. 
Recall that, by \cref{lemma:odd-M-m=n}, we have $m = n$. 

\begin{defi}\label{def:Gamma-T-beta-kappa-odd}
    Let $n > 0$ be a natural number and consider a finite subgroup $T \subseteq G^\#$, $T\not\subseteq G$, a nondegenerate alternating bicharacter $\beta\from T\times T \to \FF^\times$ and a map $\kappa\from G/T^+ \to \ZZ_{\geq 0}$ with finite support such that $k \sqrt{|T|} = n$, where $k \coloneqq |\kappa|$.
    Let $p\from G^\#=G\times \ZZ_2 \to \ZZ_2$ be the projection on the second component and define $\tilde\beta\from T\times T \to \FF^\times$ by $\tilde\beta(t,s) \coloneqq \sign{t}{s} \beta(t,s)$ for all $t,s\in T$.
    Choose:
    \begin{enumerate}
        \item $\D$ to be a standard realization of type $M$ associated to $(T,\tilde\beta)$;
        \item $\gamma$ to be a tuple realizing $\kappa$.
    \end{enumerate}
    Consider the algebra $M_{k}(\FF)$ equipped with the elementary grading associated to $\gamma$.
    The odd grading $\Gamma_M(T, \tilde\beta, \kappa)$ on $M(n,n)$ is induced by the identification $M(n,n) \iso M_k(\D) = M_{k}(\FF) \tensor \D$ via the Kronecker product.
    The resulting graded superalgebra is denoted $M(T, \tilde\beta, \kappa)$. 
\end{defi}

\begin{cor}\label{cor:iso-M-odd}
    Every odd $G$-grading on $M(n,n)$ is isomorphic to $\Gamma_M (T, \tilde\beta,\kappa)$ as in \cref{def:Gamma-T-beta-kappa-odd}. 
    Two such gradings $\Gamma_M (T, \tilde\beta, \kappa)$ and $\Gamma_M (T', \tilde\beta', \kappa')$ are isomorphic \IFF $T = T'$, $\tilde\beta = \tilde\beta'$, and there is a $g\in G$ such that $g\cdot \kappa = \kappa'$. \qed
\end{cor}

\phantomsection\label{phsec:Q-assoc-only-G}

Finally, we classify the gradings on the superalgebra $Q(n)$. 
Note that we only have odd gradings in this case, but we have a parametrization with no reference to $G^\#$.

\begin{defi}\label{def:Gamma-T-beta-kappa-Q}
    Let $n > 0$ be a natural number and consider a finite subgroup $T^+ \subseteq G$, a nondegenerate bicharacter $\beta\from T^+ \times T^+ \to \FF^\times$, an element $h\in G$ such that $h^2 = 1$ and a map $\kappa\from G/T^+ \to \ZZ_{\geq 0}$ with finite support such that $k \sqrt{|T^+|} = n$, where $k \coloneqq |\kappa|$.
    Choose:
    \begin{enumerate}
        \item $\D$ to be a standard realization of type $Q$ associated to $(T^+, \beta^+, h)$;
        \item $\gamma$ to be a tuple realizing $\kappa$.
    \end{enumerate}
    Consider the algebra $M_{k}(\FF)$ equipped with the elementary grading associated to $\gamma$.
    The grading $\Gamma_Q (T^+, \beta^+, h, \kappa)$ on $Q(n)$ is induced by the identification $Q(n) \iso M_k(\D) = M_{k}(\FF) \tensor \D$ via the Kronecker product.
    The resulting graded superalgebra is denoted $Q (T^+, \beta^+, h, \kappa)$. 
\end{defi}

Note that $Q (T^+, \beta^+, h, \kappa) = M(T^+, \beta^+, \kappa) \oplus u M(T^+, \beta^+, \kappa)$, where $u^2 = 1$ and the $G^\#$-degree of $u$ is $t_p = (h, \bar 1)$.

\begin{cor}\label{cor:iso-Q}
    Every $G$-grading on $Q(n)$ is isomorphic to $\Gamma_Q (T^+, \beta^+, h, \kappa)$ as in \cref{def:Gamma-T-beta-kappa-Q}. 
    Two such gradings $\Gamma_Q (T^+, \beta^+, h, \kappa)$ and $\Gamma_Q (T'^+, \beta'^+, h', \kappa')$ are isomorphic \IFF $T^+ = T'^+$, $\beta^+ = \beta'^+$, $h = h'$ and there is a $g\in G$ such that $g\cdot \kappa = \kappa'$. \qed
\end{cor}

\section{Graded-Superinvolution-Simple Superalgebras}\label{sec:grdd-sinv-simple}

This section establishes general results on graded-superinvolution-simple associative superalgebras satisfying the descending chain condition on graded left superideals.
Such superalgebras naturally split into two disjoint classes: those that are graded-simple and those that are not.
For both cases, we will need the concept of superdual of a graded $\D$-module, which we introduce in \cref{subsec:superdual}.  

We first consider the non-graded-simple case, for which the classification easily reduces to the results of \cref{sec:grdd-simple-ass}.
Indeed, as shown in \cref{prop:only-SxSsop-is-simple}, these superalgebras decompose as $S \times S\sop$ with exchange superinvolution, where $S$ is a graded-simple superalgebra.
This case is addressed in \cref{subsec:R-phi-not-graded-simple}.
All remaining subsections are dedicated to the graded-simple case, where a theory of super-Hermitian forms on graded $\D$-modules is required.

\subsection{The superdual of a graded \texorpdfstring{$\D$}{D}-module}\label{subsec:superdual}

\begin{defi}\label{def:superdual-supermodule}
    Let $\D$ be a graded-division superalgebra and let $\U$ be a graded right $\D$-supermodule of finite rank. 
    The \emph{superdual of $\U$} is defined to be \[\U\Star \coloneqq \Hom_\D (\U,\D)\,.\] 
    We give $\U\Star$ the structure of a graded \emph{left} $\D$-module using the multiplication on $\D$: if $d \in \D$ and $f \in \U \Star$, we define $d\cdot f$ by
    \[\label{eq:U-Star-left-D-module}
    \forall u\in \mc U,\quad (d\cdot f)(u) = d f(u)\,.
    \]
    Given graded right $\D$-modules of finite rank $\U$ and $\V$, and a homogeneous $\D$-linear map $L\from\U \rightarrow \V$, we define the \emph{superdual of $L$} to be the $\FF$-linear map $L\Star\from \V\Star \rightarrow \U\Star$ determined by
    \[
        \forall f\in (\V\Star)\even \cup (\V\Star)\odd,\quad L\Star (f) = (-1)^{|L||f|} f \circ L\,.
    \] 
    It is easily seen that $L\Star$ is a homomorphism of left $\D$-modules.
    We extend the definition to non-homogeneous maps in $\Hom_\D (\U, \V)$ by linearity. 
\end{defi}

If $\B = \{u_1, \ldots, u_k\}$ is a graded basis, we can consider its two \emph{superdual bases} in $\U\Star$: ${}\Star \mc B = \{{}\Star u_1, \ldots, {}\Star u_k\}$ and $\mc B\Star = \{u_1\Star, \ldots, u_k\Star\}$, where ${}\Star u_i \from \U \rightarrow \D$ is defined by ${}\Star u_i(u_j) = \delta_{ij}$ and $u_i\Star \from \U \rightarrow \D$ is defined by $u_i\Star (u_j) = (-1)^{|u_i||u_j|} \delta_{ij}$. 
Clearly, $\deg ({}\Star u_i) = \deg  (u_i\Star) = (\deg u_i)\inv$.

\begin{remark}\label{rmk:supertranspose-U-Star}
	In the case $\D = \FF$ and $L\from \U \to \U$, if we denote by $[L]$ the matrix of $L$ with respect to the graded basis $\mc B$, with even elements preceding the odd ones, then the supertranspose $[L]\sT$ is the matrix of $L\Star\from \U\Star \to \U\Star$ with respect to the superdual basis $\mc B\Star$.
\end{remark}

It is easy to see that $\D\sop$ is also a graded-division superalgebra and, since $\U\Star$ is a graded left $\D$-module, we will regard $\U\Star$ as a graded right $\D\sop$-module by means of the action
\[\label{eq:U-Star-right-D-sop-module}
    \forall d\in \D\even \cup \D\odd,\,f\in (\U\Star)\even \cup (\U\Star)\odd,\quad  f \cdot \bar d \coloneqq (-1)^{|d||f|} d \cdot f\,.
\]

\begin{lemma}\label{lemma:double-dual}
    Let $\U$ be a $G$-graded $\D$-module of finite rank.
    \begin{enumerate}
        \item \label{it:double-dual} The $\FF$-linear map $\epsilon\from \U \to \U^{\star\star}$ defined by $\epsilon(u)(f) = \sign{u}{f}\, \overline{f(u)}$, for all $u \in \U\even \cup \U\odd$ and $f \in (\U\Star)\even \cup (\U\Star)\odd$, is an isomorphism of graded right $\D$-supermodules. 
        \item \label{it:super-anti-iso} The map $\End_\D (\U) \rightarrow \End_{\D\sop} (\U\Star)$ defined by $L \mapsto L\Star$ is a super-anti-isomorphism.
    \end{enumerate}
\end{lemma}

\begin{proof}
    For \ref{it:double-dual}, with routine computations using the rule of signs, one can check that $\epsilon(u) \in \U^{\star\star}$ (\ie, it is a $\D\sop$-linear map from $\U\Star$ to $\D\sop$) and that $\epsilon$ is $\D$-linear.
    We note in passing that the sign in the definition of $\epsilon$ is essential: without it, $\epsilon$ would not be well-defined in the case of odd $\D$.
    To see that $\epsilon$ is an isomorphism, let $\{u_1, \ldots, u_k\}$ be a graded basis of $\U$. 
    We have that $\epsilon_\U(u_i)(u_j\Star) = \sign{u_i}{u_j}\, \overline{u_j\Star(u_i)} = \delta_{ij}$, \ie, $\epsilon(u_i) = {}\Star (u_i \Star) = ({}\Star u_i) \Star$. 
    Therefore, $\epsilon_\U$ sends a graded basis to a graded basis.

    For \ref{it:super-anti-iso}, one can easily check that the inverse of the function given is the map $L \mapsto \epsilon\inv \circ L\Star \circ \epsilon$, for all $L\in \End_{\D\sop} (\U\Star)$, 
    and that $(L\circ S)\Star = \sign{L}{S} S\Star \circ L\Star$ for all $L, S \in \End_\D (\U)\even \cup \End_\D (\U)\odd$.
\end{proof}

We will now find the parameters describing the $G$-graded superalgebra $\End_{\D\sop} (\U\Star)$.

\begin{prop}
    Suppose $\FF$ is algebraically closed.
    Let $\D$ be associated to $(T, \tilde\beta)$ and $\U$ be associated to $\kappa\from G^\#/T \to \ZZ_{\geq 0}$.
    Then $\D\sop$ is associated to $(T, \tilde\beta\inv)$ and $\U\Star$ is associated to $\kappa\Star \from G^\#/T \to \ZZ_{\geq 0}$ defined by $\kappa\Star (x) \coloneqq \kappa (x\inv)$.
\end{prop}

\begin{proof}
    If $(T, \tilde\beta)$ is the pair associated with $\D$, then it is clear that $\supp \D\sop = T$. 
    Moreover, if $s,t \in T$, $0 \neq X_s \in \D_s$ and $0 \neq X_t \in \D_t$, then, following the notation in \cref{def:superopposite}, we have:
    \begin{align}
        \overline{X_s} \, \overline{X_t} &= \sign{s}{t} \, \overline{X_tX_s} =  \tilde\beta(t, s) \overline{X_sX_t}\\ &= \sign{s}{t}\, \tilde\beta(t, s) \overline{X_t} \, \overline{X_s} = \sign{s}{t}\, \tilde\beta(s, t)\inv \, \overline{X_t} \, \overline{X_s}.
    \end{align}
    We conclude that $(T,\tilde\beta\inv)$ is the pair associated to $\D\sop$.

    Since $G$ is abelian, $G^\#/T$ is a group and, hence, the map $G^\#/T \to G^\#/T$ given by $x \mapsto x\inv$ is well-defined. 
    From the construction of the superdual basis, it is easy to see that $\dim_\D \U_x = \dim_{\D\sop} \U\Star_{x\inv}$, for all $x \in G^\#/T$.  
\end{proof}

It is straightforward to translate this to the maps $G/T \to \ZZ_{\geq 0}$ (even $\D$) and $G/T^+ \to \ZZ_{\geq 0}$ (odd $\D$) associated to the $G$-graded supermodule $\U$. 
If $\D$ is even and $\kappa_\bz, \kappa_\bo$ are the maps associated to $\U$, then $\kappa_\bz\Star, \kappa_\bo\Star$ are the maps associated to $\U\Star$. 
If $\D$ is odd and $\kappa$ is the map associated to $\U$, then $\kappa\Star$ is the map associated to $\U\Star$.   

\subsection{Non-graded-simple case}\label{subsec:R-phi-not-graded-simple}

Our goal now is to classify, up to isomorphism, the superalgebras of the form $S\times S\sop$ endowed with exchange superinvolution, where $S$ is graded-simple.
We begin with a straightforward result:

\begin{lemma}\label{lemma:iso-SxSsop}
    Let $S_1$ and $S_2$ be graded-simple superalgebras. 
    Then $S_1\times S_1\sop \iso S_2\times S_2\sop$ as graded superalgebras with superinvolution \IFF $S_1 \iso S_2$ or $S_1 \iso S_2\sop$ as graded superalgebras. \qed
\end{lemma}

When convenient, we can replace $S\sop$ with an isomorphic graded superalgebra. 
Specifically, if $\theta\from S \to S'$ is a su\-per-an\-ti-iso\-mor\-phism of graded superalgebras, then $S\times S\sop$ with the exchange superinvolution is isomorphic to $S\times S'$ endowed with the superinvolution $(s_1, s_2) \mapsto (\theta\inv (s_2), \theta (s_1))$. 

\begin{defi}\label{defi:superdual-exchange}
    Let $\D$ be a graded-division superalgebra and $\U$ be a graded right $\D$-supermodule of finite rank. 
    Recall that $\U\Star \coloneqq \Hom_\D(\U, \D)$ is a graded right $\D\sop$-module, and that the map $\End_\D (\U) \to \End_{\D\sop} (\U\Star)$ given by $L \mapsto L\Star$ is a super-anti-isomorphism whose inverse is $L \mapsto {}\Star L$. 
    We define $\Eex (\D, \U)$ to be the graded superalgebra $\End_\D (\U) \times \End_{\D\sop} (\U\Star)$ endowed with the superinvolution $(L_1, L_2) \mapsto ({}\Star L_2, L_1\Star)$.
\end{defi}

Combining \cref{thm:iso-D-even,thm:iso-D-odd,lemma:iso-SxSsop} with the description of parameters for $\D\sop$ and $\U\Star$ at the end of the previous subsection, we obtain:

\begin{thm}\label{thm:iso-D-even-ExEsop}
	Let $(\D, \U)$ and $(\D', \U')$ be pairs as in Definition \ref{def:E(D-U)-super}, with both $\D$ and $\D'$ even. 
	Let $(T, \tilde\beta, \kappa_\bz, \kappa_\bo)$ and $(T', \tilde\beta', \kappa_\bz', \kappa_\bo')$ be the parameters of $(\D, \U)$ and $(\D', \U')$, respectively. 
	Then $\Eex(\D, \U) \iso \Eex(\D', \U')$ \IFF $T=T'$ and one of the following holds:
	\begin{enumerate}
	    \item $\tilde\beta'=\tilde\beta$ and there is $g\in G$ such that either $\kappa_{\bar 0}'=g \cdot \kappa_{\bar 0}$ and $\kappa_{\bar 1}'=g \cdot \kappa_{\bar 1}$, or $\kappa_{\bar 0}'=g \cdot \kappa_{\bar 1}$ and $\kappa_{\bar 1}'=g \cdot \kappa_{\bar 0}$;
	    \item $\tilde\beta'=\tilde\beta\inv$ and there is $g\in G$ such that either $\kappa_{\bar 0}'=g \cdot \kappa_{\bar 0}\Star$ and $\kappa_{\bar 1}'=g \cdot \kappa_{\bar 1}\Star$, or $\kappa_{\bar 0}'=g \cdot \kappa_{\bar 1}\Star$ and $\kappa_{\bar 1}'=g \cdot \kappa_{\bar 0}\Star$. \qed
	\end{enumerate}
\end{thm}

\begin{thm}\label{thm:iso-D-odd-ExEsop}
    Let $(\D, \U)$ and $(\D', \U')$ be pairs as in Definition \ref{def:E(D-U)-super}, with both $\D$ and $\D'$ odd. 
    Let $(T, \tilde\beta, \kappa)$ and $(T', \tilde\beta', \kappa')$ be the parameters of $(\D, \U)$ and $(\D', \U')$, respectively. 
	Then $\Eex(\D, \U) \iso \Eex(\D', \U')$ \IFF $T=T'$ and one of the following holds:
	\begin{enumerate}
	    \item $\tilde\beta'=\tilde\beta$ and there is $g\in G$ such that $\kappa' = g \cdot \kappa$;
	    \item $\tilde\beta'=\tilde\beta\inv$ and there is $g\in G$ such that $\kappa' = g \cdot \kappa\Star$. \qed
	\end{enumerate}
\end{thm}

\subsection{Super-anti-automorphisms and sesquilinear forms}\label{subsec:sesquilinear-forms}

We begin our study of superinvolutions on graded-simple associative superalgebras.  
It is worthwhile to first consider the general case of su\-per-an\-ti-au\-to\-mor\-phisms, specializing to superinvolutions at a later stage (\cref{subsec:superinv-sesquilinear-forms}).

The main result of this subsection is \cref{thm:vphi-iff-vphi0-and-B}, but we begin with preliminary results that are necessary for its proof.
There are many technical details to pay attention to, for example, whether maps are written on the left or the right because moving a map from one side to the other slightly alters its definition (\cref{def:change-map-to-the-left}).

Fix a graded-division superalgebra $\D$ and a $G$-graded $\D$-module $\U$ of finite rank, and set $R \coloneqq \End_\D(\U)$.  
By construction, $\U$ is an $(R, \D)$-bimodule.  
In \cref{subsec:superdual}, we showed that $\U\Star \coloneqq \Hom_\D(\U, \D)$ is a graded right $\D\sop$-module and a left $\End_{\D\sop}(\U\Star)$-module.  
By \cref{lemma:double-dual}, $R\sop \iso \End_{\D\sop}(\U\Star)$, allowing us to equip $\U\Star$ with a left $R\sop$-module structure:  
\begin{equation}\label{eq:Rsop-on-left-of-U-Star}
    \forall r \in R\even \cup R\odd,\ f \in (\U\Star)\even \cup (\U\Star)\odd,\quad
    \overline{r} \cdot f \coloneqq r\Star(f) = (-1)^{\bar{r}\bar{f}} r \circ f.
\end{equation}
Thus, $\U\Star$ becomes an $(R\sop, \D\sop)$-bimodule.  
Our goal is to reinterpret this as an $(R, \D)$-bimodule using a fixed super-anti-automorphism $\vphi$ on $R$.

\begin{lemma}\label{lemma:U-star-R-sop}
    \begin{enumerate}
        \item As a left $R\sop$-supermodule, $\mc U\Star$ is graded-simple;
        \item The representation $\D\sop \to \End_{R\sop}(\mc U\Star)$ corresponding to the right $\D\sop$-action on $\U\Star$ is an isomorphism.
    \end{enumerate}
\end{lemma}

\begin{proof}
    The first assertion follows from $\U\Star$ being a graded-simple left module over $\End_{\D\sop} (\U\Star)$, and the second follows from \cref{lemma:converse-density-thm}.
\end{proof}

Even though $\D\sop$ is defined as the same space as $\D$ but with a different multiplication (\cref{def:superopposite}), it will be essential to distinguish them.
Indeed, using the isomorphism in \cref{lemma:converse-density-thm}, we will identify $\D$ with $\End_R(\U)$ and, using the isomorphism in \cref{lemma:U-star-R-sop}, we will identify $\D\sop$ with $\End_{R\sop}(\mc U\Star)$. 

On the other hand, treating the super-anti-automorphism $\vphi \colon R \to R$ as an isomorphism $R \to R\sop$ (mapping $r$ to $\overline{\vphi(r)}$), we use $\vphi$ to transform the $R\sop$-action in \cref{eq:Rsop-on-left-of-U-Star} into an $R$-action:
\begin{equation}\label{eq:R-action-back-on-the-right}
	\forall r \in R\even \cup R\odd, f\in (\U\Star)\even \cup (\U\Star)\odd,\quad 
    r\cdot f \coloneqq \vphi(r)\Star\cdot f = \sign{r}{f} f \circ \vphi(r)\,.
\end{equation}
Under these identifications, we obtain $\D\sop = \End_{R\sop}(\mc U\Star) = \End_R(\mc U\Star)$.

From \cite[Lemma 2.7]{livromicha}, $R$ has a unique graded-simple supermodule up to isomorphism and shift.  
Thus, there exists an invertible $R$-linear map $\nu \colon \mc U \to \mc U\Star$, homogeneous of degree $(g_0, \alpha) \in G^\#$.  
Fix such a $\nu$.  
Since the $R$-action is left-sided, we adopt the convention of writing $R$-linear maps on the right.  
The following lemma clarifies the nonuniqueness of such maps:

\begin{lemma}\label{lemma:nonuniqueness-of-vphi1}
    A map $\nu' \from \U \to \U\Star$ is $R$-linear and homogeneous \IFF $\nu' = \nu \bar{d}$ for some homogeneous element $\bar{d} \in \D\sop = \End_R(\U\Star)$, where juxtaposition represents composition of maps written on the right. \qed
\end{lemma}
Using the identifications $\D = \End_R(\U)$ and $\D\sop = \End_R(\U\Star)$, define the map $\vphi_0 \colon \D \to \D\sop$ by  
\[
    \forall d\in \D,\quad \vphi_0(d) = (-1)^{\bar{d}\bar{\nu}} \nu^{-1} d \nu,
\]  
where juxtaposition denotes composition of maps on the right. 
Since $\vphi_0$ is an isomorphism, we regard it as a super-anti-automorphism $\vphi_0 \colon \D \to \D$.  
This allows us to equip $\U\Star$ with a $G$-graded right $\D$-module structure via $\vphi_0$:  
\begin{equation}\label{eq:right-D-action}
    \forall f \in \U\Star,\ d \in \D,\quad f \cdot d \coloneqq f \cdot \overline{\vphi_0(d)}.
\end{equation}  
Thus, $\U\Star$ becomes a graded $(R, \D)$-bimodule, as desired. 

However, the invertible $R$-linear map $\nu \colon \U \to \U\Star$ is not $\D$-linear.  
To address this, we use $\theta \coloneqq \nu^\circ$ (see \cref{def:change-map-to-the-left}), which reinterprets $\nu$ as a left-sided map with sign adjustments.
By \cref{lemma:change-of-side-properties}, $\theta$ is not $R$-linear.
But we can check that it is $\D$-linear.
Indeed, for all $u\in \U$ and $d\in \D$,
\begin{align}\label{eq:sesquilinear-before-B}
	\theta(ud) &= (-1)^{|\theta|(|u| + |d|)} (ud)\nu
                = \sign{\theta}{u}\sign{\theta}{d} (u)(d\nu)\\
               &= \sign{\theta}{u}\sign{\theta}{d} (u)(\nu\nu\inv d \nu)
                = \sign{\theta}{u}(u)(\nu\vphi_0(d))
                = \theta(u)\vphi_0(d)\,.
\end{align} 

We will now define a third map to complement $\nu$ and $\theta$.  
Given $u \in \U$, we have $\theta(u) \in \U\Star = \Hom_\D(\U, \D)$.  
For any $v \in \U$, the evaluation $\theta(u)(v)$ lies in $\D$.  
This allows us to define a bilinear map $B \colon \U \times \U \to \D$ via  
\[\label{eq:definition-B}
    \forall u, v \in \U, \quad B(u, v) \coloneqq \theta(u)(v).
\]

\begin{defi}\label{def:sesquilinear-form}
	A map $B \colon \U \times \U \to \D$ is called a \emph{sesquilinear form on $\U$} if it is $\FF$-bilinear, $G^\#$-homogeneous if considered as a linear map $\U\tensor \U \to \D$, and there is a degree-preserving su\-per\--an\-ti\--auto\-mor\-phism $\vphi_0\from \D \to \D$ such that, for all $u,v \in \U$ and $d\in \D$,
	\begin{enumerate}
		\item $B(u,vd) = B(u,v)d$; \label{enum:linear-on-the-second}
		\item $B(ud, v) = (-1)^ {(|B| + |u|)|d|}\vphi_0(d) B(u, v)$. \label{enum:vphi0-linear-on-the-first}
	\end{enumerate}
	If we want to specify the super-anti-automorphism $\vphi_0$, we will say that $B$ is \emph{sesquilinear with respect to $\vphi_0$} or that $B$ is \emph{$\vphi_0$-sesquilinear}.
	The \emph{(left) radical} of $B$ is the set $\rad B \coloneqq \{u\in \U \mid B(u, v) = 0 \text{ for all } v\in \U\}$. 
	If $\rad B = 0$, we say that $B$ is \emph{nondegenerate} .
\end{defi}

It is straightforward to verify that $B$ as defined in \cref{eq:definition-B} is a $\vphi_0$-sesquilinear form.
One can actually prove more:

\begin{prop}\label{prop:sesquilinear-form-iff-D-linear-map}
	Fix a su\-per\--an\-ti\--auto\-mor\-phism $\vphi_0\from \D \to \D$ and consider the right $\D$-module structure on $\U\Star$ given by \cref{eq:right-D-action}.
	There is a bijection between $\vphi_0$-sesquilinear forms $B \colon \U \times \U \to \D$ and homogeneous $\D$-linear maps $\theta \from \U \to \U\Star$, given by $B(u, v) = \theta(u)(v)$.
	Moreover, $B$ is nondegenerate \IFF $\theta$ is bijective. \qed
\end{prop}




\begin{remark}\label{lemma:B-determines-vphi_0}
	Let $B \neq 0$ be a sesquilinear form on $\U$.
    \begin{enumerate}
        \item \label{it:unique-vphi0} There is a unique super-anti-automorphism $\vphi_0$ on $\D$ such that $B$ is sesquilinear with respect to $\vphi_0$;
        \item \label{it:new-vphi0} Given a homogeneous $0 \neq d \in \D$, the bilinear map $dB\from \U\times \U \to \D$, given by $(u, v) \mapsto dB(u,v)$, is a $\mathrm{sInt}_d \circ \vphi_0$-sesquilinear form, where $\operatorname{sInt}_d\from \D \to \D$ is the \emph{superinner automorphism} $\operatorname{sInt}_d (c) \coloneqq\sign{c}{d} dcd\inv$, for all $c\in \D$;
        \item \label{it:B-and-dB-nondegenerate} The form $B$ is nondegenerate \IFF $dB$ is nondegenerate.
    \end{enumerate}
\end{remark}

\begin{proof}
	For \ref{it:unique-vphi0}, since $B\neq 0$, there are homogeneous elements $u, v\in \U$ such that $0 \neq B(u,v) \in \D$. 
	Suppose $B$ is sesquilinear with respect to super-anti-automorphisms $\vphi_0$ and $\vphi_0'$ on $\D$.
	Then, for all $d\in \D\even \cup \D\odd$, we have
	\[ B(ud,v) = (-1)^ {(|B| + |u|)|d|} \vphi_0(d) B(u,v) = (-1)^ {(|B| + |u|)|d|} \vphi_0'(d) B(u,v) \]
	and, therefore, $\vphi_0(d) = \vphi_0'(d)$.
    Item \ref{it:new-vphi0} is an easy computation, and item \ref{it:B-and-dB-nondegenerate} follows from the fact that $\rad B = \rad dB$.
\end{proof}

We now relate the super-anti-automorphism $\vphi$ on $R$ to $B$.  
Using \cref{lemma:change-of-side-properties} and \eqref{eq:R-action-back-on-the-right},
\[
	\begin{split}
		B(ru,v) &= \theta (ru)(v)
                 = \sign{r}{\theta} \big(r \cdot \theta (u) \big) (v)\\
                &= \sign{r}{\theta} (-1)^{|r|(|\theta| + |u|)} \big(\theta (u) \circ \vphi(r) \big)(v) \\
                &= \sign{r}{u} \theta (u) \big( \vphi(r)v \big)
                 = \sign{r}{u} B(u,\vphi(r)v)\,.
	\end{split}
\]
We have proved one direction of:

\begin{thm}\label{thm:vphi-iff-vphi0-and-B}
	Let $\D$ be a graded-division superalgebra and let $\U$ be a nonzero right graded module of finite rank over $\D$.
	Given a super-anti-automorphism $\vphi$ on $R \coloneqq \End_\D(\U)$, there is a pair $(\vphi_0, B)$, where $\vphi_0$ is a super-anti-automorphism on $\D$ and $B\from \U \times \U \to \D$ is a nondegenerate $\vphi_0$-sesquilinear and 
    \[\label{eq:superadjunction}
		\forall r\in R\even \cup R\odd,\,\forall u, v \in \U\even \cup \U\odd,  \quad B(ru,v) = \sign{r}{u} B(u,\vphi(r)v)\,.
	\]
	Conversely, given a pair $(\vphi_0, B)$ as above, there is a unique su\-per-an\-ti-au\-to\-mor\-phism $\vphi$ on $R$ satisfying \cref{eq:superadjunction}, which we will refer as the \emph{superadjunction} with respect to $B$.
	Moreover, another pair $(\vphi_0', B')$ determines the same super-anti-automorphism $\vphi$ \IFF there is a nonzero $G^\#$-homogeneous element $d\in \D$ such that $B' = dB$ and, hence, $\vphi_0' = \mathrm{sInt}_d \circ \vphi_0$.
\end{thm}

\begin{proof}
    For the ``conversely'' part, let \(\theta\) correspond to \(B\) via \cref{prop:sesquilinear-form-iff-D-linear-map}. 
    \Cref{eq:superadjunction} becomes:
	\begin{alignat*}{2}
		\forall r\in R\even \cup R\odd,\,\forall u, v & \in \U\even \cup \U\odd,            & \theta (ru)(v)          & = \sign{r}{u} \theta(u)(\vphi(r)v)                                                   \\
		                                              &                                     &                         & = \sign{r}{u} \big(\theta (u)\circ \vphi(r)\big)(v)                                  \\
		\intertext{and, hence, equivalent to}
		\forall r\in R\even \cup R\odd,\,             & \forall u \in \U\even \cup  \U\odd, & \theta (ru)             & = \sign{r}{u} \theta(u) \circ \vphi(r).
		\addtocounter{equation}{1}\tag{\theequation}\label{eq:theta-is-almost-R-superlinear}                                                                                                                 \\
		\intertext{Recalling the definition of superdual of an operator, \cref{eq:theta-is-almost-R-superlinear} simplifies to}
		\forall r\in R\even \cup R\odd,\,\forall u    & \in \U\even \cup \U\odd,            & (\theta \circ r) (u)    & 
                                                      =  \sign{r}{\theta} \big(\big(\vphi(r)\big)\Star \circ \theta \big) (u),
		\intertext{which is the same as}
		\forall r                                     & \in R\even \cup R\odd,              & \theta \circ r          & = \sign{r}{\theta}  \big(\vphi(r)\big)\Star \circ \theta\,,
		\intertext{yielding}
		\forall r                                     & \in R\even \cup R\odd,              & \big(\vphi(r)\big)\Star & = \sign{r}{\theta}\, \theta \circ r \circ \theta\inv.
		\addtocounter{equation}{1}\tag{\theequation}\label{eq:vphi-r-Star-is-a-superconjugation}
	\end{alignat*}
	By \cref{lemma:double-dual}, the map $\End_\D (\U) \to \End_{\D\sop} (\U\Star)$ associating an operator to its superdual is invertible and, hence, $\vphi$ is uniquely determined.
	Also, the properties of the superdual imply that $\vphi$ is a super-anti-automorphism.

    For the last assertion, the ``if'' direction is immediate.
	For the ``only if'' part, let $\theta,\theta'\from \U \to \U\Star$ be defined by $\theta(u) = B(u, \cdot)$ and $\theta' \coloneqq B'(u, \cdot)$.
    Combining \cref{eq:R-action-back-on-the-right,eq:theta-is-almost-R-superlinear}, we have that $\theta(ru) = \sign{\theta}{r} r\cdot \theta(u)$ and similarly for $\theta'$.
    Defining maps $\nu,\nu'\from \U \to \U\Star$, written on the right, by $(u)\nu = \sign{u}{\theta} \theta(u)$ and $(u)\nu' = \sign{u}{\theta'} \theta'(u)$, it follows that both are $R$-linear.
    By \cref{lemma:nonuniqueness-of-vphi1}, there is $\bar d\in \D\sop$ such that $\nu' = \nu \bar d$.
	Applying \cref{eq:change-of-side-composition}, this implies \[\theta' = \sign{\theta}{\bar d} \bar d^\circ \theta.\]
	By the definition of the left $\D$-action on $\U\Star$, we have $\bar d^\circ \theta (u) = d\theta(u)$, concluding the proof.  
\end{proof}

\begin{remark}\label{conv:pick-even-form}
    If $\D$ is an odd graded division superalgebra, then a super-anti-isomorphism $\vphi$ is always determined by an even sesquilinear form, since we can substitute an odd form $B$ by $dB$ using some $d\in \D\odd$.
\end{remark}

\begin{defi}\label{def:superadjunction}
	Let $\D$ be a graded-division superalgebra, $\U$ a $G$-graded right $\D$-module of finite rank and $B$ a nondegenerate sesquilinear form on $\U$.
	We will denote by $E(\D, \U, B)$ the graded superalgebra $\End_\D(\U)$ endowed with the superadjunction $\vphi$ with respect to $B$. 
\end{defi}

We will finish this subsection with results about centers, matrix representation and simplicity.

\phantomsection\label{phsec:identification-centers-phi}

Recall that by \cref{prop:R-and-D-have-the-same-center} we could identify the centers of $\D$ and $\End_\D(\U)$.
More precisely, we have an isomorphism of $G^\#$-graded algebras $\iota\from Z(\D) \to Z(R)$ given by $\iota (d)(u) \coloneqq ud$, for all $d\in Z(\D)$ and $u\in \U$.
We can extend this result to consider the super-anti-automorphisms $\vphi_0$ and $\vphi$:

\begin{lemma}\label{prop:R-and-D-have-the-same-center-vphi}
    $\vphi (\iota (d)) = (-1)^{|B||d|} \iota (\vphi_0(d))$. 
\end{lemma}

\begin{proof}
	Fix $d\in Z(\D)\even \cup Z(\D)\odd$ and let $u, v \in \U\even \cup \U\odd$.
	On the one hand,
	\begin{align*}
		B(ud,v) = B(\iota (d)(u), v) = (-1)^{|d||u|} B(u, \vphi(\iota (d)) v).
	\end{align*}
	On the other hand,
	\begin{align*}
		B(ud, v) &= (-1)^{(|B| + |u|) |d|} \vphi_0(d) B(u, v) = (-1)^{(|B| + |u|) |d|} B(u, v) \vphi_0(d)     \\
		 & = (-1)^{|B||d| + |u||d|} B(u, v \vphi_0(d) ) 
		 = (-1)^{|B||d|} (-1)^{|u||d|} B(u, \iota (\vphi_0(d))(v)). 
	\end{align*}
	Since $B$ is nondegenerate, we conclude that $\vphi (\iota (d)) = (-1)^{|B||d|} \iota (\vphi_0(d))$, as desired. 
\end{proof}

\phantomsection\label{phsec:matrix-representation-vphi}

We now wish to express the super-anti-automorphism \(\vphi\) in terms of matrices over $\D$.
Let $\mc B = \{u_1, \ldots, u_k\}$ be a fixed homogeneous $\D$-basis of $\U$ following Convention \ref{conv:pick-even-basis} (\ie, if $\D$ is odd, we take $\mc B$ with only even elements), and use $\mc B$ to identify $\End_\D (\U)$ with $M_k (\D)$.
For brevity, we denote \(|u_i|\) by \(|i|\) for all \(i \in \{1, \ldots, k\}\) in what follows.  

\begin{defi}
    For \(X = (x_{ij}) \in M_k(\D)\), define \(\vphi_0(X)\) by applying $\vphi_0$ entry-wise, \ie, \(\vphi_0(X) \coloneqq (\vphi_0(x_{ij}))\).  
    Moreover, we extend the definition of the \emph{supertranspose} (see \cref{def:supertranspose}) to matrices over $\D$ by setting \(X\stransp \coloneqq \big( (-1)^{(|i| + |j|)|i|} x_{ji} \big)\), where \(i, j\) are the row and column indices of \(X\).  
    By Convention~\ref{conv:pick-even-basis}, \(X\stransp = X\transp\) (ordinary transpose) when $\D$ is odd.  
\end{defi}

\begin{prop}\label{prop:matrix-vphi}
	Let $\Phi$ be the matrix representing the sesquilinear form $B$, \ie, $\Phi_{}ij \coloneqq B(u_i, u_j)$.
	For \(r \in R\even \cup R\odd\) with matrix \(X \in M_k(\D)\), let \(Y \in M_k(\D)\) represent \(\vphi(r)\).  
    Then, if $\D$ is even, we have
	\begin{align}
		Y & = \Phi\inv\, \vphi_0( X\stransp )\, \Phi, \addtocounter{equation}{1}\tag{\theequation}\label{eq:matrix-vphi-D-even} 
		\intertext{and, if $\D$ is odd (following Convention~\ref{conv:pick-even-basis}),}
		Y & = \sign{B}{r}\,\Phi\inv\, \vphi_0( X\transp )\, \Phi.\addtocounter{equation}{1}\tag{\theequation}\label{eq:matrix-vphi-D-odd}
	\end{align}
\end{prop}

\begin{proof}
    First, let us write \cref{eq:superadjunction} in terms of $\B$:
    \[\label{eq:superadjunction-basis}
        \forall u_i, u_j \in \mc B, \quad B(ru_i, u_j) = \sign{r}{i} B(u_i, \vphi(r)u_j)\,.
    \]
    It suffices to consider the case where \(X\) have a single nonzero \(G^\#\)-homogeneous entry \(x_{pq}\) at \((p,q)\).  
    In this case, for all $1 \leq i \leq n$, we have \(r u_i = u_p x_{pi}\).  
    In particular, \(|r| = |p| + |q| + |x_{pq}|\).
    Computing both sides of \cref{eq:superadjunction-basis}, we have:
    \begin{align}
        B(ru_i, u_j) = B (u_p x_{pi}, u_j) & = (-1)^{ (|B| + |p|) |x_{pi}|} \vphi_0(x_{pi}) B(u_p, u_j)
		\\&= (-1)^{ (|B| + |p|) |x_{pi}|} \vphi_0(x_{pi}) \Phi_{pj},
        \intertext{and}
        \sign{r}{i} B(u_i, \vphi(r)u_j) &=\sign{r}{i} B\bigg(u_i, \sum_{\ell=1}^k u_\ell y_{\ell j}\bigg)\\
        &= (-1)^{ (|p| + |q| + |x_{pq}|) |i| } \sum_{\ell=1}^k \Phi_{i \ell} y_{\ell j}.
    \end{align}
    
    For even \(\D\), rearranging the terms:
    \[
        (-1)^{ (|p| + |q|) |i| } \vphi_0(x_{pi}) \Phi_{pj} = \sum_{\ell=1}^k \Phi_{i \ell} y_{\ell j}
    \]
    The left hand-side is the $(i,j)$-entry of $\Phi Y$.
    Using that $X_{pq}$ is the only non-zero entry of $X$, the right-hand side can be rewritten as $\sum_{\ell=1}^k (-1)^{ (|\ell| + |i|) |i| } \vphi_0(x_{\ell i}) \Phi_{\ell j}$, which is the $(i,j)$-entry of $\vphi_0 (X\stransp) \Phi$.
    Thus, $\Phi Y = \vphi_0 (X\stransp) \Phi$, proving \cref{eq:matrix-vphi-D-even}.
    
    For odd \(\D\), we have:
    \[
        (-1)^{|B| |x_{pi}|} \vphi_0(x_{pi}) \Phi_{pj} = \sum_{\ell=1}^k \Phi_{i \ell} y_{\ell j}.
    \]
    Again, the left hand-side the $(i,j)$-entry of $\Phi Y$.
    By the same trick as in the even case, the right-hand side is $(-1)^{|B| |x_{pq}|} \sum_{\ell=1}^k \vphi_0(x_{\ell i}) \Phi_{\ell j}$, which is the $(i,j)$-entry of $(-1)^{|B| |r|} \vphi_0 (X\transp) \Phi$.
    This establishes \cref{eq:matrix-vphi-D-odd}.
\end{proof}

\phantomsection\label{phsec:vphi-R-simple-D-simple}

Using this matrix representation, it is easy to prove the following (compare with \cref{prop:simple-R-D-super}):

\begin{prop}\label{prop:vphi-R-simple-D-simple}
    The $\vphi$-invariant superideals of $M_k(\D)$ are precisely the sets $M_k(I)$, where $I$ is a $\vphi_0$-invariant superideal of $\D$.  
    In particular, the superalgebra $E(\D, \U, B)$ is $\vphi$-simple \IFF the superalgebra $\D$ is $\vphi_0$-simple. \qed
\end{prop}


    

\subsection{Isomorphisms of graded-simple superalgebras with super-anti-au\-to\-mor\-phism}

We now aim to present necessary and sufficient conditions for $E(\D, \U, B)$ and $E(\D', \U', B')$ to be isomorphic.
In \cref{thm:iso-abstract}, we considered the case without super-anti-automorphisms.
There, we needed isomorphisms and shifts of graded modules. 
We now adapt these concepts to handle sesquilinear forms.

\begin{defi}\label{def:iso-(U,B)}
    Let $\D$ be a graded-division superalgebra.
    Given graded right $\D$-supermodules $\U$ and $\U'$ and sesquilinear forms $B$ and $B'$ on $\U$ and $\U'$, respectively, we define an \emph{isomorphism from $(\U, B)$ to $(\U', B')$} as an isomorphism of graded $\D$-modules $\theta\from \U \to \U'$ such that $B'(\theta(u), \theta(v)) = B(u, v)$ for all $u, v \in \U$.
\end{defi}

\begin{defi}\label{defi:shift-on-B}
    Let $\U$ be a graded right $\D$-module and $B$ be a $\vphi_0$-sesquilinear form on $\U$.
	Given $g\in G^\#$, we define $B^{[g]}\from \U^{[g]}\times \U^{[g]} \to \D$ by $B^{[g]}(u,v) \coloneqq \sign{u}{g} B(u,v)$ for all $u,v \in \U$.
\end{defi}

Note that $\deg B^{[g]} = g^{-2} \deg B$ and, in particular, $|B^{[g]}| = |B|$.
The following can be proved by routine computations using the rule of signs:

\begin{lemma}\label{lemma:B^[g]-does-the-job}
	For every $g\in G^\#$, $B^{[g]}$ is a homogeneous $\vphi_0$-sesquilinear form on $\U^{[g]}$.
	Further, if $B$ is nondegenerate, then so is $B^{[g]}$, and the superadjunction with respect to both is the same super-anti-automorphism $\vphi$ on $\End_\D(\U) = \End_\D(\U^{[g]})$.\qed
\end{lemma}

We should also consider the effect of replacing the graded-division superalgebra via an isomorphism $\D \to \D'$ (see \cref{def:twist}).

\begin{lemma}\label{lemma:twist-on-(U-B)}
    Let $\D$ and $\D'$ be graded-division superalgebras and let $\psi_0 \from \D \to \D'$ be an isomorphism.
    If $\U'$ is a graded right $\D'$-supermodule and $B'$ is a homogeneous $\vphi_0'$-sesquilinear form on it, then $\psi_0^{-1} \circ B'$ is a homogeneous $(\psi_0^{-1} \circ \vphi_0' \circ \psi_0)$-sesquilinear form on the $\D$-supermodule $(\U')^{\psi_0}$ of the same degree as $B'$.
    Further, if $B'$ is nondegenerate, then so is $\psi_0^{-1} \circ B'$, and the superadjunction with respect to both is the same super-anti-automorphism $\vphi'$ on $\End_\D((\U')^{\psi_0}) = \End_{\D'}(\U')$. \qed
\end{lemma}

\begin{thm}\label{thm:iso-abstract-vphi}
    Let $\D$ and $\D'$ be graded-division superalgebras, $\U$ and $\U'$ be nonzero graded right supermodules of finite rank over $\D$ and $\D'$, respectively, and $B$ and $B'$ be nondegenerate sesquilinear forms on $\U$ and $\U'$, respectively.
    If $\psi\from E(\D, \U, B) \to E(\D', \U', B')$ is an isomorphism (see \cref{def:superadjunction}), then there exist $g \in G^\#$, a homogeneous element $0 \neq d \in \D$, an isomorphism $\psi_0\from \D \to \D'$, and an isomorphism
    \begin{equation}\label{eq:iso-B-implies-vphi}
        \psi_1 \from (\U^{[g]}, dB^{[g]}) \to ((\U')^{\psi_0}, \psi_0^{-1} \circ B')
    \end{equation}
    such that
    \begin{equation}\label{eq:iso-super-anti-auto}
        \forall r \in R, \quad \psi(r) = \psi_1 \circ r \circ \psi_1^{-1}.
    \end{equation}
    Conversely, for any $g$, $d$, $\psi_0$, and $\psi_1$ as above, the map defined by \cref{eq:iso-super-anti-auto} is an isomorphism $\psi\from E(\D, \U, B) \to E(\D', \U', B')$.
\end{thm}

\begin{proof}
	Given an isomorphism of graded superalgebras $\psi\from R \to R'$, we define
	$\tilde \vphi \coloneqq \psi\inv \circ \vphi' \circ \psi$.
	Then $\psi$ is an isomorphism $(R, \vphi) \to (R', \vphi')$ \IFF $\vphi = \tilde\vphi$.
	Since $\psi$ is an isomorphism of $G^\#$-graded algebras, we can apply \cref{thm:iso-abstract} to conclude that there are $g\in G^\#$, an isomorphism of graded superalgebras $\psi_0\from \D \to \D'$, and an isomorphism of graded modules $\psi_1\from \U^{[g]} \to (\U')^{\psi_0}$ such that $\psi(r) = \psi_1 \circ r \circ \psi_1\inv$, for all $r\in R$.

	Consider $\vphi_0'' \coloneqq \psi_0\inv \circ \vphi_0' \circ \psi_0$ and $B'' \coloneqq \psi_0\inv \circ B'$.
	Then define $\widetilde B \from \U^{[g]} \times \U^{[g]} \to \D$ by
	$
		\widetilde B(u, v) \coloneqq B'' \big( \psi_1(u), \psi_1(v) \big)
	$
	for all $u, v \in \U^{[g]}$.
	It is easy to check that $\widetilde B$ is $\vphi_0''$-sesquilinear and that $\tilde\vphi$ is the superadjunction with respect to $\widetilde B$.
	Hence, applying \cref{thm:vphi-iff-vphi0-and-B} for $\U^{[g]}$ and Lemma \ref{lemma:B^[g]-does-the-job}, we conclude that $\vphi = \tilde\vphi$ \IFF there is a homogeneous $0 \neq d \in \D$ such that $\widetilde B = dB^{[g]}$.
	The result follows.
\end{proof}

\phantomsection\label{subsubsec:isomorphisms-with-actions}

We can interpret \cref{thm:iso-abstract-vphi} in terms of group actions.
For that, fix a graded-division superalgebra $\D$.
We will define three (left) group actions on the class of pairs $(\U, B)$, where $\U \neq 0$ is a graded $\D$-supermodule and $B$ is a nondegenerate sesquilinear form.
Recall that, by \cref{lemma:B-determines-vphi_0}, $B$ is $\vphi_0$-sesquilinear for a unique super-anti-automorphism $\vphi_0$ of $\D$.

Let $\D^\times_{\mathrm{gr}} \coloneqq \big( \bigcup_{g \in G^\#} \D_g \big)\backslash \{ 0 \}$, the group of nonzero homogeneous elements of $\D$.
Given $d\in \D^\times_{\mathrm{gr}}$, we define
\begin{equation}\label{eq:Dx_gr-action}
	d\cdot (\U, B) \coloneqq (\U, dB).
\end{equation}
Note that $dB$ is $(\operatorname{sInt}_d \circ \vphi_0)$-sesquilinear by \cref{lemma:B-determines-vphi_0}.
Let $A \coloneqq \Aut (\D)$, the group of automorphisms of $\D$ as a graded superalgebra.
Given $\tau \in A$, we define
\begin{equation}\label{eq:Aut(D)-action}
	\tau \cdot (\U, B) \coloneqq (\U^{\tau\inv}, \tau \circ B).
\end{equation}
Note that $\tau \circ B$ is $(\tau \circ \vphi_0 \circ \tau\inv)$-sesquilinear by \cref{lemma:twist-on-(U-B)}.
Finally, consider the group $G^\#$.
Given $g \in G^\#$, we define
\begin{equation}\label{eq:G-action}
	g \cdot (\U, B) \coloneqq (\U^{[g]}, B^{[g]}).
\end{equation}
Note that $B^{[g]}$ is $\vphi_0$-sesquilinear by \cref{lemma:B^[g]-does-the-job}.

\begin{lemma}\label{lemma:action-on-(U,B)}
	The three actions defined above give rise to a $(\D^\times_{\mathrm{gr}} \rtimes A) \times G^\#$-action, where $A$ acts on $\D^\times_{\mathrm{gr}}$ by evaluation.
\end{lemma}

\begin{proof}
	Let $d\in \D^\times_{\mathrm{gr}}$, $\tau \in A$, $g \in G^\#$ and $u, v \in \U$.
	First, note that the action of $d$ does not change $\U$, so we only have to consider its effect on $B$.
	Since
	\begin{align*}
		(\tau \circ dB)(u,v) & = \tau \big( d B(u,v) \big) = \tau (d) \tau \big( B(u,v) \big) = \big( \tau(d) (\tau \circ B) \big) (u,v),
	\end{align*}
	the $\D^\times_{\mathrm{gr}}$-action combined with the $A$-action gives us a $(\D^\times_{\mathrm{gr}} \rtimes A)$-action.
	The $G^\#$-action commutes with the $\D^\times_{\mathrm{gr}}$-action since
	\begin{align*}
		(dB)^{[g]} (u, v) = \sign{g}{u} dB(u,v) = d B^{[g]}(u,v).
	\end{align*}
	Finally, the $G^\#$-action also commutes with the $A$-action since $(\U^{\tau\inv})^{[g]} = (\U^{[g]})^{\tau\inv}$ and
	\[
		(\tau \circ B^{[g]}) (u,v)
		  = \tau \big( \sign{g}{u} B(u,v) \big)
		  = \sign{g}{u} (\tau \circ B) (u,v)
        = (\tau \circ B)^{[g]} (u,v).
	\]
\end{proof}

\begin{cor}\label{cor:iso-with-actions}
	Under the assumptions of Theorem \ref{thm:iso-abstract-vphi}, if $\D \not \iso \D'$, then $(R, \vphi) \not \iso (R', \vphi')$.
	Otherwise, let $\psi_0\from \D \to \D'$ be any isomorphism.
	Then $(R, \vphi) \iso (R', \vphi')$ \IFF $\big( (\U')^{\psi_0}, \psi_0\inv \circ B' \big)$ is isomorphic to an object in the $(\D^\times_{\mathrm{gr}} \rtimes A) \times G^\#$-orbit of $(\U, B)$. \qed
\end{cor}

\subsection{Superinvolutions and super-Hermitian forms}\label{subsec:superinv-sesquilinear-forms}

In this subsection, we investigate conditions on a $\vphi_0$-sesquilinear form $B$ to ensure that the superadjunction with respect to $B$ is a superinvolution.

\begin{defi}\label{def:barB}
	Given a super-anti-automorphism $\vphi_0$ on $\D$ and a $\vphi_0$-sesqui\-li\-near form $B$ on $\U$, we define $\overline {B}\from \U\times \U \to \D$ by $\overline {B} (u,v) \coloneqq \sign{u}{v} \vphi_0\inv (B(v, u))$ for all $u, v \in \U$.
\end{defi}

\begin{prop}\label{prop:barB-determines-vphi-inv}
	The map $\overline {B}$ is a $\vphi_0\inv$-sesquilinear form of the same degree and parity as $B$.
	Further, if $B$ is nondegenerate and $\vphi$ is the superadjunction with respect to $B$, then $\overline{B}$ is nondegenerate and $\vphi\inv$ is the superadjunction with respect to $\barr B$. \qed
\end{prop}

\begin{proof}
    Checking that $\barr B$ is $\vphi_0\inv$-sesquilinear is a matter of routine calculations using the sign rule.
    The same is true for checking that
    \begin{equation}\label{eq:superadjunction-barr-B}
		\forall r\in R\even \cup R\odd,\,\forall u, v \in \U\even \cup \U\odd,  \quad \barr B(ru,v) = \sign{r}{u} \barr B(u,\vphi\inv(r)v)\,,
	\end{equation}
    if we assume that $\vphi$ is the superadjunction with respect to $B$.

    Finally, \cref{eq:superadjunction-barr-B} together with $B$ being nondegenerate implies that $\overline{B}$ is nondegenerate.
	To see that, assume there is $0 \neq u \in \rad \overline{B}$.
	For for every $r\in R\even \cup R\odd$ and $v\in \U\even \cup \U\odd$, we would have that $\overline{B}(u, \vphi\inv (r) v) = 0$, hence $\overline{B}(ru, v) = 0$.
	Since $r \in R\even \cup R\odd$ and $v\in \U\even \cup \U\odd$ were arbitrary, this implies $\overline{B} (Ru, \U) = 0$.
	But $\U$ is simple as a graded $R$-supermodule, so we get $\overline{B}(\U, \U) = 0$ and, hence, $B (\U, \U) = 0$, a contradiction.
\end{proof}

\begin{lemma}\label{lemma:bar-dB}
	Let $d$ be a nonzero $G^\#$-homogeneous element of $\D$ and consider $\vphi_0' \coloneqq \operatorname{sInt}_d\circ\, \vphi_0$ and $B' \coloneqq d B$.
	Then $\overline {B'} = (-1)^{|d|} \vphi_0\inv (d) \overline B$. 
\end{lemma}

\begin{proof}
	Note that $(\vphi_0')\inv = \vphi_0\inv \circ \operatorname{sInt}_{d\inv}$.
	Hence, for all $u, v \in \U\even \cup \U\odd$,
	\begin{align*}
		\overline {B'} (u,v)
          & = \sign{u}{v} (\vphi_0\inv \circ \operatorname{sInt}_{d\inv})  (d B(v, u) )                    \\
		   & = \sign{u}{v} \vphi_0\inv \big( (-1)^{|d| (|d| + |B| + |u| + |v|)}\,  d\inv d B (v, u) d \big) \\
          &= \sign{u}{v} (-1)^{|d|}\, \vphi_0\inv (d) \vphi_0\inv(B(v, u))
          =  (-1)^{|d|}\,\vphi_0\inv (d) \overline {B} (u, v).
	\end{align*}
\end{proof}

\begin{defi}
    We say that a sesquilinear form $B$ on $\U$ is \emph{super-Hermitian} if $\barr B = B$ and \emph{super-skew-Hermitian} if $\barr B = - B$.
\end{defi}

The following is a graded version of \cite[Theorem 7]{racine} (applied to Artinian superalgebras).
The proof is essentially the same as that of \cite[Theorem 3.1]{ELDUQUE202261}:

\begin{thm}\label{thm:vphi-involution-iff-delta-pm-1}
	Let $\D$ be a graded-division superalgebra, $\U$ be a nonzero $G$-graded right module of finite rank over $\D$, and $\vphi$ be a degree-preserving super-anti-automorphism on $R \coloneqq \End_\D (\U)$.
    Then $\vphi$ is a superinvolution \IFF there is a super-Hermitian or super-skew-Hermitian form $B$ such that $\vphi$ is the superadjunction with respect to $B$. 
	In this case, $B$ is $\vphi_0$-sesquilinear where $\vphi_0\from \D \to \D$ is a superinvolution. 
    Moreover, for a nonzero homogeneous $d\in\D$, the sesquilinear form $B'\coloneqq dB$ satisfies $\overline{B'}=(-1)^{|d|}\delta\vphi_0(d)d^{-1}B'$.
\end{thm}

\begin{proof}
    By \cref{prop:barB-determines-vphi-inv}, $\vphi = \vphi\inv$ \IFF $B$ and $\overline B$ have the same superadjunction.
    By \cref{lemma:B-determines-vphi_0}, this is the case \IFF there is a $G^\#$-homogeneous element $0 \neq \delta \in \D$ such that $\overline {B} = \delta B$. 
    Hence, the ``if'' part is straightforward. 
    
    For the ``only if'' part, fix a sesquilinear for $B$ for which $\overline {B} = \delta B$, where $\delta\in\pmone$. 
	Since $B$ and $\overline {B}$ have the same $G^\#$-degree, it follows that $\delta \in \D_e$. 
	By \cref{lemma:bar-dB}, we have that
	\[
		B = \overline {\overline B} = \overline {\delta B} = (-1)^{|\delta|} \vphi_0\inv(\delta) \overline B= \vphi_0\inv(\delta) \overline B = \vphi_0\inv(\delta) \delta B,
	\]
	so $\vphi_0\inv(\delta) \delta = 1$. 

    Let $\mathbb{L}$ be the subfield of the division-algebra $\D_e$ generated by $\FF$ and $\delta$, and denote by $\sigma\from \mathbb L \to \mathbb L$ the restriction of $\vphi_0\inv$ to $\mathbb{L}$.
    Since $\sigma$ fixes $\FF$ and $\sigma(\delta) = \vphi_0\inv(\delta) = \delta\inv$, it follows that $\sigma$ is an involutive automorphism of $\mathbb{L}$ over $\FF$.

    If $\sigma = \id$, then $\delta\inv = \sigma(\delta) = \delta$ and, hence, $\delta \in \pmone$.
    If $\sigma \neq \id$, consider the subfield $\mathbb L_0 \coloneqq \{ x\in \mathbb L \mid \sigma(x) = x \} \subseteq \mathbb L$. 
    Then $\mathbb L/\mathbb L_0$ is a quadratic Galois extension and $\langle \sigma \rangle$ is its Galois group.
    By Hilbert's Theorem 90, there is $\lambda \in \mathbb L \subseteq \D_e$ such that $\delta = \lambda \sigma(\lambda)\inv$.
    If we replace $B$ by $B'\coloneqq\lambda B$, using \cref{lemma:bar-dB}, we have that
    \begin{align}
        \overline {B'} &= (-1)^{|\lambda|}\vphi_0\inv(\lambda) \barr B = \sigma(\lambda) \barr B\\
        &= \lambda \delta\inv \barr B = \lambda\delta\inv \delta B\\
        &= \lambda \lambda\inv B' = B'\,,
    \end{align}
    concluding the ``only if'' direction.
    Also, if $\overline{B}=\delta B$, then we have $\vphi_0\inv = \operatorname{sInt}_\delta \circ \,\vphi_0 = \vphi_0$ by \cref{lemma:B-determines-vphi_0}.
    
    The ``moreover'' part follows from \cref{lemma:bar-dB}. 
\end{proof}

Recall from \cref{lemma:B^[g]-does-the-job} that replacing $(\U, B)$ by $(\U^{[g]}, B^{[g]})$ does not change $\End_\D(\U)$ and $\vphi$.

\begin{lemma}\label{lemma:B-g-delta}
    If $\barr B = \delta B$, then $\overline{B^{[g]}} = (-1)^{|g|} \delta B^{[g]}$.
\end{lemma}

\begin{proof}
    Let $u, v\ \in \U\even \cup \U\odd$.
    We will write $u^{[g]}$ and $v^{[g]}$ for $u$ and $v$ when they are considered as elements of $\U^{[g]}$. Then
\begin{align*}
	\overline{B^{[g]}}(u^{[g]},v^{[g]}) & = \sign{u^{[g]}}{v^{[g]}} \vphi_0\inv \big( B^{[g]} (v^{[g]}, u^{[g]}) \big)                                              \\
	                                    & = (-1)^{(|g| + |u|) (|g| + |v|)} \vphi_0\inv \big( \sign{g}{v} B(v, u) \big)                                              \\
	                                    & = (-1)^{|g| + |g||u| + |u||v|} \vphi_0\inv \big(B(v, u) \big) 
                                          = (-1)^{|g| + |g||u|} \overline B(u,v) \\
	                                    & = (-1)^{|g|} \sign{g}{u} \delta B(u,v) = (-1)^{|g|} \delta B^{[g]}(u^{[g]},v^{[g]})\,.
\end{align*}
\end{proof}

\begin{remark}\label{rmk:only-super-hermitian}
    Note that reversing parity on $\U$ does not affect the graded superalgebra with superinvolution $E(\D, \U, B)$ and changes $\delta$ to $-\delta$, so we can always make $B$ super-Hermitian.
\end{remark}

Combining with \cref{prop:only-SxSsop-is-simple,thm:End-over-D}, we have:

\begin{cor}\label{cor:SxSsop-with-dcc}
    Let $(R, \vphi)$ be a graded superalgebra with superinvolution and suppose $R$ satisfies the \dcc on graded left superideals. 
    Then $(R, \vphi)$ is graded-superinvolution-simple \IFF there exists a graded-division superalgebra $\D$ and graded right $\D$-supermodule $\U$ of finite rank such that $(R, \vphi)$ is isomorphic either to $\Eex (\D, \U)$ (\cref{defi:superdual-exchange}) or to $E(\D, \U, B)$ for some nondegenerate super-Hermitian form $B$ on $\U$ (\cref{def:superadjunction}). \qed
\end{cor}

\subsection{Parametrization and classification}\label{subsec:param-R-vphi}

In this section, we will parametrize the objects in \cref{cor:SxSsop-with-dcc} over an algebraically closed field $\FF$.

\subsubsection{Parametrization of \texorpdfstring{$(\D, \vphi_0)$}{(D, phi0)}}

Recall that the isomorphism class of a finite-dimensional graded-division superalgebra $\D$ is determined by a pair $(T, \tilde\beta)$ where $T \coloneqq \supp \D \subseteq G^\#$ is a finite subgroup and $\tilde\beta\from T\times T \to \FF^\times$ is a skew-symmetric bicharacter.
Also, recall the parity map $p\from T\to \ZZ_2$ and the alternating bicharacter $\beta\from T\times T \to \FF^\times$ associated to $\tilde\beta$.

It will be convenient to use the following notation from group cohomology:

\begin{defi}\label{def:coboundary}
	Let $H$ and $K$ be groups, and let $f\from H \to K$ be any map.
	We define $\mathrm{d} f\from H\times H \to K$ as the map given by $(\mathrm{d} f)\, (a,b) = f(ab)f(a)\inv f(b)\inv$ for all $a,b \in H$.
\end{defi}

Since $\FF$ is algebraically closed and $\D$ is finite-dimensional, each component $\D_t$ of $\D$ is one-dimensional, hence an invertible degree-preserving map $\vphi_0\from \D \to \D$ is completely determined by a map $\eta\from T \to \FF^\times$ such that $\vphi_0(X_t) = \eta(t)X_t$ for all $X_t \in \D_t$.

\begin{prop}\label{prop:superpolarization}
    Let $\eta\from T \to \FF^\times$ be any map and let $\vphi_0\from \D \to \D$ be the linear map determined by $\vphi_0(X_t) = \eta(t) X_t$ for all $t\in T$ and $X_t\in \D_t$.
	Then $\vphi_0$ is a super-anti-automorphism \IFF 
	\begin{equation}\label{eq:superpolarization}
		\forall a,b\in T, \quad \eta(ab) = \tilde\beta(a,b) \eta(a) \eta(b)\,
	\end{equation}
    or, in other words, $\mathrm{d} f = \tilde \beta$.
	Moreover, $\D$ admits a super-anti-automorphism \IFF $\tilde\beta$ (or, equivalently, $\beta$) only takes values $\pm 1$.
\end{prop}

\begin{proof}
	For all $a,b \in T$, let $X_a \in \D_a$ and $X_b\in \D_b$. 
	Then:
	\begin{alignat*}{2}
		     &  & \vphi_0(X_a X_b)             & = (-1)^{p(a) p(b)} \vphi_0(X_b) \vphi_0(X_a)          \\
		\iff &  & \,\, \eta(ab)X_a X_b         & = (-1)^{p(a) p(b)} \eta(a) \eta(b) X_b X_a            \\
		\iff &  & \, \eta(ab)X_a X_b           & = \tilde\beta(b,a) \eta(a) \eta(b) X_a X_b \\
		\iff &  & \eta(ab)                     & = \tilde\beta(b,a) \eta(a) \eta(b)
	\end{alignat*}
	If $a$ and $b$ are switched, since $T$ is abelian, we get $\eta(ab) = \tilde\beta(a,b) \eta(a) \eta(b)$, as desired. 
	Also, it follows that $\tilde\beta(b,a) = \tilde\beta(a,b)$. 
	Using that $\tilde\beta$ is skew-symmetric, \ie,  $\tilde\beta(b, a) = \tilde\beta (a, b)\inv$, we have that $\tilde\beta(a,b)^2 = 1$ and, hence, $\tilde\beta$ only takes values $\pm 1$, proving one direction of the ``moreover'' part. 
	The converse follows from the fact that the isomorphism class of $\D\sop$ is determined by $(T, \tilde\beta\inv)$, so if $\beta$ takes only values in $\{ \pm 1 \}$, there must be an isomorphism from $\D$ to $\D\sop$, which can be seen as a super-anti-automorphism of $\D$.
\end{proof}

\begin{cor}\label{cor:super-anti-auto-squares-in-radical}
    A graded-division superalgebra $\D$ admits a super-anti-auto\-mor\-phism \IFF 
    $t^2\in \rad\tilde\beta$, for all $t \in T$. \qed
\end{cor}

\begin{cor}\label{cor:eta-t-square}
    Suppose $\eta$ determines a superinvolution on $\D$, \ie, $\eta$ takes values in $\pmone$. 
    For every element $t\in T$, we have $\eta(t^2) = (-1)^{p(t)}$ for all $t\in T$. 
    In particular, every element in $T^-$ has order at least $4$. \qed
\end{cor}

\phantomsection\label{def:Parameters-T-beta-eta}

It follows that we can parametrize graded-division superalgebras with super-anti-automorphism by pairs $(T, \eta)$, where $\mathrm{d} \eta$ is a skew-symmetric bicharacter.
But it will be convenient to keep $\tilde\beta$ as a parameter, so we say that $(\D,\vphi_0)$ is \emph{associated to the triple $(T, \tilde\beta, \eta)$}. 

Note that, by \cref{prop:parametrization-T-beta,prop:superpolarization}, for any finite abelian group $T$, skew-symmetric bicharacter $\tilde\beta\from T\times T \to \FF^\times$ and map $\eta\from T \to \FF^\times$ such that $\mathrm{d}\eta = \tilde\beta$, there is a graded-division superalgebra with super-anti-automorphism associated to $(T, \tilde\beta, \eta)$. 

\phantomsection\label{phsec:quadratic-maps}

\begin{remark}
    In the non-super setting, the map $\eta$ is a multiplicative quadratic form (see, \eg, \cite[Section 2.4]{livromicha}).
    This property fails in the super case, as $\eta(t^{-1}) \neq \eta(t)$ in general. 
    Nevertheless, specific instances where $\eta$ is a quadratic form remain of particular interest and are studied in \cite[Section 4.4]{caios_thesis}.
    Key examples are the standard realizations $\D$ of type $M$ or $Q$, for which the queer supertranspose (\cref{def:supertranspose}) corresponds to a quadratic form.
\end{remark}

\subsubsection{Parametrization of \texorpdfstring{$(\U, B)$}{(U, B)}}\label{ssec:parameters-(U-B)}

Let $(\D, \vphi_0)$ be a graded-division superalgebra with su\-per-an\-ti-au\-to\-mor\-phism, let $\U$ be a graded right $\D$-module.  
Recall that the isomorphism class of a $G$-graded supermodule $\U$ is determined by a map $\kappa\from G^\#/T \to \ZZ_{\geq 0}$ with finite support.
Explicitly, $\kappa (gT) = \dim_\D \U_{gT}$, where $\U_{gT}$ is the isotypic component associated to the coset $gT$.

\phantomsection\label{phsec:paired-by-B}

Let $g_0 \in G^\#$ be the degree of $B$.  
For $x, y \in G^\#/T$, if $B(\U_{x}, \U_{y}) \neq 0$, then $g_0 x y = T$.  
In this case, we say the isotypic components $\U_{x}$ and $\U_{y}$ are \emph{paired by $B$}.  
Define $\V_x \coloneqq \U_x + \U_y$ (so $\V_x = \U_x$ if $x = y$ and $\V_x = \U_x \oplus \U_y$ if $x \neq y$), and let $B_x$ denote the restriction of $B$ to $\V_x$.  
Then $B$ is nondegenerate \IFF $B_x$ is nondegenerate for every $x \in G^\#/T$.  
When this holds, $\U_x$ and $\U_y$ are dual to each other, and consequently $\kappa(x) = \kappa(y)$.

\phantomsection\label{phsec:reduce-B-to-FF-bilinear}

We will reduce $B$ to a collection of $\FF$-bilinear forms.
Fix elements $0 \neq X_t \in \D_t$ for each $t\in T$,  and let $\xi\from G^\#/T \to G^\#$ be a set-theoretic section of the natural homomorphism $G^\# \to G^\#/T$.
Note that $\U_{\xi(x)} \tensor \D \iso \U_x$ via the map $u \tensor d \mapsto ud$, so an $\FF$-basis of $\U_{\xi(x)}$ is a graded $\D$-basis for $\U_x$.
In view of Convention~\ref{conv:pick-even-basis}, if $\D$ is odd, we require $\xi$ to take values in $G = G\times \{ \bar 0 \}$.

For all $x \in G^\#/T$, set $y \coloneqq g_0\inv x\inv \in G^\#/T$ and $t \coloneqq g_0 \xi(x) \xi(y) \in T$, and define the bilinear form $\tilde{B}_x\from V_x \times V_x \to \FF$ by
\begin{equation}\label{eq:B_x-tilde}
	\forall u,v \in V_x, \quad \tilde{B}_x (u,v) \coloneqq X_{t}\inv B_x (u,v)\,.
\end{equation}

\begin{lemma}\label{lemma:B_x-nondeg}
	The sesquilinear form $B_x$ is nondegenerate \IFF the bilinear form $\tilde{B}_x$ is nondegenerate. \qed
\end{lemma}

Let $(T, \tilde\beta, \eta)$ be the parameters of $(\D, \vphi_0)$.
We are interested in the case where the superadjunction is involutive.
By \cref{thm:vphi-involution-iff-delta-pm-1}, we can restrict our attention to the case where $\vphi_0$ is involutive.
Hence, we assume that $\eta$ takes values in $\pmone$.

\begin{lemma}\label{lemma:B_x-delta}
	Let $\delta \in \pmone$.
	Then $\overline{B_x} = \delta B_x$ \IFF
	\[
		\forall u, v \in V_x\even \cup V_x\odd, \quad \tilde{B}_x (v, u) = (-1)^{|u| |v|} \eta(t) \delta \tilde{B}_x (u, v)\,.
	\]
\end{lemma}

\begin{proof}
	Let $u,v \in \V_x$.
	By definition of $\overline{B_x}$, we have:
	\begin{alignat*}{3}
		\overline{B_x} (u, v) & = \sign{u}{v} \vphi_0\inv( B_x(v, u) )
		                      &                                                  & = \sign{u}{v} \vphi_0\inv( X_t \tilde{B}_x(v, u) ) \\
		                      & = \sign{u}{v} \tilde{B}_x(v, u) \vphi_0\inv(X_t)
		                      &                                                  & = \sign{u}{v} \tilde{B}_x(v, u) \eta(t)\inv X_t,
	\end{alignat*}
	where we have used the fact that $\tilde B_x (v, u) \in \FF$.
    The result follows.
\end{proof}

Recall the identification $M_k (\D) = \M_k(\FF) \tensor \D$.
In the next two propositions we consider two cases: an isotypic component paired to itself, and two different components paired to each other.

\begin{prop}\label{prop:self-dual-components}
	Let $\delta \in \pmone$.
	Suppose $g_0 x^2 = T$,  and  set $t \coloneqq g_0\xi(x)^2 \in T$.
	Then
	\begin{equation}\label{eq:mu_x}
		\mu_{x} \coloneqq (-1)^{|\xi(x)|} \eta(t) \delta \in \pmone
	\end{equation}
	does not depend on the choice of the section $\xi\from G^\#/T \to G^\#$.
	Moreover, the restriction $B_x$ of $B$ to $\U_x$ is nondegenerate and satisfies $\overline{B_x} = \delta B_x$ \IFF there exists a $\D$-basis of $\U_{x}$ consisting only of elements of degree $\xi(x)$ such that the matrix representing $B_x$ is given by
	\begin{enumerate}
		\item $I_{\kappa(x)} \tensor X_t$ if $\mu_x = +1$;
		\item $J_{\kappa(x)} \tensor X_t$ if $\mu_{x} = -1$, where $\kappa (x)$ is even and $J_{\kappa(x)} \coloneqq \begin{pmatrix}
				      0                & I_{\kappa(x)/2} \\
				      -I_{\kappa(x)/2} & 0
			      \end{pmatrix}$.
	\end{enumerate}
\end{prop}

\begin{proof}
    Let $g \coloneqq \xi(x)$. 
    If $\xi'\from G^\#/T \to G^\#$ is another section, then there exists $s \in T$ such that $\xi'(x) = gs$.
    Hence:
    \begin{align*}
        (-1)^{|\xi'(x)|} \eta(g_0 \xi'(x)^2) 
            &= (-1)^{|gs|} \eta(g_0 g^2 s^2) \\
            &= (-1)^{|g| + |s|} \sign{g_0 g^2}{s^2} \beta(g_0 g^2, s^2)\eta(g_0 g^2)\eta(s^2) \\
            &= (-1)^{|g|} \eta(g_0 g^2).
    \end{align*}
    For the ``moreover'' part, \cref{lemma:B_x-nondeg,lemma:B_x-delta} imply that $B_x$ is nondegenerate and $\overline{B_x} = \delta B_x$ \IFF $\tilde{B}_x$ is nondegenerate and $B_x (u,v) = \mu_x B_x(v, u)$, for all $u,v \in V_x = \U_{\xi(x)}$.
    Then the result follows from the well-known classification of (skew-)symmetric bilinear forms over an algebraically closed field of characteristic different from $2$ and the fact that an $\FF$-basis for $\U_{\xi(x)}$ is a $\D$-basis for $\U_x$.
\end{proof}

Note that although $\mu_{x}$ is independent of the choice of $\xi$, the element $t = g_0\xi(x)^2$ may still depend on $\xi$.

\begin{prop}\label{prop:pair-of-dual-components}
	Let $\delta \in \pmone$.
	Suppose $g_0 x y = T$ for $x\neq y$ and set $t \coloneqq g_0\xi(x)\xi(y) \in T$.
	Then the restriction $B_x$ of $B$ to $\U_x \oplus \U_y$ is nondegenerate and satisfies $\overline{B_x} = \delta B_x$ \IFF there is a $\D$-basis of $\U_x$ with all elements having degree $\xi(x)$ and a $\D$-basis of $\U_y$ with all elements having degree $\xi(y)$ such that the matrix representing $B_x$ is
	\[
		\begin{pmatrix}
			0                                                  & I_{\kappa(x)} \\
			\sign{\xi(x)}{\xi(y)} \eta(t) \delta I_{\kappa(x)} & 0
		\end{pmatrix} \tensor X_t.
	\]
\end{prop}

\begin{proof}
    Assume $B_x$ is nondegenerate.
    By \cref{lemma:B_x-nondeg}, the bilinear form $\tilde{B}_x$ on $V_x = \U_{\xi(x)} \oplus \U_{\xi(y)}^*$ is nondegenerate.
    Consequently, the map $\U_{\xi(x)} \to \U_{\xi(y)}$ given by $u \mapsto \tilde{B}_x(u, \cdot)$ is a vector space isomorphism.
    Fix a basis $\{u_1, \ldots, u_{\kappa(x)}\}$ for $\U_{\xi(x)}$ and let $\{v_1, \ldots, v_{\kappa(x)}\}$ be its dual basis in $\U_{\xi(y)}$, \ie, $\tilde{B}_x(u_i, v_j) = \delta_{ij}$.
    If we further assume $\overline{B_x} = \delta B_x$, \cref{lemma:B_x-delta} yields
    \(
        \tilde{B}_x(v_i, u_j) = \sign{\xi(x)}{\xi(y)} \eta(t) \delta \delta_{ij}.
    \)
    This establishes the ``only if'' direction.
    The converse is straightforward.
\end{proof}

\begin{defi}\label{def:parameter-of-(U-B)}
	Let $\U\neq 0$ be a graded $\D$-module of finite rank and let $B$ be a nondegenerate $\vphi_0$-sesquilinear form such that $\overline{B} = \delta B$ for some $\delta \in \pmone$.
    The \emph{inertia of $(\U, B)$} is the quadruple $(\eta, \kappa, g_0, \delta)$, where $\eta$ defines $\vphi_0$ by $\vphi_0(X_t) = \eta(t)X_t$, $g_0 = \deg B \in G^\#$ and $\kappa(x) = \dim_\D \U_x$ for all $x \in G^\#/T$.
\end{defi}

The following definition summarizes the conditions that the quadruple $(\eta, \kappa, g_0, \delta)$ must satisfy.

\begin{defi}\label{defi:X(D)}
	Given $\eta\from T \to \pmone$, $\kappa\from G^\#/T \to \ZZ_{\geq 0}$, $g_0 \in G^\#$ and $\delta \in \pmone$, we say $(\eta, \kappa, g_0, \delta)$ is \emph{admissible} if it satisfies:
	\begin{enumerate}
		\item $\mathrm{d}\eta = \tilde\beta$; \label{item:eta-is-eta}
		\item $\kappa$ has finite support; \label{item:kappa-finite-support}
		\item $\kappa(x) = \kappa(g_0\inv x\inv)$ for all $x \in G^\#/T$; \label{item:kappa-duality}
		\item for any $x\in G^\#/T$, if $g_0 x^2 = T$ and $\mu_x = -1$, then $\kappa (x)$ is even (where \\$\mu_x\coloneqq (-1)^{|g|}\eta(g_0 g^2)\delta$ for $g\in x$, see Proposition \ref{prop:self-dual-components}). \label{item:kappa-parity}
	\end{enumerate}
	The set of all admissible quadruples is denoted by $\mathbf{I}(\D)$ or $\mathbf{I}(T, \tilde\beta)$.
\end{defi}

It is clear that for any quadruple $(\eta, \kappa, g_0, \delta) \in \mathbf{I} (\D)$, we can construct a pair $(\U, B)$ such that $(\eta, \kappa, g_0, \delta)$ is its inertia.
We will give explicit constructions in \cref{subsubsec:construction-U-B}.

\begin{thm}\label{thm:iso-(U,B)}
	Suppose $\FF$ is an algebraically closed field and $\Char \FF \neq 2$. 
	Let $\D$ be a finite-dimensional graded-division superalgebra and let $\vphi_0$ be a degree-preserving superinvolution on $\D$. 
	The assignment of inertia to a pair $(\U, B)$ as in Definition \ref{def:parameter-of-(U-B)} gives a bijection between the isomorphism classes of these pairs and the set $\mathbf{I} (\D)$. 
\end{thm}

\begin{proof}
	Suppose that there is an isomorphism $\psi\from (\U, B) \to (\U', B')$.
    Since $\psi$ is an isomorphism of graded $\D$-modules, $\U$ and $\U'$ correspond to the same map $\kappa\from G^\#/T \to \ZZ_{\geq 0}$.
	Also, the compatibility condition $B'(\psi(u), \psi(v)) = B(u,v)$ implies $\deg B' = \deg B = g_0$.
	Moreover, if $\overline{B} = \delta B$, for homogeneous $u,v\in\U$:
	\begin{align*}
		\overline{B'} \big(\psi(u), \psi(v) \big) & = \sign{\psi(u)}{\psi(v)} \vphi_0\inv \Big( B'\big( \psi(v) , \psi(u) \big) \Big) \\
		                                          & = \sign{u}{v} \vphi_0\inv \big( B(v, u) \big)
		= \overline{B} (u, v)                                                                                                         \\
		                                          & = \delta B(u, v) = \delta B' \big( \psi(u), \psi(v) \big),
	\end{align*}
	By bijectivity of $\psi$, we conclude $\overline{B'} = \delta B'$.

	Conversely, assume $(\U,B)$ and $(\U',B')$ share the same inertia $(\eta,\kappa,g_0,\delta)$. 
	To show that $(\U, B)$ and $(\U, B')$ are isomorphic, it suffices to find homogeneous $\D$-bases $\{u_1, \ldots, u_k\}$ of $\U$ and $\{u_1', \ldots, u_k'\}$ of $\U'$ such that $\deg u_i = \deg u_i'$, $1 \leq i \leq k$, and $B$ and $B'$ are represented by the same matrix.
	The existence of such bases follows from $\dim_\D (\U_x) = \dim_\D (\U_x') = \kappa(x)$, for all $x \in G^\#/T$, and \cref{prop:self-dual-components,prop:pair-of-dual-components}.
\end{proof}

\subsubsection{Parametrization of \texorpdfstring{$(R, \vphi)$}{(R,phi)}}\label{subsec:param-(R-phi)}

Combining the parameters for $(\D, \vphi_0)$ and $(\U, B)$, we have:

\begin{defi}\label{def:E(D-U-B)}
	Let $\D$ be a finite-dimensional graded-division superalgebra over an algebraically closed field $\FF$, $\Char \FF \neq 2$, let $\U\neq 0$ be a graded right $\D$-module of finite rank, and let $B$ be a nondegenerate $\vphi_0$-sesquilinear form on $\U$ such that $\overline{B} = \delta B$ for some $\delta \in \pmone$.
	If $\D$ is associated to $(T, \tilde\beta)$ and $(\U, B)$ has inertia $(\eta, \kappa, g_0, \delta) \in \mathbf{I}(\D)=\mathbf{I}(T,\tilde{\beta})$, then we say that $(T, \tilde\beta, \eta, \kappa, g_0, \delta)$ are the parameters of the triple $(\D, \U, B)$.
\end{defi}

In \cref{subsubsec:isomorphisms-with-actions}, we defined actions of the groups $\D^\times_{\mathrm{gr}} \coloneqq \big( \bigcup_{g \in G^\#} \D_g \big)\backslash \{ 0 \}$, $A \coloneqq \Aut (\D)$ and $G^\#$ on the set of pairs $(\U, B)$, which appear in \cref{cor:iso-with-actions}. 
We now translate these actions to the set $\mathbf{I}(T, \tilde\beta)$.  

\begin{lemma}\label{lemma:twist-same-inertia}
    Let $(\D', \U', B')$ be a triple as in \cref{def:E(D-U-B)}, let $\D$ be a graded-division superalgebra, and let $\psi_0\from \D \to \D'$ be an isomorphism.
	The inertias of $\bigl(\U', B'\bigr)$ and $\bigl( (\U')^{\psi_0},\, \psi_0^{-1} \circ B' \bigr)$ coincide.
\end{lemma}

\begin{proof}
    Let $(\eta', \kappa', g_0', \delta')$ be the inertia of $(\U', B')$.
	By \cref{lemma:twist-on-(U-B)}, $\psi_0\inv \circ B'$ is $(\psi_0\inv \circ \vphi_0' \circ \psi_0)$-sesquilinear, where $\vphi_0'$ is the superinvolution given by $\eta'$.
	Given $X_t \in \D_t$,
	we have $\psi_0 (X_t)\in \D_t'$ and, hence, $\vphi_0' \big(\psi_0 (X_t) \big) = \eta' (t) \psi_0 (X_t)$.
	It follows that $(\psi_0\inv \circ \vphi_0' \circ \psi_0) (X_t) = \eta'(t) X_t$, therefore the superinvolution $(\psi_0\inv \circ \vphi_0' \circ \psi_0)$ also corresponds to the map $\eta'\from T \to \FF^\times$.

	Since $\dim_{\D'} \U_x' = \dim_\D (\U_x')^{\psi_0\inv}$, for all $x \in G^\#/T$, the graded $\D$-module $(\U')^{\psi_0\inv}$ corresponds to $\kappa'$.
	Also, it is clear that $\deg (\psi_0\inv \circ B') = g_0'$.
    
    Finally, using that $\psi_0\inv \circ \vphi_0 \circ \psi_0$ and $\vphi_0$ are involutive, we have that
	\begin{align*}
		\overline{(\psi_0\inv \circ B)} (u,v) & = \sign{u}{v} (\psi_0\inv \circ \vphi_0 \circ \psi_0) \big( (\psi_0\inv \circ B)(v, u) \big)                            \\
		                                      & = \psi_0\inv \bigg( \sign{u}{v} \vphi_0 \big( B(v,u) \big) \bigg)                                                       \\
		                                      & = \psi_0\inv \big( \overline B (u,v) \big) = \psi_0\inv \big( \delta B (u,v) \big) = \delta (\psi_0\inv \circ B) (u,v),
	\end{align*}
	for all $u, v \in \U\even \cup \U\odd$.
\end{proof}

The $\D^\times_{\mathrm{gr}}$-action is defined by $d \cdot (\U, B) = (\U, d B)$.  
Let $(\eta', \kappa', g_0', \delta')$ denote the inertia of $(\U, d B)$.  
\begin{itemize}  
    \item Since $d B$ is $\bigl(\mathrm{sInt}_d \circ \vphi_0\bigr)$-sesquilinear, $\eta'$ corresponds to this composition, yielding $\eta'(s) = \tilde{\beta}(t, s) \eta(s)$ for all $s \in T$, where $t \coloneqq \deg d$.  
    \item As $\U$ remains unchanged, $\kappa' = \kappa$.  
    \item Clearly, $g_0' = \deg(d B) = t g_0$.  
    \item By the ``moreover'' part in \cref{thm:vphi-involution-iff-delta-pm-1}, $\overline{d B} = (-1)^{|t|} \eta(t) \delta(d B)$, hence $\delta' = (-1)^{|t|} \eta(t) \delta$.  
\end{itemize}  
The parameters $(\eta', \kappa', g_0', \delta')$ depend only on $t \in T$.  
Thus, the $\D^\times_{\mathrm{gr}}$-action on $\mathbf{I}(T, \tilde\beta)$ factors through the quotient $T \iso \D^\times_{\mathrm{gr}} / \FF^\times$.

Next, by \cref{lemma:twist-same-inertia}, $A$ acts trivially on $\mathbf{I}(T, \tilde\beta)$.
Finally, consider the $G^\#$-action.
Let $g\in G^\#$ and let $(\eta', \kappa', g_0', \delta')$ be the inertia of $g \cdot (\U, B) = (\U^{[g]}, B^{[g]})$.
\begin{itemize}
\item By \cref{lemma:B^[g]-does-the-job}, $B^{[g]}$ is $\vphi_0$-sesquilinear, so $\eta' = \eta$. 
\item By the definition of $\U^{[g]}$, $\kappa' = g\cdot \kappa$ where $(g\cdot \kappa) (x) \coloneqq \kappa(g\inv x)$ for all $x\in G^\#/T$.
\item We have $g_0' = \deg B^{[g]} = g_0 g^{-2}$.
\item From \cref{lemma:B-g-delta}, $\delta' = (-1)^{|g|} \delta$.
\end{itemize}

    Thus, the group $T \times G^\#$ acts on $\mathbf{I}(T, \tilde\beta)$ via
    \[\label{def:TxG-action}
        \begin{split}
        t \cdot (\eta, \kappa, g_0, \delta) & \coloneqq \bigl(\tilde{\beta}(t, \cdot)\eta,\; \kappa,\; t g_0,\; (-1)^{|t|} \eta(t)\delta\bigr)\\
        g \cdot (\eta, \kappa, g_0, \delta) & \coloneqq \bigl(\eta,\; g \cdot \kappa,\; g_0 g^{-2},\; (-1)^{|g|} \delta\bigr)
        \end{split}
    \]
    for all $t \in T$, $g \in G^\#$, and $(\eta, \kappa, g_0, \delta) \in \mathbf{I}(T, \tilde\beta)$.
Now, \cref{cor:iso-with-actions} can be restated as follows:

\begin{thm}\label{thm:iso-(R-vphi)-with-parameters}
    Let $(\D, \U, B)$ and $(\D', \U', B')$ be as in \cref{def:E(D-U-B)}, with parameters $(T, \tilde\beta, \eta, \kappa, g_0, \delta)$ and $(T', \tilde\beta', \eta', \kappa', g_0', \delta')$.  
    Then $E(\D, \U, B) \iso E(\D', \U', B')$ \IFF $T = T'$, $\tilde\beta = \tilde\beta'$, and $(\eta, \kappa, g_0, \delta)$ and $(\eta', \kappa', g_0', \delta')$ lie in the same orbit under the $T \times G^\#$-action on $\mathbf{I}(T, \tilde\beta)$ defined by \cref{def:TxG-action}. \qed
\end{thm}

\subsubsection{Simplifying the parametrization of \texorpdfstring{$(R, \vphi)$}{(R,phi)}}\label{subsec:simplify-parameters}

We can simplify \cref{thm:iso-(R-vphi)-with-parameters} by reducing the complexity of both the parameter set $\mathbf{I}(T, \tilde\beta)$ and the acting group.
To this end, we will repeatedly employ the following straightforward observation:

\begin{lemma}\label{lemma:lemma-on-actions}
    Let $H$ be a group, let $X$ and $Y$ be $H$-sets, and let $\pi\from X \to Y$ be an $H$-equivariant map. 
    For any $y \in Y$, the $H$-action on $X$ restricts to a $\operatorname{Stab}_H (y)$-action on $\pi\inv(y)$ and, if the $H$-action on $Y$ is transitive, the inclusion map $\pi\inv(y) \hookrightarrow X$ induces a bijection between $\operatorname{Stab}_H (y)$-orbits in $\pi\inv(y)$ and $H$-orbits in $X$. \qed
\end{lemma}

    

First, consider the case where $\D$ is odd.
Recall from \cref{conv:pick-even-form} that we may assume the sesquilinear form $B$ is even by changing $B$ to $dB$ for some homogeneous $0\neq d \in \D\odd$ if necessary. 
In terms of parameters, this means that we can choose the inertia $(\eta,\kappa,g_0,\delta) \in \mathbf{I}(T,\tilde\beta)$ with $|g_0| = \bar{0}$.
Thus, by \cref{lemma:lemma-on-actions}, it suffices to restrict the action defined by \cref{def:TxG-action} to a $T^+ \times G^\#$-action on the subset
\[
    \mathbf{I}(T, \tilde\beta)_\bz \coloneqq \left\{ (\eta, \kappa, g_0, \delta) \in \mathbf{I}(T, \tilde\beta) \; \middle| \; |g_0| = \bz\right\}.
\]
In the case where $\D$ is even, we keep
\[
\mathbf{I}(T, \tilde\beta)_\bz \coloneqq \mathbf{I}(T, \tilde\beta).
\]

We now address the parameter $\eta$.  
Consider the following equivalence relation:  

\begin{defi}\label{def:equiv-eta-even}  
    Let $\eta, \eta'\colon T \to \pmone$ be maps satisfying $\mathrm{d}\eta = \mathrm{d}\eta' = \tilde\beta$.  
    We say $\eta$ and $\eta'$ are \emph{evenly equivalent} and write $\eta \sim_\bz \eta'$ if there exists $t \in T^+$ such that $\eta' = \tilde\beta(t, \cdot)\eta$.  
\end{defi}  

For a fixed map $\eta$, let $Y_\eta$ denote its equivalence class with respect to $\sim_\bz$.
Consider the $T^+ \times G^\#$-action on \( X_\eta \coloneqq \left\{(\eta', \kappa, g_0, \delta) \in \mathbf{I}(T, \tilde\beta)_\bz \mid \eta' \in Y_\eta\right\} \subseteq \mathbf{I}(T, \tilde\beta)_\bz\). 
Applying \cref{lemma:lemma-on-actions} and recalling that $\rad\tilde\beta \subseteq T^+$ (\cref{lemma:rad-tilde-beta}), we get that the $T^+ \times G^\#$-orbits in $X_\eta$ are in bijection with the $\rad\tilde\beta \times G^\#$-orbits in
\[
    \mathbf{I}(T, \tilde\beta)_{\bz,\eta} \coloneqq \left\{ (\kappa, g_0, \delta) \mid (\eta, \kappa, g_0, \delta) \in \mathbf{I}(T, \tilde\beta)_\bz \right\}.
\]
Note that by choosing a representative $\eta$ for each equivalence class, we now consider a collection of $\rad\tilde\beta \times G^\#$-actions.
In the next section (\cref{prop:osp-p-unique-eta,lemma:2-etas}), we will see that if $(\D, \vphi_0)$ is a superinvolution-simple superalgebra, then we have at most two equivalence classes.

Finally, we wish to fix $\delta = 1$, which corresponds to choosing $B$ super-Hermitian, as in \cref{rmk:only-super-hermitian}.
Note that in every $\rad\tilde\beta \times G^\#$-orbit in $\mathbf{I}(T, \tilde\beta)_{\bz,\eta}$, we have a triple $(\kappa, g_0, \delta)$ with $\delta = 1$, since the action by $(e, \bar1) \in G^\# = G \times \ZZ_2$ changes the sign of $\delta$.

To apply \cref{lemma:lemma-on-actions}, let $\pi\from \mathbf{I}(T, \tilde\beta)_{\bz,\eta} \to \pmone$ be the projection on the third entry.
The action by $\rad\tilde\beta \times G^\#$ on $\pmone$ is given by
\[
    \begin{split}
        t\cdot \delta &\coloneqq \eta(t)\delta\\
        g\cdot \delta &\coloneqq (-1)^{|g|}\delta\,,
    \end{split}
\] 
for all $t\in \rad \tilde\beta$, $g\in G^\#$ and $\delta \in \pmone$.
It is well defined since $\eta$ is fixed and $\eta\!\restriction_{\rad \tilde\beta}$ is a group homomorphism.
Taking $y\coloneqq +1$, we reduce to the action of
\[\label{eq:mathcal-G}
	\mathcal G_\eta \coloneqq \{ (t,g) \in (\rad \tilde\beta) \times G^\# \mid \eta(t) = (-1)^{|g|} \}
\]
on
\[\label{eq:I-eta-plus}
	\mathbf{I}(T, \tilde\beta)_{\bz,\eta,+} \coloneqq \{ (\kappa, g_0) \mid (\kappa, g_0, 1) \in \mathbf{I}(T, \tilde\beta)_{\bz,\eta} \}
\]
given by
\begin{equation}\label{eq:mc-G-action}
\begin{split}
    t \cdot (\kappa, g_0) & \coloneqq (\kappa, t g_0)\\
    g \cdot (\kappa, g_0) & \coloneqq (g\cdot \kappa, g_0 g^{-2})\,.
\end{split}
\end{equation}
From \cref{thm:iso-(R-vphi)-with-parameters}, we conclude:

\begin{cor}\label{cor:collection-of-orbits}
    Let $(T, \tilde\beta)$ be the pair associated to a graded-division superalgebra $\D$, and let $\{\eta_i\}_{i\in I}$ be a choice of representatives for each equivalence class of superinvolutions on $\D$ (\cref{def:equiv-eta-even}). 
    Then the map sending $(\kappa, g_0)\in \mathbf{I}(T, \tilde\beta)_{\bz,\eta, +}$ to $E(\D,\U, B)$ with parameters $(T, \tilde\beta, \eta, \kappa, g_0, +1)$ as in \cref{def:E(D-U-B)} induces a bijection between the collection of the orbits of the $\mc G_{\eta_i}$-actions on $\mathbf{I}(T, \tilde\beta)_{\bz, \eta, +}$, as defined in \cref{eq:mathcal-G,eq:I-eta-plus,eq:mc-G-action,eq:mc-G-action}, and the isomorphism classes of $E(\D', \U', B')$ with $\D \iso \D'$. \qed
\end{cor}

\phantomsection\label{subsubsec:construction-U-B}

To conclude this section, for a given graded-division superalgebra with superinvolution $(\D, \vphi_0)$, associated to $(T, \tilde\beta, \eta)$, and a given $(\kappa, g_0)\in \mathbf{I}(T,\tilde\beta)_{\bz,\eta,+}$, we will show how to construct a pair $(\U, B)$, and hence the graded-superinvolution-simple superalgebra $E(\D,\U,B)$.
To this end, as done in \cref{phsec:reduce-B-to-FF-bilinear}, we fix elements $0 \neq X_t \in \D_t$, for every $t\in T$, and a set-theoretic section $\xi\from G/T^+ \to G$ of the natural homomorphism $G \to G/T^+$.

\begin{defi}\label{inertia-even-and-odd-case}\label{defi:odd-D-kappa-g_0-admissible}
    Let $g_0 \in G^\#$ and let $\kappa\from G^\#/T \to \FF^\times$ be a map with finite support.
    Recall from \cref{phsec:kappas-even-and-odd} that we reinterpret $\kappa$ as follows: for even $\D$, as two maps $\kappa_\bz, \kappa_\bo\from G/T \to \FF^\times$; for odd $\D$, as a map $\kappa\from G/T^+ \to \FF^\times$.
    We will say that $(\kappa_\bz, \kappa_\bo)$ or $\kappa$ (depending on the case) is \emph{$g_0$-admissible} if $(\kappa, g_0)\in \mathbf{I}(T,\tilde\beta)_{\bz,\eta,+}$.
\end{defi}

We have three cases to consider: even $\D$ and even $B$, even $\D$ and odd $B$, and odd $\D$ and even $B$.
We will spell out the $g_0$-admissibility conditions for each case below.


\paragraph{Even $\D$ and even $B$}\label{para:even-D-even-B}
If $T=T^+$ and $g_0 = (h_0, \bar 0)$,  where $h_0 \in G$, then $(\kappa_\bz, \kappa_\bo)$ is $g_0$-admissible \IFF, for all $i\in \ZZ_2$ and $x\in G/T$, 
\begin{enumerate}
    \item $\kappa_i(x) = \kappa_i(h_0\inv x\inv)$;
    \item if $h_0 x^2 = T$ and $\eta(h_0 g^2) = -(-1)^{i}$ for some (and hence any) $g\in x$, then $\kappa_i(x)$ is even.
\end{enumerate}
Given such a pair $(\kappa_\bz, \kappa_\bo)$, define
\[
    \forall i\in\ZZ_2, \quad \U^i \coloneqq \bigoplus_{x \in \supp \kappa_i} (\D^{\kappa_i(x)})^{[\xi(x)]}\,.
\]
We will define the matrix $\Phi\in M_{k_\bz|k_\bo}(\D)$, with $k_i \coloneqq |\kappa_i|$, representing the sesquilinear form $B$ as follows.
For $x, y \in \supp \kappa_i$ satisfying $h_0xy = T$, set $t \coloneqq h_0 \xi(x)\xi(y) \in T$.
If $x = y$, define $\Phi(i, x)$ as the $\kappa_i(x) \times \kappa_i(x)$-block given by
\begin{enumerate}
    \item $I_{\kappa_i(x)} \tensor X_{t}$ if $(-1)^i \eta(t) = 1$;
    \item $J_{\kappa_i(x)} \tensor X_{t}$, where $J_{\kappa_i(x)} \coloneqq \begin{pmatrix}
                  0                & I_{\kappa_i(x)/2} \\
                  -I_{\kappa_i(x)/2} & 0
              \end{pmatrix}$, if $(-1)^i \eta(t) = -1$ (recall that $\kappa_i(x)$ is even in this case by \cref{inertia-even-and-odd-case}). 
\end{enumerate}
If $x \neq y$, define $\Phi(i, x)$ as the $2\kappa_i(x) \times 2\kappa_i(x)$-block with entries in $\D$:
\begin{enumerate}[resume]
    \item $\begin{pmatrix}
        0                                                  & I_{\kappa_i(x)} \\
        (-1)^{i} \eta(t)I_{\kappa_i(x)} & 0
    \end{pmatrix} \tensor X_t$. 
\end{enumerate}
Then we define $\Phi$ as the block matrix:
\[\label{eq:puting-the-blocks-of-Phi-together}
    \sbox0{$\begin{matrix}
        \Phi(\bar 0, x_1)&& \\
        & \ddots &\\
        && \Phi(\bar 0, x_{\ell_\bz})
    \end{matrix}$}
    \sbox1{$\begin{matrix}
        \Phi(\bar 1, x_1')&& \\
        & \ddots &\\
        && \Phi(\bar 1, x_{\ell_\bo}')
    \end{matrix}$}
    \Phi \coloneqq
    \left(\begin{array}{c|c}
            \usebox{0} & 0\\
            \hline
            \rule{0pt}{6ex}
            0 & \usebox{1}
        \end{array}\right)\,,
\] 
where $\ell_i$ is the number of pairs $x, y \in \supp \kappa_i$ satisfying $h_0xy = T$. 


\paragraph{Even $\D$ and odd $B$}\label{para:even-D-odd-B}
If $T=T^+$ and $g_0 = (h_0, \bar 1)$, where $h_0 \in G$, then $(\kappa_\bz, \kappa_\bo)$ is $g_0$-admissible \IFF $\kappa_\bo(x) = \kappa_\bz(h_0\inv x\inv)$ for all $x \in G/T$.
In particular, $\kappa_\bo$ is determined by $\kappa_\bz$.
Given any $\kappa_\bz$, we set
\[
    \U\even \coloneqq \bigoplus_{x\in \supp{\kappa_\bz}}(\D^{\kappa_\bz(x)})^{[\xi(x)]}\, \AND \,\U\odd \coloneqq ((\U\even)\Star)^{[h_0\inv]}\,,
\]
and define $B\from \U\times\U \to \D$ by
\[\label{eq:defi-B-for-P(n)}
    \forall u_\bz, u_\bz' \in \U\even, \, u_\bo, u_\bo' \in \U\odd, \quad B(u_\bz + u_\bo, u_\bz' + u_\bo') \coloneqq u_\bo(u_\bz') + \vphi_0( u_\bo' (u_\bz) )\,.
\]
For any graded basis $\B$ of $\U\even$, $\B \cup \B\Star$ is a graded basis of $\U$, and the matrix representing $B$ is given by
\[
    \Phi = \left(\begin{array}{c|c}
            0 & I_{k}\\
            \hline
            I_{k} & 0
        \end{array}\right) \in M_{k|k}(\D),\, \text{where}\, k \coloneqq |\kappa_\bz|.
\]


\paragraph{Odd $\D$ and even $B$}\label{para:odd-D-odd-B}
If $T \neq T^+$ and $g_0 = (h_0, \bar 0) \in G^\#$, then $\kappa\from G/T^+ \to \ZZ_{\geq 0}$ is $g_0$-admissible \IFF, for all $x \in G/T^+$,
\begin{enumerate}
    \item $\kappa(x) = \kappa(h_0\inv x\inv)$;
    \item if $h_0 x^2 = T^+$ and $\eta(h_0 g^2) = -1$ for some (hence any) $g \in x$, then $\kappa(x)$ is even.
\end{enumerate}
Given such $\kappa$, we define
\[
    \U \coloneqq \bigoplus_{x \in \supp \kappa} (\D^{\kappa(x)})^{[\xi(x)]}
\]
and the matrix $\Phi\in M_{k}(\D)$, with $k \coloneqq |\kappa|$, representing the sesquilinear form $B$ by
\[
    \Phi \coloneqq 
    \begin{pmatrix}
        \Phi(\bar 0, x_1)&& \\
        & \ddots &\\
        && \Phi(\bar 0, x_{\ell})
    \end{pmatrix}\,,
\]
where the blocks $\Phi(\bar 0, x)$ are as above and $\ell$ is the number of pairs $x, y \in \supp \kappa$ satisfying $h_0xy = T^+$.

\section{Gradings on Superinvolution-Simple Associative Superalgebras}\label{sec:gradings-on-vphi-simple}

In \cref{subsec:grdd-simple-ass-algebraically-closed}, we presented classification results for graded-simple associative superalgebras over an algebraically closed field $\FF$, $\Char \FF \neq 2$, and then specialized these results to gradings on simple superalgebras. 
In \cref{sec:grdd-sinv-simple}, we established general classification theorems for graded-superinvolution-simple superalgebras. 
Specializing these results to gradings on superinvolution-simple superalgebras is more involved, so we dedicate a whole section to it.  

Throughout this section, we assume $\FF$ is algebraically closed and $\Char \FF \neq 2$.
Recall that finite-dimensional superinvolution-simple associative superalgebras are of types $M$, $M\times M\sop$ or $Q\times Q\sop$ (\cref{subsec:simple-superalgebras}).

\subsection{Simple superalgebras of type \texorpdfstring{$M$}{M} with superinvolution}\label{sec:vphi-grds-on-type-M}

Let $(R, \vphi)$ be a finite-dimensional graded superalgebra with superinvolution of type $M$. 
Then $(R, \vphi) \iso E(\D, \U, B)$ for some triple $(\D, \U, B)$ as in \cref{def:E(D-U-B)}, with $\D$ also of type $M$. 
Let $\vphi_0$ be the superinvolution on $\D$ such that $B$ is $\vphi_0$-sesquilinear and let $(T, \tilde\beta, \eta, \kappa, g_0, \delta)$ be the parameters of $(\D, \U, B)$.

\subsubsection{\texorpdfstring{$(\D,\vphi_0)$}{(D, phi0)} as a matrix superalgebra with superinvolution}\label{subsec:D-vphi_0-matrix}

Since $(\D, \vphi_0)$ is of type $M$, we have important restrictions on $(T, \tilde\beta)$. 
The next result can be found in \cite[Theorem 4.3]{TT} (assuming that $G$ is finite and $\Char \FF = 0$): 

\begin{lemma}\label{lemma:sinv-type-M}
    Let $\D$ be a finite-dimensional graded-division superalgebra of type $M$.
    If $\D$ admits a degree-preserving superinvolution, then $\D$ is even and $T = T^+$ is an elementary $2$-group.
    In particular $\beta = \tilde\beta$.
\end{lemma}

\begin{proof}
    By \cref{prop:simple-tilde-beta,cor:super-anti-auto-squares-in-radical}, we have that $t^2 
    \in \rad\tilde\beta = \{e\}$ for all $t\in T$, so $T$ is an elementary $2$-group. 
    Hence, by \cref{cor:eta-t-square}, $T^-$ is empty. 
\end{proof}

We can write $\D$ explicitly as a matrix algebra with involution by revisiting \cref{subsubsec:standard-realizations}. 
Decompose $T$ as $A\times B$ where $\beta(A,A) = \beta(B,B) = 1$, and indentify $\D$ with the corresponding standard realization $\End_\FF (N)$, where $N$ is a vector space with a fixed basis $\{ e_b \}_{b\in B}$. 
Then, the transposition on the $\D = \End(N)$ seem as a matrix algebra by means of this basis is the adjunction with respect the symmetric bilinear form determined by:
\[\label{eq:bilinear-form-transposition}
\forall b, b'\in B, \quad
        \langle e_b , e_{b'} \rangle
        \coloneqq
        \begin{cases}
            1 & \text{if}\ b=b' \\
            0 & \text{if}\ b\neq b'\ .
        \end{cases}
    \]

\begin{lemma}\label{lemma:transp-std-realization}
    The transposition in the standard realization $\D = \End_\FF (N)$ is given by $X_{ab}\transp = \beta(a,b) X_{ab\inv}$, for all $a\in A$ and $b\in B$. 
\end{lemma}

\begin{proof}
    Let $a \in A$. 
    For all $b',b'' \in B$, the following is true since its non-zero only if $b=b'$:
    \[
        \langle X_a e_{b'}, e_{b''} \rangle = \beta(a, b') \langle e_{b'}, e_{b''} \rangle = \beta(a, b'') \langle e_{b'} , e_{b''} \rangle = \langle e_{b'}, X_a e_{b''} \rangle\,.
    \]
    Hence, $X_a\transp = X_a$. 
    Now consider $b \in B$. 
    For all $b',b'' \in B$,
    \[
        \langle X_b e_{b'}, e_{b''} \rangle = \langle e_{bb'}, e_{b''} \rangle = \langle e_{b'}, e_{b\inv b''} \rangle = \langle e_{b'}, X_{b\inv} e_{b''} \rangle,
    \]
    so $X_b\transp = X_{b\inv}$. 
    Finally, 
    \[
        X_{ab}\transp \!= (X_a X_b)\transp \!= X_{b\inv}X_a \!= \beta(b\inv, a) X_a X_{b\inv} \!= \beta(b,a)\inv X_{ab\inv} \!= \beta(a,b) X_{ab\inv}.
    \]
\end{proof}

When $T$ is an elementary $2$-group, \cref{lemma:transp-std-realization} tells us that the transposition is a degree-preserving superinvolution, with $\eta(ab) = \beta(a,b)$ for all $a\in A$ and $b\in B$.

\begin{prop}\label{prop:osp-p-unique-eta}
    There is a unique equivalence class with respect to $\sim_\bz$ (see \cref{def:equiv-eta-even}) of maps $\eta\from T \to \FF^\times$ such that $\mathrm{d}\eta = \tilde\beta = \beta$.
\end{prop}

\begin{proof}
    Let $\eta'\from T \to \FF^\times$ be a map corresponding to another degree-preserving involution $\vphi_0'$ on $\D$. 
    Clearly, $\vphi_0'$ is the transposition composed with a degree-preserving automorphism of $\D$, so, by \cref{lemma:Aut(D)-widehat-T}, $\eta' = \chi \eta$, for some $\chi \in \widehat{T}$. 
    Since $\tilde\beta = \beta$ is nondegenerate, there is a $t\in T = T^+$ such that $\chi = \tilde\beta(t, \cdot)$, and hence $\eta' \sim_\bz \eta$.
\end{proof}

\subsubsection{\texorpdfstring{$E(\D, \U, B)$}{E(D, U, B)} as a superalgebra with superinvolution}

Disregarding the $G$-grading, $E(\D, \U, B)$ is a simple (super)algebra equipped with a superinvolution.
Our goal now is to write it as a superalgebra of $\FF$-linear operators $\End_\FF(U)$, and to find a bilinear form $\langle \cdot, \cdot\rangle$ on $U$ whose superadjunction corresponds to $\vphi$.

We will identify $\D$ with $\End_\FF(N)$, 
where $N = N\even$.
Since $N$ is a left $\D$-module, we can define $U \coloneqq \U \tensor_\D N$. 
It follows that $U$ is a superspace, with $U\even = \U\even \tensor_\D N$ and $U\odd = \U\odd \tensor_\D N$, but it is not $G$-graded since $N$ is not $G$-graded.

The following two lemmas are standard, so we omit the proofs.

\begin{lemma}\label{lemma:basis-of-N}
    If $\{ u_1, \ldots, u_k \}$ is a graded $\D$-basis of $\U$ and $\{ v_1, \ldots, v_\ell \}$ is an $\FF$-basis of $N$, then $\{ u_i \tensor_\D v_j \mid 1\leq i \leq k \AND 1\leq j \leq \ell\}$ is an $\FF$-basis of $U$. \qed
\end{lemma}


\begin{lemma}\label{lemma:End-over-D-becomes-over-FF}
    The map $\psi\from \End_\D(\U) \to \End_\FF (U) = \End_\FF (\U \tensor_\D N)$ given by $\psi(L) \coloneqq L \tensor_\D \id_N$ is an isomorphism of superalgebras. \qed
\end{lemma}


Having established the isomorphism $\psi\from \End_\D(\U) \to \End_\FF (U)$, we can transfer the superinvolution $\vphi$ on $\End_\D(\U)$ to a superinvolution on $\End_\FF (U)$ by taking $\psi \circ \vphi \circ \psi\inv$. 
This sends $L\tensor_\D \id_N$ to $\vphi(L)\tensor_\D \id_N$.
We now wish to transfer the sesquilinear form $B$ on $\U$ to a bilinear form on $U$.
Recall that $\barr B = \delta B$, where $\delta \in \pmone$.

First, let us consider the involution $\vphi_0$ on $\D = \End_\FF(N)$. 
It is well known (and also follows from \cref{thm:vphi-involution-iff-delta-pm-1} with trivial grading) that there is a bilinear form $\langle \, , \rangle_N \from N\times N \to \FF$ such that $\vphi_0$ is the adjunction with respect to $\langle \, , \rangle_N$. 
Since $\vphi_0$ is an involution, $\langle \, , \rangle_N$ is either symmetric or skew-symmetric. 
We will write $\delta_N \coloneqq 1$ in the former case and $\delta_N \coloneqq -1$ in the latter. 

\begin{lemma}\label{lemma:the-same-vphi}
    There is an $\FF$-bilinear map $\langle \, , \rangle_U \from U\times U \to \FF$ determined by 
    \[
        \langle u \tensor_\D v , u' \tensor_\D v' \rangle_U \coloneqq \langle v , B(u, u') v' \rangle_N,
    \]
    for all $u, u'\in \U$ and $v, v' \in N$. 
    This bilinear map is supersymmetric if $\delta \delta_N =1$ and super-skew-symmetric if $\delta \delta_N = -1$. 
    Moreover, $\vphi \tensor \id_N$ is the superadjunction with respect to $\langle \, , \rangle_U$. 
\end{lemma}

\begin{proof}
    For every $u\in \U$, let $C_u \from \U \times N \to \FF$ be defined by $C_u(u', v') \coloneqq B(u, u') v'$. 
    Clearly, $C_u$ is $\FF$-bilinear and $\D$-balanced. 
    Hence, there is a unique $\FF$-linear map $\U \tensor_\D N \to N$, which we will also denote by $C_u$, such that $u' \tensor_\D v' \mapsto B(u, u')v'$. 
    It is easy to see that the mapping $\U \to \Hom_\FF(\U \tensor N, N)$ given by $u \mapsto C_u$ also is $\FF$-linear. 
    
    Now let $u'\tensor_\D v' \in \U \tensor_\D N$ be fixed, and let $C' = C'_{u'\tensor_\D v'}\from \U\times N \to \FF$ be given by $C'(u,v) \coloneqq \langle v, C_u(u',v') \rangle_N$, \ie, $C'(u,v) = \langle v, B(u, u')v' \rangle_N$. 
    Clearly, $C'$ is $\FF$-bilinear. 
    It is $\D$-balanced since, for all $u \in \U\even \cup \U\odd$, $v\in N$ and $d \in \D = \D\even$,
    \begin{align}
        C'(ud,v) = \langle v, B(ud, u')v' \rangle_N &= \langle v, (-1)^{(|B| + |u|)|d|} \vphi_0(d) B(u, u')v' \rangle_N \\
        &= \langle v, \vphi_0(d) B(u, u')v' \rangle_N = \langle d v, B(u, u')v' \rangle_N\,.
    \end{align}
    It follows that there is a unique $\FF$-linear map $\U \tensor_\D N \to \FF$ such that $u\tensor_\D v \mapsto C' (u,v) = C'_{u'\tensor_\D v'}(u,v)$. 
    We conclude that the $\FF$-bilinear map $\langle \, , \rangle_U \from U\times U \to \FF$ is well-defined. 
    
    The form $\langle \, , \rangle_U$ is supersymmetric or super-skew-symmetric, depending on $\delta \delta_N$, since
    \begin{align}
        \langle u' \tensor_\D v' ,\,  u \tensor_\D v \rangle_U 
        &= \langle v' , B(u', u) v \rangle_N 
        = \langle v' , \sign{u}{u'}\delta\, \vphi_0( B(u, u') ) v \rangle_N \\
        &= \sign{u}{u'}\delta \langle B(u, u') v' , v \rangle_N \\
        &= \sign{u}{u'} \delta \delta_N \langle v, B(u, u') v' \rangle_N \\
        &= \sign{u}{u'} \delta \delta_N \langle u \tensor_\D v , u' \tensor_\D v' \rangle_U, 
    \end{align}
    for all $u, u' \in \U\even \cup \U\odd$ and $v \in N = N\even$.
    
    For the ``moreover'' part, recall from \cref{lemma:End-over-D-becomes-over-FF} that every element in $\End_D(U)$ is of the form $L\tensor_\D \id_N$ for some $L \in \End_\D (\U)$. 
    Hence:
    {\allowdisplaybreaks
    \begin{align}
        \langle (L\tensor_\D \id_N) (u \tensor_D v), u' \tensor_D v' &\rangle_U 
        = \langle L(u) \tensor_\D v, u' \tensor_D v' \rangle_U \\
        &= \langle v, B(L(u), u') v' \rangle_N \\
        &= \langle v, \sign{L}{u} B(u, \vphi(L) (u')) v' \rangle_N \\
        &= \sign{L}{u} \langle u \tensor_\D v, \vphi(L) (u') \tensor_D v' \rangle_U \\
        &= \sign{L}{u} \langle  u \tensor_D v, ( \vphi(L)\tensor_\D \id_N) (u' \tensor_D v') \rangle_U,
    \end{align}
    }
    for all $L \in \End_\D (\U)\even \cup \End_\D (\U)\odd$, $u, u'\in \U\even \cup \U\odd$ and $v,v' \in N$. 
\end{proof}

\begin{prop}\label{prop:g_0-goes-to-p_0}
    As a superalgebra with superinvolution, $E(\D, \U, B)$, with $\barr B = \delta B$, is isomorphic to $M^*(m,n, p_0)$, where $p_0 \coloneqq |B|$ and
    \begin{enumerate}
        \item if $\delta\delta_N = 1$, then $m \coloneqq \dim \U\even \tensor_\D N$ and $n \coloneqq \dim \U\odd \tensor_\D N$;\label{it:delta-delta-1}
        \item if $\delta\delta_N = -1$, then $m \coloneqq \dim \U\odd \tensor_\D N$ and $n \coloneqq \dim \U\even \tensor_\D N$.
    \end{enumerate}
\end{prop}

\begin{proof}
    It is clear that the bilinear form $\langle \, , \rangle_U \from U\times U \to \FF$ defined in \cref{lemma:the-same-vphi} has the same parity as $B$, which is $p_0$. 
    If $\delta\delta_N = -1$, as noted in \cref{phsec:parity-reversed}, we can replace $U$ by $U^{[\barr 1]}$ so $\langle \, , \rangle_U$ becomes supersymmetric.
    By \cref{prop:self-dual-components,prop:pair-of-dual-components}, we can find a ($\ZZ_2$-graded) basis of $U$ where the matrix representing $\vphi \tensor_\D \id_N$ is the one as in \cref{defi:M(m-n-p_0)}, so $(\End_\FF (U), \vphi \tensor_\D \id_N) \iso M^*(m, n, p_0)$. 
    By \cref{lemma:End-over-D-becomes-over-FF}, we have the desired result. 
\end{proof}

\subsubsection{Parametrization of gradings on \texorpdfstring{$M^*(m, n, p_0)$}{M*(m, n, p0)}}

We are now in a position to define explicit models for gradings on $M^*(m,n,p_0)$.
For each pair $(T, \beta)$ where $T$ is a finite $2$-elementary subgroup of $G$ and $\beta\from {T\times T} \to \FF^\times$ is a nondegenerate alternating bicharacter, we set $\tilde\beta \coloneqq \beta$ and fix a standard realization $\D = \End(N)$ of type $M$ (\cref{def:standard-realization-M}) associated to $(T, \tilde\beta)$.

Let $\vphi_0\from \D \to \D$ be the transposition on $\D$ and $\eta\from T \to \FF^\times$ be the map corresponding to it.
Recall that we have a unique equivalence class of $\sim_\bz$ (\cref{prop:osp-p-unique-eta}), and that $\vphi_0$ is the adjunction with respect to the symmetric bilinear form in \cref{eq:bilinear-form-transposition}, so $\delta_N = 1$.
By \cref{rmk:only-super-hermitian}, we can always choose $\delta = 1$, so $\delta\delta_N = 1$ and we are in \cref{it:delta-delta-1} of \cref{prop:g_0-goes-to-p_0}.
Therefore, we are aligned with all the simplifications of parameters done in \cref{subsec:simplify-parameters}.

\begin{defi}\label{def:model-grd-M(m-n-0)}
    Let $T \subseteq G$ be a finite $2$-elementary subgroup, let $\beta\from {T\times T} \to \FF^\times$ be a nondegenerate alternating bicharacter, let $g_0 = (h_0, p_0) \in G^\#$ and let $(\kappa_\bz, \kappa_\bo)$ be a $g_0$-admissible pair of maps $G/T \to \ZZ_{\geq 0}$ (\cref{inertia-even-and-odd-case}). 
    Construct the corresponding graded $\D$-module $\U$ and nondegenerate sesquilinear form $B\from \U\times \U \to \D$ as in \cref{subsubsec:construction-U-B}.
    We denote the $G$-graded superalgebra with superinvolution $E(\D, \U, B)$ by $M^*(T, \beta, \kappa_\bz, \kappa_\bo, g_0)$. 
    By choosing a graded basis for $\U$, we get an isomorphism $E(\D, \U, B) \iso M_{k_\bz | k_\bo}(\D)$, with $k_i \coloneqq |\kappa_i|$.
    Then, via Kronecker product, this becomes the superalgebra with superinvolution $M^*(m,n, p_0)$, where $m \coloneqq k_\bz\sqrt{|T|}$ and $n \coloneqq k_\bo\sqrt{|T|}$ (compare with \cref{def:Gamma-T-beta-kappa-even}), endowed with the superinvolution $\vphi$ given by 
    \[
        \forall X \in M(m,n),\quad \vphi(X) \coloneqq \Phi\inv X\stransp \Phi,
    \]
    where $\Phi$ is the matrix representing $B$ with respect to the chosen basis (interpreted as a matrix with entries in $\FF$ rather than $\D$). 
\end{defi}

Note that, in the case $|g_0| = \bar{1}$, the matrix $\Phi \in M(n, n)$ takes a specially simple form, given in \cref{eq:matrix-for-P}.

\begin{thm}\label{thm:osp-and-p-associative}
    Suppose the superalgebra with superinvolution $M^*(m,n,p_0)$ is endowed with a $G$-grading. 
    Then it is isomorphic, as a graded superalgebra with superinvolution, to $M^*(T,\beta, \kappa_\bz, \kappa_\bo, g_0)$ as in \cref{def:model-grd-M(m-n-0)}, with $p_0 = |g_0|$. 
    Moreover, $M^*(T, \beta, \kappa_\bz, \kappa_\bo, g_0) \iso M^*(T', \beta', \kappa_\bz', \kappa_\bo', g_0')$ \IFF $T =T'$, $\beta = \beta'$ and there is $g \in G$ such that $\kappa_\bz' = g\cdot\kappa_\bz$, $\kappa_\bo' = g\cdot\kappa_\bo$ and $g_0' = g^{-2}g_0$. 
\end{thm}

\begin{proof}
    The first assertion follows from the discussion above. 
    The isomorphism condition follows from \cref{prop:parametrization-T-beta,cor:collection-of-orbits} by noting that, in this case, $\mc G_\eta = \{e\} \times G$ (see \cref{eq:mathcal-G}). 
\end{proof}

\subsection{Superinvolution-simple superalgebras of types \texorpdfstring{$M\times M\sop$ and $Q\times Q\sop$}{MxMsop and QxQsop}}\label{sec:MxM-and-QxQ-associative}

We now consider the case $R \coloneqq S \times S\sop$ equipped with the exchange superinvolution $\vphi$, where $S$ is a finite-dimensional simple superalgebra. 
When endowed with a grading, $(R, \vphi)$ becomes graded-superinvolution-simple but not necessarily graded-simple.
As in \cref{sec:grdd-sinv-simple}, we will address the non-graded-simple case first as it is simpler.

\subsubsection{Non-graded-simple case}\label{subsubsec:non-graded-simple-MxMop-or-QxQsop}

By \cref{cor:SxSsop-with-dcc}, since $R$ is not graded-simple, $(R, \vphi) \iso \Eex (\D, \U)$. 
Given that $S \times \{ 0 \}$ and $\{ 0 \} \times S\sop$ are the only nonzero proper superideals of $R$, it follows one of them must be isomorphic to $\End_\D(\U) \times \{0\}$ and the other to $\{0\}\times \End_{\D\sop}(\U\Star)$.
In particular, $\D$ must be simple as a superalgebra and of the same type as $S$.
The results below are straightforward specializations of \cref{thm:iso-D-even-ExEsop,thm:iso-D-odd-ExEsop} to the simple superalgebras of types $M$ and $Q$ (see \cref{para:gradings-on-M-and-Q}).
Recall that $\kappa\Star \from G^\#/T \to \ZZ_{\geq 0}$ is defined by $\kappa\Star (x) \coloneqq \kappa (x\inv)$.

\begin{thm}\label{thm:MxM-type-I}
    Let $(R, \vphi)$ be a superinvolution-simple superalgebra of type ${M\times M\sop}$ endowed with a $G$-grading that does not make it graded-simple.
    As a graded superalgebra with superinvolution, $(R, \vphi)$ is isomorphic to $S\times S\sop$ with exchange superinvolution, where either $S$ is isomorphic to $M(T, \beta, \kappa_\bz, \kappa_\bo)$, as in \cref{def:Gamma-T-beta-kappa-even}, or $S$ is isomorphic to $M(T, \tilde\beta, \kappa)$, as in \cref{def:Gamma-T-beta-kappa-odd}, with these cases being mutually exclusive.
    Moreover,
    \[M(T, \beta, \kappa_\bz, \kappa_\bo) \times M(T, \beta, \kappa_\bz, \kappa_\bo)\sop \iso M(T', \beta', \kappa_\bz', \kappa_\bo') \times M(T', \beta', \kappa_\bz', \kappa_\bo')\sop\]
    \IFF $T = T'$ and one of the following conditions holds:
	\begin{enumerate}
	    \item $\beta'=\beta$ and there is $g\in G$ such that either $\kappa_{\bar 0}' = g \cdot \kappa_{\bar 0}$ and $\kappa_{\bar 1}'= g \cdot \kappa_{\bar 1}$, or $\kappa_{\bar 0}'=g \cdot \kappa_{\bar 1}$ and $\kappa_{\bar 1}'=g \cdot \kappa_{\bar 0}$; 
	    \item $\beta'=\beta\inv$ and there is $g\in G$ such that either $\kappa_{\bar 0}'=g \cdot \kappa_{\bar 0}\Star$ and $\kappa_{\bar 1}'=g \cdot \kappa_{\bar 1}\Star$, or $\kappa_{\bar 0}'=g \cdot \kappa_{\bar 1}\Star$ and $\kappa_{\bar 1}'=g \cdot \kappa_{\bar 0}\Star$;
	\end{enumerate}
    and
    \[{M (T, \tilde\beta, \kappa) \times M (T, \tilde\beta, \kappa)\sop} \iso {M (T', \tilde\beta', \kappa') \times M (T', \tilde\beta', \kappa')\sop}\] 
    \IFF $T = T'$ and one of the following conditions holds:
    \begin{enumerate}
        \setcounter{enumi}{2}
	    \item $\tilde\beta'=\tilde\beta$, and there is $g\in G$ such that $\kappa' = g \cdot \kappa$;
	    \item $\tilde\beta'=\tilde\beta\inv$, and there is $g\in G$ such that $\kappa' = g \cdot \kappa\Star$. \qed
	\end{enumerate}
\end{thm}

\begin{thm}\label{thm:QxQ-type-I}
    Let $(R, \vphi)$ be a superinvolution-simple superalgebra of type ${Q\times Q\sop}$ endowed with a $G$-grading that does not make it graded-simple.
    As a graded superalgebra with superinvolution, $(R, \vphi)$ is isomorphic to
    $S\times S\sop$ with exchange superinvolution, where $S$ is isomorphic to $Q (T^+, \beta^+, h, \kappa)$, as in \cref{def:Gamma-T-beta-kappa-Q}.
    Moreover,
    \[Q (T^+, \beta^+, h, \kappa) \times Q (T^+, \beta^+, h, \kappa)\sop \iso Q (T^+, \beta^+, h, \kappa) \times Q (T^+, \beta^+, h, \kappa)\sop\]
    \IFF $T^+ = T'^+$, $h = h'$ and one of the following conditions holds: 
    \begin{enumerate}
	    \item $\beta'^+ = \beta^+$ and there is $g\in G$ such that $\kappa' = g \cdot \kappa$;
	    \item $\beta'^+ = (\beta^+)\inv$ and there is $g\in G$ such that $\kappa' = g \cdot \kappa\Star$. \qed
	\end{enumerate}
\end{thm}


\subsubsection{Graded-division superalgebras of types \texorpdfstring{$M\times M\sop$}{MxMsop} and \texorpdfstring{$Q\times Q\sop$}{QxQsop}}

We start with examples of division gradings on the superalgebras $\FF \times \FF\sop$, $Q(1)\times Q(1)\sop$, and $M(1,1)\times M(1,1)\sop$, all endowed with exchange superinvolution.

\begin{ex}\label{ex:now-FZ2-is-division}
    The superalgebra with superinvolution $\FF \times \FF\sop$ of \cref{ex:FxF-iso-FZ2} admits a $\ZZ_2$-grading that makes it isomorphic to $\FF\ZZ_2$: $\deg (1, 1) = \bar 0$ and $\deg (1, -1) = \bar 1$.
    The exchange superinvolution is determined by $\eta(\bar 0) = 1$ and $\eta(\bar 1) = -1$.
\end{ex}

\begin{ex}\label{ex:now-FZ4-is-division}
    The superalgebra with superinvolution $Q(1) \times Q(1)\sop$ of \cref{ex:FZ2xFZ2sop-iso-FZ4} admits a $\ZZ_4$-grading that makes it isomorphic to $\FF\ZZ_4$:
    $\deg (1,1) = \bar 0$,
    $\deg (u, \bar u) = \bar 1$,
    $\deg (1,-1) = \bar 2$,
    and $\deg (u,- \bar u) = \bar 3$.
    The exchange superinvolution is determined by $\eta(\bar 0) = 1$, $\eta(\bar 1) = 1$, $\eta(\bar 2) = -1$, and $\eta(\bar 3) = -1$.
\end{ex}

\newlength{\origarraycolsep}
\setlength{\origarraycolsep}{\arraycolsep}
\newcommand{\restorearraystretch}{\renewcommand{\arraystretch}{1}}

\setlength{\arraycolsep}{2.9pt}
\renewcommand{\arraystretch}{1}

\begin{ex}\label{ex:superalgebra-O}
	Let $\mc O$ denote the superalgebra $M(1,1) \times M(1,1)\sop$ endowed with the $\ZZ_2\times \ZZ_4$-grading determined by: 
	\begin{align*}
		\deg\! \left(\begin{pmatrix}
			\phantom{.}1 & \phantom{-}0\phantom{.} \\
			\phantom{.}0 & \phantom{-}1\phantom{.}
		\end{pmatrix}, \overline{\begin{pmatrix}
				\phantom{.}1 & \phantom{-}0\phantom{.} \\
				\phantom{.}0 & \phantom{-}1\phantom{.}
			\end{pmatrix}}\right) = (\bar 0, \bar 0),\quad &
		\deg\! \left(\begin{pmatrix}
			\phantom{.}1 & \phantom{-}0\phantom{.} \\
			\phantom{.}0 & \phantom{-}1\phantom{.}
		\end{pmatrix}, -\overline{\begin{pmatrix}
				\phantom{.}1 & \phantom{-}0\phantom{.} \\
				\phantom{.}0 & \phantom{-}1\phantom{.}
			\end{pmatrix}}\right) = (\bar 0, \bar 2),    \\
		\deg\! \left(\begin{pmatrix}
			\phantom{.}1 & \phantom{-}0\phantom{.} \\
			\phantom{.}0 & -1\phantom{.}
		\end{pmatrix}, \overline{\begin{pmatrix}
				\phantom{.}1 & \phantom{-}0\phantom{.} \\
				\phantom{.}0 & -1\phantom{.}
			\end{pmatrix}}\right) = (\bar 1, \bar 0),\quad &
		\deg\! \left(\begin{pmatrix}
			\phantom{.}1 & \phantom{-}0\phantom{.} \\
			\phantom{.}0 & -1\phantom{.}
		\end{pmatrix}, -\overline{\begin{pmatrix}
				\phantom{.}1 & \phantom{-}0\phantom{.} \\
				\phantom{.}0 & -1\phantom{.}
			\end{pmatrix}}\right) = (\bar 1, \bar 2),    \\
		\deg\! \left(\begin{pmatrix}
			\phantom{.}0 & \phantom{-}1\phantom{.} \\
			\phantom{.}1 & \phantom{-}0\phantom{.}
		\end{pmatrix}, \overline{\begin{pmatrix}
				\phantom{.}0 & \phantom{-}1\phantom{.} \\
				\phantom{.}1 & \phantom{-}0\phantom{.}
			\end{pmatrix}}\right) = (\bar 0, \bar 1),\quad &
		\deg\! \left(\begin{pmatrix}
			\phantom{.}0 & \phantom{-}1\phantom{.} \\
			\phantom{.}1 & \phantom{-}0\phantom{.}
		\end{pmatrix}, -\overline{\begin{pmatrix}
				\phantom{.}0 & \phantom{-}1\phantom{.} \\
				\phantom{.}1 & \phantom{-}0\phantom{.}
			\end{pmatrix}}\right) = (\bar 0, \bar 3),    \\
		\deg\! \left(\begin{pmatrix}
			\phantom{.}0 & -1\phantom{.}           \\
			\phantom{.}1 & \phantom{-}0\phantom{.}
		\end{pmatrix}, \overline{\begin{pmatrix}
				\phantom{.}0 & -1\phantom{.}           \\
				\phantom{.}1 & \phantom{-}0\phantom{.}
			\end{pmatrix}}\right) = (\bar 1, \bar 3),\quad &
		\deg\! \left(\begin{pmatrix}
			\phantom{.}0 & -1\phantom{.}           \\
			\phantom{.}1 & \phantom{-}0\phantom{.}
		\end{pmatrix}, -\overline{\begin{pmatrix}
				\phantom{.}0 & -1\phantom{.}           \\
				\phantom{.}1 & \phantom{-}0\phantom{.}
			\end{pmatrix}}\right) = (\bar 1, \bar 1).
	\end{align*}
    \setlength{\arraycolsep}{\origarraycolsep}
    \restorearraystretch
	Then $\mc O$ is an odd graded division superalgebra, and the exchange superinvolution on it preserves degrees.
    The corresponding map $\eta$ takes the value $1$ on $(\bar 0, \bar 0)$, $(\bar 1, \bar 0)$, $(\bar 0, \bar 1)$, and $(\bar 1, \bar 3)$, and the value $-1$ on $(\bar 0, \bar 2)$, $(\bar 1, \bar 2)$, $(\bar 0, \bar 3)$, and $(\bar 1, \bar 1)$.
\end{ex}

\begin{remark}\label{cor:associative-type-II-odd-m=n}
    If $M(m,n) \times M(m,n)\sop$ admits an odd division grading, then the same argument as in \cref{lemma:odd-M-m=n} shows that $m = n$.
\end{remark}

    Given a graded-division superalgebra $\D$, a general method to refine the grading on $\D \times \D\sop$ into a division grading is developed in \cite[Section 4.4]{caios_thesis}.
    The examples above are special cases where $\D$ is, respectively, $\FF \iso M(1,0)$, $Q(1)$ and $M(1,1)$, with the last two graded as in \cref{ex:Q(1)-as-grd-div-SA,ex:Pauli-2x2-super}.

\phantomsection\label{phsec:param-D-MxM-or-QxQ}

We now fix a graded-division superalgebra with superinvolution $(\D, \vphi_0)$ of type either $M \times M\sop$ or $Q \times Q\sop$.
Our goal is to understand how the type is reflected in the parameters $(T, \tilde\beta, \eta)$ (see \cref{def:Parameters-T-beta-eta}).

\begin{lemma}\label{cor:T+-is-elem-2-grp}
    If $(\D, \vphi_0)$ of type either $M \times M\sop$ or $Q \times Q\sop$, then 
    \begin{enumerate}
        \item\label{it:rad-is-<f>} $\rad \tilde\beta = \langle f \rangle$, where $f$ is an order $2$ element such that $\eta(f) = -1$;
        \item\label{it:squares-of-t} for every $t \in T^+$, we have $t^2 = e$, and for every $t \in T^-$, we have $t^2 = f$.
    \end{enumerate}
\end{lemma}

\begin{proof}
    For \cref{it:rad-is-<f>}, recall that $\rad \tilde\beta$ is the support of the supercenter of $\D$.
    By \cref{cor:Z-sZ-division}, the supercenter of a graded-division superalgebra is the even part of its center, and we know the centers of superalgebras of types $M \times M\sop$ and $Q \times Q\sop$ by \cref{prop:types-of-SA-via-center}.

    For \cref{it:squares-of-t}, let $t \in T$. 
    By \cref{cor:super-anti-auto-squares-in-radical}, $t^2 \in \rad \tilde\beta = \{e, f\}$. 
    Note that, by \cref{prop:superpolarization}, $\eta(t^2) = \tilde\beta(t, t)\eta(t)^2$. 
    If $t$ is even, then $\eta(t^2) = \eta(t)^2 = 1$, hence $t^2 = e$.
    If $t$ is odd, then $\eta(t^2) = -\eta(t)^2 = -1$, hence $t^2 = f$.
\end{proof}

As we will see in \cref{thm:tensor-prd-decomposition}, the converse of \cref{cor:T+-is-elem-2-grp} also holds.

\begin{cor}\label{cor:eta-parity-element}
    There exist exactly two parity elements $t_p, t_p' \in T$ (\cref{def:parity-element}), and they satisfy $\eta(t_p) = -\eta(t_p')$.
\end{cor}

\begin{proof}
    It follows from the fact that the set of parity elements is the coset $t_p (\rad \tilde\beta)$, so $t_p' = t_p f$ and, hence, $\eta(t_p f) = \tilde\beta(t_p, f) \eta(f)\eta(t_p) = -\eta(t_p)$.
\end{proof}

\begin{defi}\label{def:eta-parity-element}
    The unique parity element $t_p \in T$ such that $\eta(t_p) = 1$ is called the \emph{$\eta$-parity element} or the \emph{$\vphi_0$-parity element}.
    (Note that $t_p = e$ \IFF $T = T^+$.)
\end{defi}

We can refine \cref{cor:T+-is-elem-2-grp} to precisely identify the type of $(\D, \vphi_0)$.
The labels (a), (b) and (c) of the cases below will be used consistently for the rest of the section.

\begin{cor}\label{cor:rad-beta+-MxM-QxQ}
    Let $t_p$ be the $\eta$-parity element. Then:
    \begin{enumerate}[label=(\alph*)]
        \item If $(\D, \vphi_0)$ is even of type $M \times M\sop$, then $t_p = e$ and $\rad \beta^+ = \rad \beta = \langle f \rangle \iso \ZZ_2$.
        \label{item:even-MxM}
        \item If $(\D, \vphi_0)$ is odd of type $M \times M\sop$, then $e \neq t_p \in T^+$, $\rad \beta^+ = \langle f, t_p \rangle \iso \ZZ_2 \times \ZZ_2$ and $\rad \beta = \langle f \rangle \iso \ZZ_2$.
        \label{item:odd-MxM}
        \item If $(\D, \vphi_0)$ is of type $Q \times Q\sop$, then $t_p \in T^-$ and $\rad \beta^+ = \langle f \rangle \iso \ZZ_2$ and $\rad \beta = \langle t_p \rangle \iso \ZZ_4$.
        \label{item:QxQ}
    \end{enumerate}
\end{cor}

\begin{proof}
    This follows from \cref{prop:types-of-SA-via-center,cor:radical-with-parity,cor:T+-is-elem-2-grp}. 
\end{proof}

The next result will be used to simplify the parameters in \cref{subsubsec:grdd-simple-MxM-or-QxQ}.  
We recall the equivalence relation $\sim_\bz$ on the maps $T \to \FF^\times$ that determine superinvolutions on $\D$ (\cref{def:equiv-eta-even}).

\begin{lemma}\label{lemma:2-etas}
    Let $\vphi_0'$ be a superinvolution on $\D$, corresponding to $\eta'\from T \to \pmone$.
    If $(\D, \vphi_0')$ is super\-in\-vo\-lu\-tion-sim\-ple, then the $\vphi_0'$-parity element is the same as the $\vphi_0$-parity element \IFF $\eta' \sim_\bz \eta$.
    In particular, we have at most two $\sim_\bz$-equivalence classes.
\end{lemma}

\begin{proof}
    If $(\D, \vphi_0')$ is su\-per\-in\-vo\-lu\-tion-sim\-ple, then it must be of type $Q\times Q\sop$ or $M \times M\sop$.  
    Hence, by \cref{cor:T+-is-elem-2-grp}, we have $\eta'(f) = -1$.  
    Define $\chi \from T \to \pmone$ by $\chi(t) = \eta(t) \eta'(t)$ for all $t \in T$.  
    Since $\mathrm{d} \eta' = \tilde\beta = \mathrm{d} \eta$, we have $\mathrm{d} \chi = 1$, \ie, $\chi$ is a character of $T$.  
    Also, $\chi(f) = 1$.  
    Hence, $\chi$ induces a character of $T / \langle f \rangle$, and since $\tilde\beta$ induces a nondegenerate skew-symmetric bicharacter on $T / \langle f \rangle$, there exists $t_0 \in T$ such that $\chi = \tilde\beta(t_0, \cdot)$.  
    We conclude that $\eta' = \tilde\beta(t_0, \cdot) \eta$. 
    If $t_0\in T^+$, then $\eta' \sim_\bz \eta$ and $\eta'(t_p) = \tilde\beta(t_0, t_p)\eta(t_p) = 1$.
    If $t_0\not\in T^+$, then $\eta' \not\sim_\bz \eta$ and $\eta'(t_p) = \tilde\beta(t_0, t_p)\eta(t_p) = - 1$.
\end{proof}

The next result describes the general structure of the graded-division superalgebras we are studying.
In particular, it establishes the converse of \cref{cor:T+-is-elem-2-grp,cor:rad-beta+-MxM-QxQ}.

\begin{thm}\label{thm:tensor-prd-decomposition}
    Let $\D$ be a graded-division superalgebra associated to $(T, \tilde\beta)$ that is not simple as a superalgebra.
    Then there is a superinvolution $\vphi_0$ on $\D$ making it superinvolution-simple \IFF $\rad \tilde\beta = \langle f \rangle$ for an element $e \neq f \in T^+$, $T^+$ is an elementary $2$-group, and $t^2 = f$ for every $t \in T^-$.
    For any such $\vphi_0$, let $t_p \in T$ be the $\vphi_0$-parity element.
    Then there are superinvolution-simple graded-division superalgebras $(\mc C,\vphi_{\mc C})$ and $(\mc M,\vphi_{\mc M})$ such that $(\D, \vphi_0) \iso (\mc C \tensor \mc M, \vphi_{\mc C} \tensor \vphi_{\mc M})$, where $(\mc M,\vphi_{\mc M})$ is of type $M$ and:  
	\begin{enumerate}[label=(\alph*)]
        \item\label{it:even-MxM-sop} If $t_p = e$, then $(\mc C,\vphi_{\mc C}) \iso \FF\ZZ_2$, where $\FF\ZZ_2$ is the even graded-division superalgebra with superinvolution from \cref{ex:now-FZ2-is-division}, and $\ZZ_2$ is identified with $\rad \tilde\beta$.  
        In this case, $(\D, \vphi_0)$ is even of type $M\times M\sop$.  
	    \item\label{it:odd-MxM-sop} If $e \neq t_p \in T^+$, then $(\mc C,\vphi_{\mc C}) \iso \mc O$, where $\mc O$ is the odd graded-division superalgebra with superinvolution from \cref{ex:superalgebra-O}, and the standard generators of $\ZZ_2 \times \ZZ_4$ are identified with $t_p$ and an arbitrary element $t_1 \in T^-$.  
        In this case, $(\D, \vphi_0)$ is odd of type $M\times M\sop$.  
        \item\label{it:QxQ-sop} If $t_p \in T^-$, then $(\mc C,\vphi_{\mc C}) \iso \FF\ZZ_4$, where $\FF\ZZ_4$ is the odd graded-division superalgebra with superinvolution from \cref{ex:now-FZ4-is-division}, and the standard generator of $\ZZ_4$ is identified with $t_p$.
        In this case, $(\D, \vphi_0)$ is of type $Q\times Q\sop$.  
	\end{enumerate}
\end{thm}

\begin{proof}
    The ``only if'' direction of the first assertion is already established.
    Let us prove the ``if'' direction.
    Let $t_p \in T$ be a parity element, which exists by \cref{cor:t_p-exists}.
    
    Since $T^+$ is an elementary $2$-group, it is a vector space over the field with $2$ elements.  
    Fix a basis $\mc B$ for $\rad \beta^+$ and complete it to a basis $\mc B \cup \mc B'$ of $T^+$.  
    In the case where $e \neq t_p \in T^+$, for any $t_1 \in T^-$, we can choose $\mc B'$ such that $\tilde\beta(t_1, b) = 1$ for all $b \in \mc B'$.  
    Indeed, since $t_p \in \rad \beta^+$, if $\tilde\beta(t_1, b) = -1$, we can replace $b$ with $t_p b$.  
    
    Let $K \subseteq T^+$ be the subgroup generated by $\mc B'$.  
    Then $T^+ = (\rad \beta^+) \times K$, and hence $\beta^+\!\restriction_{K\times K}$ is a nondegenerate alternating bicharacter.
    Let $\mc M$ be a standard realization of type $M$ (\cref{def:standard-realization-M}) associated to $(K, \beta^+\!\restriction_{K\times K})$ and consider the involution $\vphi_\mc M$ defined by transposition (\cref{lemma:transp-std-realization}).
    
    Now define $C \subseteq T$ as follows:
    \begin{enumerate}[label=(\alph*)]
        \item If $t_p = e$, take $C \coloneqq \langle f \rangle \iso \ZZ_2$.
        \item If $e \neq t_p \in T^+$, take $C \coloneqq \langle t_p, t_1 \rangle \iso \ZZ_2 \times \ZZ_4$.
        \item If $t_p \in T^-$, take $C \coloneqq \langle t_p \rangle \iso \ZZ_4$.
    \end{enumerate}
    Let $\mc C$ be a graded-division superalgebra associated to $(C, \tilde\beta\!\restriction_{C \times C})$.
    It is easy to see that, in each case, $\mc C$ is as in the statement of the Theorem and, hence, it has a superinvolution $\vphi_\mc C$ (see \cref{ex:now-FZ2-is-division,ex:superalgebra-O,ex:now-FZ4-is-division}).
    
    By construction, we have $T = C \times K$ and $\tilde\beta(C, K) = 1$, and it is straightforward that this implies $\D \iso \mc C \tensor \mc M$.
    Therefore, the superinvolution $\vphi_\mc C \tensor \vphi_\mc M$ on $\mc C \tensor \mc M$ corresponds to a superinvolution $\vphi_0$ on $\D$.

    For the second assertion, let $\vphi_0$ be a superinvolution on $\D$ and $t_p$ the $\vphi_0$-parity element.
    Define the subgroups $K$ and $C$ as above and let $\mc M$ and $\mc C$ be the corresponding graded subsuperalgebras of $\D$.
    Then $\D \iso \mc C \tensor \mc M$ and $\vphi_0$ restricts to each factor.
    The restriction to $\mc C$ is as in the statement.
\end{proof}

Note that, in cases \ref{it:even-MxM-sop} and \ref{it:QxQ-sop} in the proof above, we chose $K$ to be an arbitrary complement of $\rad \beta^+$ in $T^+$. 
The next result shows that we can proceed similarly in case \ref{it:odd-MxM-sop}, and then choose the element $t_1 \in T^-$ depending on $K$.

\begin{lemma}\label{lemma:K-before-t_1}
    Let $(T, \tilde\beta)$ be as in \cref{thm:tensor-prd-decomposition}\ref{it:odd-MxM-sop}, and let $K \subseteq T^+$ be any subgroup such that $T^+ = (\rad \beta^+) \times K$. 
    Then $S \coloneqq \{ t_1 \in T^- \mid \tilde\beta(t_1, K) = 1 \}$ is a coset of $\rad \beta^+$. 
\end{lemma}

\begin{proof}
    Given $t_1 \in S$, it is straightforward that $t_1' \in S$ \IFF $t_1\inv t_1' \in \rad \beta^+$. 
    Hence we only have to show that $S \neq \emptyset$. 
    Let $\chi$ be the character of $T^+$ determined by $\chi(K) = \chi(f) = 1$ and $\chi(t_p) = -1$, and extend $\chi$ to a character of $T$. 
    Since $\rad \tilde\beta = \langle f \rangle$, $\chi$ induces a character on $T/\rad \tilde\beta$ and, by the nondegeneracy of the bicharacter induced by $\tilde \beta$ on $T/\rad \tilde\beta$, there is $t_1 \in T$ such that $\chi = \tilde\beta(t_1, \cdot)$. 
    Since $\tilde\beta(t_1, t_p) = \chi(t_p) = -1$, $t_1\in T^-$, and hence, since $\tilde\beta(t_1, K) = \chi(K) = 1$, $t_1 \in S$. 
\end{proof}

In \cref{subsubsec:standard-realizations}, we defined standard realizations of types $M$ and $Q$.
We will now define standard realizations of types $M\times M\sop$ and $Q \times Q\sop$.

\begin{defi}\label{def:std-realization-MxM-QxQ}
    Let $T \subseteq G^\#$ be a finite subgroup and let $\tilde\beta\from T\times T \to \pmone$ be a skew-symmetric bicharacter such that $\rad \tilde \beta = \langle f \rangle$ for an element $e \neq f\in T^+$, $T^+$ is an elementary $2$-group, and $t^2 = f$ for every $t\in T^-$.
    It follows that we have precisely $2$ parity elements in $T$; 
    let $t_p \in T$ be one of them, taking $t_p \coloneqq e$ in the case $T = T^+$. 
    Then choose:
    \begin{enumerate}
        \item a subgroup $K \subseteq T^+$ such that $(\rad \beta^+) \times K = T^+$;
        \label{item:K-can-be-orthogonal-to-t_1}
        \item a standard realization $\mc M$ (see \cref{def:standard-realization-M}) of type $M$ associated to $(K, \tilde\beta\!\restriction_{K\times K})$;
        \label{item:choose-mc-M}
        \item if $e\neq t_p \in T^+$, an element $t_1 \in T^-$ such that $\beta(t_1, K) = 1$ (which exists by \cref{lemma:K-before-t_1}). 
        \label{item:choose-t_1-std-realization}
    \end{enumerate}
    Let $\vphi_{\mc M}$ denote the transposition on $\mc M$, and let $(\mc C, \vphi_{\mc C})$ be as in \cref{thm:tensor-prd-decomposition}.
    We say that $(\mc C \tensor \mc M, \vphi_{\mc C} \tensor \vphi_{\mc M})$ is a \emph{standard realization of a superinvolution-simple graded-division superalgebra} (of type $M\times M\sop$ or $Q \times Q\sop$) associated to $(T, \tilde\beta, t_p)$.
    Note that, as a superalgebra, $\mc C \tensor \mc M$ is isomorphic to:
    \begin{enumerate}[label=(\alph*)]
        \item $M_\ell(\FF)\times M_\ell(\FF)\sop$ where $\ell = \sqrt{|T|/2}$, if $t_p =e$;
        \item $M(\ell,\ell) \times M(\ell,\ell)\sop$ where $\ell = \sqrt{|T|/8}$, if $e \neq t_p \in T^+$;
        \item $Q(\ell)\times Q(\ell)\sop$ where $\ell = \sqrt{|T|/4}$, if $t_p\in T^-$.
    \end{enumerate}
\end{defi}

\begin{remark}\label{rmk:eta-for-MxM-or-QxQ}
    The map $\eta\from T \to \pmone$ corresponding to $ \vphi_{\mc C} \tensor \vphi_{\mc M}$ is characterized by the following conditions: $\mathrm{d}\eta = \tilde\beta$, $\eta\!\restriction_{K}$ is the map associated to the transposition on $\mc M$ and one of 
    \begin{enumerate}[label=(\alph*)]
        \item\label{it:std-realization-(a)} if $t_p =e$, then $\eta(f) = -1$;
        \item if $e \neq t_p \in T^+$, then $\eta(t_p) = 1$ and $\eta(t_1) = 1$;
        \item\label{it:std-realization-(c)} if $t_p \in T^-$, then $\eta(t_p) = 1$.
    \end{enumerate} 
    In particular, in all cases $\eta(f) = -1$ and the chosen $t_p$ is the $\eta$-parity element of $T$. 
    Also, a different choice $t_1'$ in \cref{item:choose-t_1-std-realization} leads to $\eta' = \eta$ if $t_1' \in \{ t_1, f t_p t_1\}$, and to $\eta' = \tilde\beta(t_p, \cdot) \eta$ if $t_1' \in \{ f t_1, t_p t_1\}$; 
    in both cases, $\eta'\restriction_{T^+} = \eta\restriction_{T^+}$.
\end{remark}

\phantomsection\label{phsec:std-realization-QxQsop}

In \cref{subsubsec:standard-realizations}, we presented distinct standard realizations for types $M$ and $Q$, where the standard realizations for $Q$ were defined exclusively in terms of the group $G$.  
We can do the same for type $Q \times Q\sop$.  

Let $(\D, \vphi_0)$ be of type $Q \times Q\sop$.  
There exists an invertible homogeneous element $w \in Z(\D)\odd$ of degree $t_p$ such that $\vphi_0(w) = w$.  
Then $\D = \D\even \oplus w \D\even$. 
Scaling $w$ if necessary, we may assume $w^2 = \zeta$, where $\zeta \coloneqq (1, -\bar{1}) \in Z(\D)\even$.  
Note that $(\D\even, \vphi_0\!\restriction_{\D\even})$ is an even graded-division superalgebra of type $M \times M\sop$. 
Let $h \in G$ denote the projection of $t_p$ to $G$ (\ie, $t_p = (h, \bar{1})$).  
Then, $h^2 = f$ and $(T^+, \beta^+, e)$ are the parameters of $\D\even$.

Conversely, given $(\D\even, \vphi_0)$ even of type $M \times M\sop$ associated to $(T^+, \beta^+, e)$ and an element $h \in G$ such that $h^2 = f$, let $w$ be a new symbol of $G$-degree $h$ that commutes with every element of $\D\even$ and satisfies $w^2 = \zeta \in \D\even$.
Define $\D\odd \coloneqq w \D\even$, and extend $\vphi_0$ by setting $\vphi_0(w) = w$.  
One can check that the resulting superalgebra with superinvolution is a graded-division superalgebra of type $Q \times Q\sop$.
The parameters of $(\D, \vphi_0)$ are $(T, \tilde\beta, t_p)$, where $t_p \coloneqq (h, \bar{1})$, $T \coloneqq T^+ \cup t_p T^+$, and $\tilde\beta\from T \times T \to \FF^\times$ is given by $\tilde\beta(s t_p^i, t t_p^j) = (-1)^{ij} \beta^+(s,t)$ for all $s, t \in T^+$ and $i, j \in \ZZ$.   

As an alternative to \cref{def:std-realization-MxM-QxQ} in the case \ref{it:std-realization-(c)}, we will use the following:

\begin{defi}\label{def:std-realization-QxQsop}
    Let $T^+$ be an elementary $2$-group, $\beta^+$ be an alternating bicharacter with $\rad \beta^+ = \langle f\rangle$ for $e \neq f \in T^+$, and an element $h\in G$ such that $h^2 = f$.
    Let $\D\even$ be a standard realization of type $M\times M\sop$ associated to $(T^+, \beta^+, e)$.
    Then $\D\even \oplus w\D\even$ with $\vphi_0$ as above will be called a \emph{standard realizations of type $Q\times Q\sop$} associated to $(T^+, \beta^+, h)$.
    Note that $\D \iso Q(\ell) \times Q(\ell)$ where $\ell = \sqrt{|T^+|/2}$.
\end{defi}  
 

\subsubsection{Graded-simple superalgebras of types \texorpdfstring{$M\times M\sop$}{MxMsop} and \texorpdfstring{$Q\times Q\sop$}{QxQsop}}\label{subsubsec:grdd-simple-MxM-or-QxQ}

Let $(\D, \U, B)$ be a triple as in \cref{def:E(D-U-B)}, with parameters $(T, \tilde\beta, \eta, \kappa, g_0, \delta)$.
We now consider the case where $(R, \vphi) \coloneqq E(\D, \U, B)$ is a superalgebra with superinvolution of type $M\times M\sop$ or $Q\times Q\sop$.
Since $(R, \vphi)$ is superinvolution-simple but not simple, \cref{prop:simple-R-D-super,prop:vphi-R-simple-D-simple} imply that $(\D, \vphi_0)$ is also superinvolution-simple but not simple. 
By \cref{prop:R-and-D-have-the-same-center}, $Z(R)$ is isomorphic to $Z(\D)$, and by \cref{prop:types-of-SA-via-center}, different types have nonisomorphic centers, hence $(R, \vphi)$ has the same type as $(\D, \vphi_0)$.
It follows that the type of $(R, \vphi)$ can be recognized from $(T, \tilde\beta)$ via \cref{cor:rad-beta+-MxM-QxQ}.

We will define explicit models for $E(\D, \U, B)$ where $\D$ is (a) even of type $M\times M\sop$, (b) odd of type $M\times M\sop$, or (c) of type $Q\times Q\sop$, and proceed to classify them up to isomorphism.
To this end, we will specialize simplified parametrization defined in \cref{subsec:simplify-parameters} to these cases.
We start by noting that, by \cref{lemma:2-etas}, in the even case we have a unique equivalence class of $\sim_\bz$, and in the odd cases the equivalence classes of $\sim_\bz$ are in bijection with the set of parity elements $T_p \subseteq T$.
Moreover, since we have $\rad\tilde\beta = \langle f \rangle$, where $f\in T^+$ is such that $\eta(f) = -1$, note that the group $\mc G_\eta$ defined in \cref{eq:mathcal-G} is
\[\label{eq:mathcal-G-MxM-QxQ}
    \mc G_\eta = \big( \{e\} \times (G \times \{ \bar 0 \}) \big) \cup \big( \{f\} \times (G \times \{ \bar 1 \}) \big) \subseteq \rad \tilde \beta \times G^\#\,,
\]
which does not depend on $\eta$.

In what follows, for each triple $(T, \tilde\beta, t_p)$, we fix a standard realization $(\D, \vphi_0)$ as in \cref{def:std-realization-MxM-QxQ}, and for each triple $(T^+, \beta^+, h)$, we fix a standard realization $(\D, \vphi_0)$ as in \cref{def:std-realization-QxQsop}.
In particular, this fixes the map $\eta\from T \to \FF^\times$ corresponding to $\vphi_0$.
The results below are a consequence of \cref{thm:vphi-iff-vphi0-and-B,thm:vphi-involution-iff-delta-pm-1,cor:collection-of-orbits}.


\begin{defi}\label{def:model-grd-MxM-even}
    Let $T$ and $\tilde\beta$ be as in \cref{def:std-realization-MxM-QxQ} with $T^+ = T$ (so $\beta = \tilde\beta$ and $t_p = e$), let $g_0 \in G^\#$ be any element, and let $(\kappa_\bz, \kappa_\bo)$ be a $g_0$-admissible pair of maps $G/T \to \ZZ_{\geq 0}$ (see \cref{inertia-even-and-odd-case}). 
    Construct the corresponding pair $(\U, B)$ as in \cref{subsubsec:construction-U-B}.
    The graded superalgebra with superinvolution $M^{\mathrm{ex}}(T, \beta, \kappa_\bz, \kappa_\bo, g_0)$ is defined to be $E(\D, \U, B)$. 
    By choosing graded bases for $\U\even$ and $\U\odd$, this becomes
    the graded superalgebra $M_{k_\bz | k_\bo} (\D)$, with $k_\bz \coloneqq |\kappa_\bz|$ and $k_\bo \coloneqq |\kappa_\bo|$, endowed with the superinvolution $\vphi$ given by 
    \[
        \forall X \in M_{k_\bz | k_\bo}(\D), \quad \vphi(X) \coloneqq \Phi\inv \vphi_0(X\stransp) \Phi,
    \]
    where $\Phi$ is the matrix representing $B$ with respect to the chosen bases. 
    Note that $M^{\mathrm{ex}}(T, \beta, \kappa_\bz, \kappa_\bo, g_0) \iso M(m,n)\times M(m,n)\sop$ as a superalgebra with superinvolution, where $m \coloneqq k_\bz \sqrt{|T|/2}$ and $n \coloneqq k_\bo \sqrt{|T|/2}$.
\end{defi}

\begin{rmk}\label{rmk:g_0-odd-implies-m=n}
    If $m\neq n$, then $g_0\in G^\#$ is necessarily even (see \cref{para:even-D-odd-B}).
\end{rmk}

\begin{thm}\label{thm:MxM-even}
    Let $M(m,n) \times M(m,n)\sop$, as superalgebra with superinvolution, be endowed with an even $G$-grading making it graded-simple. 
    Then it is isomorphic to a graded superalgebra with superinvolution $M^{\mathrm{ex}}(T,\beta, \kappa_\bz, \kappa_\bo, g_0)$ as in \cref{def:model-grd-MxM-even}. 
    Moreover, $M^{\mathrm{ex}} (T, \beta, \kappa_\bz, \kappa_\bo, g_0) \iso M^{\mathrm{ex}} (T', \beta', \kappa_\bz', \kappa_\bo', g_0')$ \IFF $T =T'$, $\beta = \beta'$ and there is $g \in G$ such that one of the following conditions holds:
    \begin{enumerate}[label=(\roman*)]
        \item $\kappa_\bz' = g\cdot\kappa_\bz$, $\kappa_\bo' = g\cdot\kappa_\bo$ and $g_0' = g^{-2}g_0$;
        \item $\kappa_\bz' = g\cdot\kappa_\bo$, $\kappa_\bo' = g\cdot\kappa_\bz$ and $g_0' = fg^{-2}g_0$. \qed
    \end{enumerate}
\end{thm}


We will now consider odd gradings.

\begin{defi}\label{def:model-grd-MxM-odd-or-QxQ}
    Let $T$, $\tilde\beta$, and $t_p$ be as in \cref{def:std-realization-MxM-QxQ} with $T^+ \neq T$, let $g_0 \in G$ be an arbitrary element, and let $\kappa\from G/T^+ \to \ZZ_{\geq 0}$ be a $g_0$-admissible map (\cref{defi:odd-D-kappa-g_0-admissible}).
    Construct the corresponding pair $(\U, B)$ as in \cref{subsubsec:construction-U-B}. 
    The graded superalgebra with superinvolution $E(\D, \U, B)$ is denoted by $M^{\mathrm{ex}}(T, \tilde\beta, t_p, \kappa, g_0)$ if $t_p \in T^+$, and by $Q^{\mathrm{ex}}(T, \tilde\beta, t_p, \kappa, g_0)$ if $t_p \in T^-$. 
    By choosing an even graded basis for $\U$, $E(\D, \U, B)$ becomes $M_k(\D)$, where $k \coloneqq |\kappa|$, endowed with the superinvolution $\vphi$ given by 
    \[
        \forall X \in M_k(\D),\quad \vphi(X) \coloneqq \Phi^{-1} \vphi_0(X^{\transp}) \Phi,
    \]
    where $\Phi$ is the matrix representing $B$ with respect to the chosen basis.
    Note that $M^{\mathrm{ex}}(T, \tilde\beta, t_p, \kappa, g_0) \iso M(n,n)\times M(n,n)\sop$, as a superalgebra with superinvolution, where $n \coloneqq k \sqrt{|T|/8}$, and $Q^{\mathrm{ex}}(T, \tilde\beta, t_p, \kappa, g_0) \iso Q(n)\times Q(n)\sop$, where $n \coloneqq k \sqrt{|T|/4}$.
\end{defi}

\begin{thm}\label{thm:MxM-odd}
    Let $M(n,n) \times M(n,n)\sop$ (respectively, $Q(n) \times Q(n)\sop$), as a superalgebra with superinvolution, be endowed with an odd $G$-grading that makes it graded-simple. 
    Then it is isomorphic to a graded superalgebra with superinvolution $M^{\mathrm{ex}}(T,\tilde\beta, t_p, \kappa, g_0)$ (resp., $Q^{\mathrm{ex}}(T, \tilde\beta, t_p, \kappa, g_0)$) as in \cref{def:model-grd-MxM-odd-or-QxQ}. 
    Moreover, the graded superalgebras with superinvolution $M^{\mathrm{ex}}(T, \tilde\beta,  t_p, \kappa, g_0)$ and $M^{\mathrm{ex}}(T', \tilde\beta',  t_p', \kappa', g_0')$ (resp., $Q^{\mathrm{ex}}(T, \tilde\beta,  t_p, \kappa, g_0)$ and $Q^{\mathrm{ex}} (T', \tilde\beta',  t_p', \kappa', g_0')$) are isomorphic \IFF $T =T'$, $\tilde\beta = \tilde\beta'$, $t_p = t_p'$ and there exists $g \in G$ such that $\kappa' = g\cdot\kappa$ and $g_0' = g^{-2}g_0$. \qed
\end{thm}

For superalgebras of type $Q \times Q\sop$, we can use the standard realization defined solely in terms of $G \subseteq G^\#$ (\cref{def:std-realization-QxQsop}).
We could also classify the odd gradings on superalgebras of type $M \times M\sop$ solely in terms of $G$, but, in the same way as in \cref{phsec:odd-param-in-terms-of-G}, we exclude it from this work for conciseness; the interested reader can find it in \cite[Corollary~4.76]{caios_thesis}.

\begin{defi}\label{def:model-grd-QxQ-only}
    Given $T^+$, $\beta^+$ and $h$ as in \cref{def:std-realization-QxQsop}, consider the corresponding $T$, $\tilde\beta$ and $t_p$.
    We define $Q^{\mathrm{ex}} (T^+, \beta^+,  h, \kappa, g_0)$ to be the graded superalgebra with superinvolution $Q^{\mathrm{ex}}(T, \tilde\beta,  t_p, \kappa, g_0)$ as in \cref{def:model-grd-MxM-odd-or-QxQ}.
\end{defi}

Note that $Q^{\mathrm{ex}} (T^+, \beta^+,  h, \kappa, g_0) = M^{\mathrm{ex}}(T^+, \beta^+, \kappa, g_0) \oplus w M^{\mathrm{ex}}(T^+, \beta^+, \kappa, g_0)$, where $w^2 = (1,-\bar 1)\in M^{\mathrm{ex}}(T^+, \beta^+, \kappa, g_0)$, $\vphi(w) = w$ (see \cref{prop:R-and-D-have-the-same-center-vphi}), and the $G^\#$-degree of $w$ is $t_p = (h, \bar 1)$.

By \cref{thm:MxM-odd}, we have:

 \begin{cor}\label{cor:QxQ-reduced-to-MxM}
    Let $Q(n) \times Q(n)\sop$, as a superalgebra with superinvolution, be endowed with a $G$-grading that makes it graded-simple.
    Then it is isomorphic to a graded superalgebra with superinvolution $Q^{\mathrm{ex}}(T^+,\beta^+, h, \kappa, g_0)$ as in \cref{def:model-grd-QxQ-only}. 
    Moreover, the graded superalgebras with superinvolution $Q^{\mathrm{ex}} (T^+, \beta^+,  h, \kappa, g_0)$ and $Q^{\mathrm{ex}} (T'^+, \beta'^+,  h', \kappa', g_0')$ are isomorphic \IFF $T^+ =T'^+$, $\beta^+ = \beta'^+$, $h = h'$ and there is $g \in G$ such that $\kappa' = g\cdot\kappa$ and $g_0' = g^{-2}g_0$. \qed
\end{cor} 

\section{Gradings on Classical Lie Superalgebras}\label{sec:Lie}

We finally turn to Lie superalgebras over an algebraically closed field \(\FF\) of characteristic zero.
In \cref{subsec:aut-and-transfer}, we show how the classification of \(G\)‑gradings on finite‑dimensional superinvolution‑simple associative superalgebras carries over to nonexceptional classical Lie superalgebras, excluding type \(A(1,1)\).
In the following subsections, we adapt the classification from the superinvolution‑simple associative superalgebras to each family of nonexceptional classical Lie superalgebras.

\subsection{Automorphism groups and transfer of gradings}\label{subsec:aut-and-transfer}

\subsubsection{Gradings and actions}\label{subsec:grds-and-actions}

We will now review the well-known correspondence between gradings by an abelian group $G$ and actions by the group of (multiplicative) characters $\widehat G \coloneqq \Hom(G, \FF^\times)$.
Let $A$ be a finite-dimensional algebra equipped with any number of multilinear operations $A^{\tensor n} \to A$, and let $\Gamma\from A = \bigoplus_{g\in G} A_g$ be a $G$-grading. 
(For a detailed treatment of gradings in this setting, see \cite[Section 1.2]{caios_thesis}; in this work, we focus on algebras with a multiplication $A \tensor A \to A$, a superalgebra structure defined by the operation $A \to A$ that sends every element to its even part, and possibly a superinvolution $A \to A$.)  
We define a $\widehat G$-action on $A$ by
\[
    \forall \chi \in \widehat G,\  g \in G,\  a \in A_g, \quad \chi \cdot a \coloneqq \chi(g) a\,.
\] 
It is straightforward the elements of $\widehat G$ act by automorphisms of $A$.

If $G$ is finitely generated, then $\widehat G$ is an algebraic group and the action above is an algebraic group action.
If, furthermore, $\FF$ is algebraically closed with $\Char \FF = 0$, we have a bijection between $G$-gradings on $A$ and algebraic $\widehat G$-actions by automorphisms.
As a consequence, if $A$ and $B$ are finite-dimensional algebras (possibly with different sets of multilinear operations) satisfying $\Aut(A) \iso \Aut(B)$ as algebraic groups, then there exists a bijection between the sets of $G$-gradings on $A$ and on $B$, which induces a bijection between the isomorphism classes of gradings.

We will use this correspondence to transfer our classification of gradings from su\-per\-in\-vo\-lu\-tion-sim\-ple associative superalgebras to non-exceptional classical Lie superalgebras different from $A(1,1)$.

\subsubsection{Automorphisms of (superinvolution-)simple associative superalgebras}\label{subsubsec:aut-associative}

We start by listing the automorphism groups of the associative superalgebras $M(m,n)$ and $Q(n)$, which are well known and can be computed from \cref{thm:iso-abstract}.

Every automorphism of $M(m,n)$ is conjugation by an invertible homogeneous element.  
We define $\mc E(m,n)$ as the normal subgroup consisting of the conjugations by \emph{even} elements, \ie, elements of the form
\[\label{eq:even-matrix}
\begin{pmatrix}
    x & 0 \\
    0 & y
\end{pmatrix}, \text{ with } x \in \GL_m \AND y \in \GL_n.
\]
Thus, we have \[
    \mc E(m,n) \iso (\GL_m \times \GL_n) / \FF^\times,
\]
where $\FF^\times$ is identified with scalar matrices. 
Odd invertible elements in $M(m,n)$ exist \IFF $m=n$; they are of the products of 
\[\label{eq:pi-is-conjugation-by}
    \Pi \coloneqq \begin{pmatrix}
    0 & I_n \\
    I_n & 0
\end{pmatrix}
\]
with even invertible elements.
The conjugation by $\Pi$ is the so-called \emph{parity transpose} and will be denoted by $\pi$.
Therefore,
\[
    \Aut(M(m,n)) =
    \begin{cases}
        \mc E(m,n) & \text{if } m \neq n;\\
        \mc E(n,n) \rtimes \langle \pi \rangle & \text{if } m = n.
    \end{cases}
\]

The automorphism group of $Q(n)$ is generated by the conjugations by invertible even elements of $Q(n)$, which are matrices of the form 
$\begin{pmatrix}
    x & 0\\
    0 & x
\end{pmatrix}$ 
with $x \in \GL_{n}$, and the parity automorphism $\nu\from Q(n) \to Q(n)$. 
Thus, we have
\[
\Aut(Q(n)) = \Aut(M_n(\FF)) \times \langle \nu \rangle \iso \operatorname{PGL}_n \times \operatorname C_2\,.
\]

Now consider $R = S \times S\sop$, where $S$ is either $M(m,n)$ or $Q(n)$, and let $\vphi\from R \to R$ be the exchange superinvolution. 
The automorphisms of $(R, \vphi)$ can be computed using \cref{lemma:iso-SxSsop} and are either of the form
\[\label{eq:psi_theta}
    \forall x,y \in S, \quad \psi_\theta (x, \bar{y}) \coloneqq (\theta(x), \overline{\theta(y)}),
\]
where $\theta\from S \to S$ is an automorphism, or of the form
\[\label{eq:psi_omega}
    \forall x,y \in S, \quad \psi_\sigma (x, \bar{y}) \coloneqq (\sigma(y), \overline{\sigma(x)}),
\]
where $\sigma\from S \to S$ is a super-anti-automorphism.
It is straightforward to verify that \cref{eq:psi_theta,eq:psi_omega} define an isomorphism $\overline{\Aut}(S) \to \Aut(R, \vphi)$, where $\overline{\Aut}(S)$ is the group of automorphisms and super-anti-automorphisms of $S$. 
Note that unless $S = \FF$, $\Aut(S)$ is a subgroup of index $2$ in $\overline{\Aut}(S)$; the extension is split for type $M$ and non-split for type $Q$ (\cref{cor:Q-no-sinv-center}).
More precisely,
\[\label{eq:barr-Aut-S}
    \overline{\Aut}(S) \iso
    \begin{cases}
        \mc E(m,n) \rtimes \operatorname C_2 & \text{if } S= M(m,n) \text{ with } m\neq n;\\
        (\mc E(n,n) \rtimes \operatorname C_2) \rtimes \operatorname C_2 & \text{if } S= M(n,n);\\
        \operatorname{PGL}_n \times \operatorname C_4 & \text{if } S= Q(n).
    \end{cases}
\]

Finally, let us consider the case $(R, \vphi) = M^*(m,n, p_0)$.
The automorphism groups can be derived from \cref{thm:iso-abstract-vphi}.
We have that $\Aut(M^*(m,n, p_0)) \subseteq \mc E(m,n)$, \ie, it consists of conjugations by matrices as in \cref{eq:even-matrix}.
For $M^*(m,n, \bar 0)$, we have that $x \in \operatorname{O}_m$ and $y \in \operatorname{Sp}_n$, thus
\[
    \Aut(M^*(m,n, \bar 0)) \iso \frac{\operatorname{O}_m \times \operatorname{Sp}_n}{\{\pm I_{m+n}\}}.
\] 
For $M(n,n, \bar 1)$, we have that $x \in \operatorname{GL}_n$ can be any element, but then $y = (x\transp)\inv$, thus
\[
    \Aut(M^*(m,n, \bar 1)) \iso \frac{\operatorname{GL}_n}{\{ \pm I_n\}}\,. 
\]

\subsubsection{Automorphisms of classical Lie superalgebras and types of gradings}\label{subsubsec:type-I-type-II}

The (outer) automorphism groups of the finite-dimensional simple Lie superalgebras were computed in \cite{serganova}.
We will compare them with those presented in \cref{subsubsec:aut-associative} to use the transfer strategy presented in \cref{subsec:grds-and-actions}.
Note that the algebraic groups of automorphisms for these Lie superalgebras appear in \cite{Pianzola}, whose description aligns more closely with ours.

Let $L$ be a finite-dimensional simple Lie superalgebra in one of the series $B$, $C$, $D$ or $P$.
By definition (see \cref{subsec:def-Lie-superalgebras}), $L = \Skew(R, \vphi)^{(1)}/Z(\Skew(R, \vphi)^{(1)})$ where $(R, \vphi) = M^*(m,n,p_0)$ for some $m,n \in \ZZ_{\geq 0}$ and $p_0\in \ZZ_2$.
Every automorphism $\psi\from (R, \vphi) \to (R, \vphi)$ restricts to an automorphism of $\Skew(R, \vphi)$, which in turn induces an automorphism on $\Skew(R, \vphi)^{(1)}/Z(\Skew(R, \vphi)^{(1)})$.
This yields an algebraic group homomorphism $\Aut(R, \vphi) \to \Aut(L)$.
Comparing the automorphism group of $(R, \vphi)$ to the one of $L$ (\cite[Theorem 1]{serganova} and \cite[Theorem 4.1]{Pianzola}), we see that the map $\Aut(R, \vphi) \to \Aut(L)$ above is an isomorphism.

Now consider the case that $L$ is a finite-dimensional simple Lie superalgebra in either series $A$ or $Q$.
Then $L = S^{(1)}/Z(S^{(1)})$, where $S$ is a finite-dimensional associative superalgebra of type $M$ or $Q$.
Analogous to the previous case, every automorphism $\psi\from S \to S$ restricts to an automorphism of $S^{(1)}$, which in turn induces an automorphism of $S^{(1)}/Z(S^{(1)})$.
This yields an algebraic group homomorphism $\Aut(S) \to \Aut(L)$.
However, comparing the automorphism group of $S$ with the automorphism group of $L$, we see that in this case the homomorphism $\Aut(S) \to \Aut(L)$ is injective but not surjective.

To recover all automorphisms of $L$, we can use another model.
Consider $R \coloneqq S \times S\sop$ and denote the exchange superinvolution by $\vphi$.
The map $S^{(-)} \to \Skew(R, \vphi)$ sending $s$ to $(s, -\barr s)$ is an isomorphism of Lie superalgebras that induces an isomorphism $L \to \Skew(R, \vphi)^{(1)}/Z(\Skew(R, \vphi)^{(1)})$.
Now, every automorphism $\psi\from (R, \vphi) \to (R, \vphi)$ induces an automorphism on $\Skew(R, \vphi)^{(1)}/Z(\Skew(R, \vphi)^{(1)})$, and if $L \not\iso A(1,1)$, this yields an isomorphism of algebraic groups $\Aut(R, \vphi) \to \Aut(L)$.

\begin{cor}\label{cor:bijection-gradings}
    Let $(R, \vphi)$ be a finite-dimensional simple associative superalgebra with superinvolution over an algebraically closed field of characteristic zero, and suppose $L \coloneqq \Skew(R, \vphi)^{(1)}/Z(\Skew(R, \vphi)^{(1)})$ is simple of type different from $A(1,1)$.
    The $G$-gradings on $(R, \vphi)$ and on $L$ are in bijection via the map that restricts a grading on $(R, \vphi)$ to $\Skew(R, \vphi)^{(1)}$ and then passes to the quotient modulo the center.
    This map also induces a bijection between the isomorphism classes of $G$-gradings. \qed
\end{cor}

Hence, the classifications of $G$-gradings are the same for $(R, \vphi)$ and $L$, and the former are classified in \cref{sec:gradings-on-vphi-simple}.
For completeness, we will slightly adapt this classification to the Lie case and present it in the following subsections.
It is convenient to divide the gradings into two types for Lie superalgebras in series $A$ and $Q$.

\begin{prop}
    Let $S$ be an associative superalgebra of type $M$ or $Q$, and set $(R, \vphi) = S\times S\sop$.
    A grading on $S^{(1)}/Z(S^{(1)}) \iso \Skew(R, \vphi)^{(1)}/Z(\Skew(R, \vphi)^{(1)})$ is induced from a grading on $S$ \IFF it is induced from a grading on $(R, \vphi)$ such that $R$ is not graded-simple.
\end{prop}

\begin{proof}
    If $R$ is not graded-simple, then by \cref{prop:only-SxSsop-is-simple}, the superideals  $S\times \{0\}$ and $\{0\}\times S\sop$ are both graded, so the projection $R \to S$ is a homomorphism of graded superalgebras.
    The restriction of this projection to $\Skew(R, \vphi)$ is the inverse of the isomorphism $S^{(-)} \to \Skew(R, \vphi)$, $s \mapsto (s, -\barr s)$.
    The result follows.
\end{proof}

\begin{defi}\label{defi:type-I-type-II}
    Let $S$ be an associative superalgebra of type $M$ or $Q$, and set $(R, \vphi) = S\times S\sop$.
    A grading on $L \coloneqq S^{(1)}/Z(S^{(1)})$ is said to be of \emph{Type I} if it is induced by a grading on $S$, and of \emph{Type II} if it is induced by a grading on $(R, \vphi)$ such that $R$ is graded-simple.
\end{defi}

It follows that for a simple Lie superalgebra in series $A$ or $Q$ except $A(1,1)$, every grading is either of Type I or Type II.
In this work, Type II gradings will only be described using the model derived from $R = S\times S\sop$.
They can also be described in terms of the defining model $L = S^{(1)}/Z(S^{(1)})$, but this approach is more involved.
The interested reader may consult \cite[Section 5.3]{caios_thesis} for details.
For type $A(1,1)$, see \cite{gradingsExceptional}.

\subsection{Orthosymplectic Lie superalgebras (\texorpdfstring{Series $B$, $C$ and $D$}{Series B, C and D})}\label{sec:grds-osp-Lie}

The gradings on $\osp(m|n)$ correspond to the gradings on $M^*(m, n, \bar 0)$. 
By \cref{thm:osp-and-p-associative}, if we endow $M^*(m, n, \bar 0)$ with a grading, then it becomes isomorphic to $(R, \vphi) \coloneqq M^*(T, \beta, \kappa_\bz, \kappa_\bo, g_0)$ as in \cref{def:model-grd-M(m-n-0)}, with $|g_0| = \bar 0$. 
We define $\osp(T, \beta, \kappa_\bz, \kappa_\bo, g_0)$ to be the graded Lie superalgebra $\Skew(R, \vphi)$. 
The following result is a direct consequence of \cref{thm:osp-and-p-associative}:

\begin{thm}\label{thm:grds-osp-final}
    Let $L$ be an orthosymplectic simple Lie superalgebra endowed with any $G$-grading. 
    Then $L$ is isomorphic to some $\osp(T, \beta, \kappa_\bz, \kappa_\bo, g_0)$. 
    Moreover, $\osp(T, \beta, \kappa_\bz, \kappa_\bo, g_0) \iso \osp(T', \beta', \kappa_\bz', \kappa_\bo', g_0')$ \IFF $T =T'$, $\beta = \beta'$ and there is an element $g \in G$ such that $\kappa_\bz' = g\cdot\kappa_\bz$, $\kappa_\bo' = g\cdot\kappa_\bo$ and $g_0' = g^{-2}g_0$. \qed
\end{thm}

It should be noted that, if we restrict ourselves only to the series $B$, \ie, the superalgebras $\osp(m,n)$ for which $m$ is odd, since $T$ is an elementary $2$-group and $m = k_0\sqrt{|T|}$, we must have that $T$ is trivial.
Also, in this case, $g_0$ is a square, so we can make $g_0 = e$.
Gradings for series $B$ were classified in \cite{Helens_thesis}.

\subsection{Periplectic Lie superalgebras (Series \texorpdfstring{$P$}{P})}\label{sec:grds-P}

Gradings on the Lie superalgebras in the series $P$ were classified in \cite{paper-MAP} using a different method, but we can also obtain them in a manner similar to what we did for series $B$, $C$, and $D$.
Gradings on $P(n-1)$ correspond to the gradings on $M^*(n, n, \bar 1)$. 
By \cref{thm:osp-and-p-associative}, if we endow $M^*(n, n, \bar 1)$ with a grading, then it becomes isomorphic to $(R, \vphi) \coloneqq M^*(T, \beta, \kappa_\bz, \kappa_\bo, g_0)$ as in \cref{def:model-grd-M(m-n-0)}, with $g_0 = (h_0, \bar 1)$, $h_0 \in G$. 
Recall from \cref{para:even-D-odd-B} that, in this case, $\kappa_\bo$ is determined by $\kappa_\bz$, so we denote by $P(T, \beta, \kappa_\bz, h_0)$ the graded Lie superalgebra $\Skew(R, \vphi)^{(1)}$. 

The following result follows from \cref{thm:osp-and-p-associative}:

\begin{thm}\label{thm:last-one-for-P}
    Let $L$ be the simple Lie superalgebra $P(n)$, $n \geq 2$, endowed with any $G$-grading.
    Then $L$ is isomorphic to some $P(T, \beta, \kappa_\bz, h_0)$. 
    Moreover, $P (T, \beta, \kappa_\bz, h_0) \iso P(T', \beta', \kappa_\bz', h_0')$ \IFF $T = T'$, $\beta = \beta'$ and there is an element $g\in G$ such that $\kappa_\bz' = g \cdot \kappa_\bz$ and $h_0' = g^{-2} h_0$. \qed 
\end{thm}

\subsection{Queer Lie superalgebras (Series \texorpdfstring{$Q$}{Q})}\label{subsec:gradings-Lie-Q}

Gradings on the Lie superalgebra $L$ of type $Q(n)$ were classified in \cite{paper-Qn} by reduction to its even part, which is the Lie algebra of type $A_n$.
Using the theory developed in this work, the classification of gradings on $L$ is the same as on the associative superalgebra with superinvolution $Q(n+1) \times Q(n+1)\sop$.

By \cref{cor:iso-Q}, Type I gradings on $L$ are induced by gradings of the form $\Gamma_Q (T^+, \beta^+, h, \kappa)$ on the associative superalgebra $Q(n+1)$ (see \cref{def:Gamma-T-beta-kappa-Q}).
We denote $L$ equipped with this grading by $Q^{\mathrm{(I)}}(T^+, \beta^+, h, \kappa)$.
Note that the classification of these graded superalgebras follows from \cref{thm:QxQ-type-I}, not from \cref{cor:iso-Q}.
By \cref{cor:QxQ-reduced-to-MxM}, Type II gradings on $L$ come from graded superalgebras with superinvolution $Q^{\mathrm{ex}} (T^+, \beta^+, h, \kappa, g_0)$ (see \cref{def:model-grd-QxQ-only}).
We denote $L$ with the induced grading by $Q^{\mathrm{(II)}}(T^+, \beta^+, h, \kappa, g_0)$.

\begin{thm}\label{thm:final-Q(n)}
    Let $n \geq 2$. 
    The simple Lie superalgebra $Q(n)$ endowed with any $G$-grading is isomorphic to either $Q^{\mathrm{(I)}}(T^+, \beta^+, h, \kappa)$ or $Q^{\mathrm{(II)}}(T^+, \beta^+, h, \kappa, g_0)$, and these cases are mutually exclusive. 
    Furthermore:
    
    \smallskip
    \noindent
    \fbox{\emph{Type I}} $Q^{\mathrm{(I)}}(T^+, \beta^+, h, \kappa) \iso Q^{\mathrm{(I)}}(T'^+, \beta'^+, h', \kappa')$ \IFF $T'^+ = T^+$, $h' = h$, and there exists $g \in G$ such that one of the following holds:
	\begin{enumerate}
	    \item $\beta'^+ = \beta^+$ and $\kappa' = g \cdot \kappa$; 
	    \item $\beta'^+ = (\beta^+)\inv$ and $\kappa' = g \cdot \kappa^{\Star}$.
	\end{enumerate}

    \noindent
    \fbox{\emph{Type II}} $Q^{\mathrm{(II)}}(T^+, \beta^+, h, \kappa, g_0) \iso Q^{\mathrm{(II)}}(T'^+, \beta'^+, h', \kappa', g'_0)$ \IFF $T'^+ = T^+$, $\beta'^+ = \beta^+$, $h' = h$, and there exists $g \in G$ such that $\kappa' = g \cdot \kappa$ and $g'_0 = g^{-2} g_0$. \qed
\end{thm}

\subsection{Special linear Lie superalgebras (Series \texorpdfstring{$A$}{A})}\label{subsec:gradings-Lie-A-m-n}

We break the classification of gradings on $A(m,n)$ into two cases, $m \neq n$ and $m=n$.
We first deal with the case $m \neq n$, since it has fewer possibilities and simpler isomorphism conditions, so let $L = \Sl(m+1 \,|\, n+1)$.

Every grading on $L$, whether of Type I or Type II, must be even (see \cref{lemma:odd-M-m=n,cor:associative-type-II-odd-m=n}).
Hence, by \cref{cor:iso-M-even}, Type I gradings on $L$ are induced by gradings of the form $\Gamma_M(T, \beta, \kappa_\bz, \kappa_\bo)$ on the associative superalgebra $M(m+1, n+1)$ (see \cref{def:Gamma-T-beta-kappa-even}).
We denote $L$ with this grading by $A^{\mathrm{(I)}}(T, \beta, \kappa_\bz, \kappa_\bo)$.
Note that the classification of these graded algebras follows from \cref{thm:MxM-type-I}, not from \cref{cor:iso-M-even}.
By \cref{thm:MxM-even}, Type II gradings on $L$ come from graded superalgebras with superinvolution $M^{\mathrm{ex}}(T, \beta, \kappa_\bz, \kappa_\bo, g_0)$ (see \cref{def:model-grd-MxM-even}) with $g_0 \in G = G\times \{\bz\} \subseteq G^\#$ since $m\neq n$ (see \cref{rmk:g_0-odd-implies-m=n}).
We denote $L$ with the induced grading by $A^{\mathrm{(II)}}(T, \beta, \kappa_\bz, \kappa_\bo, g_0)$.
The classification follows from these results and \cref{rmk:m-different-n-even-grading}:

\begin{thm}\label{thm:final-m-not-n} 
    Let $m > n \geq 0$.
    The simple Lie superalgebra $\mathfrak{sl}(m+1 \,|\, n+1)$ with any $G$-grading is isomorphic to either $A^{\mathrm{(I)}}(T, \beta, \kappa_\bz, \kappa_\bo)$ or $A^{\mathrm{(II)}}(T, \beta, \kappa_\bz, \kappa_\bo, g_0)$, and these cases are mutually exclusive. 
    Furthermore:
    
    \smallskip\noindent
    \fbox{\emph{Type I}} $A^{\mathrm{(I)}}(T, \beta, \kappa_\bz, \kappa_\bo) \iso A^{\mathrm{(I)}}(T', \beta', \kappa_\bz', \kappa_\bo')$ \IFF $T' = T$ and there exists $g \in G$ such that one of the following holds:
	\begin{enumerate}
	    \item $\beta' = \beta$, $\kappa_{\bar 0}' = g \cdot \kappa_{\bar 0}$ and $\kappa_{\bar 1}' = g \cdot \kappa_{\bar 1}$; 
	    \item $\beta' = \beta\inv$, $\kappa_{\bar 0}' = g \cdot \kappa_{\bar 0}\Star$ and $\kappa_{\bar 1}' = g \cdot \kappa_{\bar 1}\Star$.
	\end{enumerate}
    
    \noindent
    \fbox{\emph{Type II}} $A^{\mathrm{(II)}}(T, \beta, \kappa_\bz, \kappa_\bo, g_0) \iso A^{\mathrm{(II)}}(T', \beta', \kappa_\bz', \kappa_\bo', g'_0)$ \IFF
    $T' = T$, $\beta' = \beta$, and there exists $g \in G$ such that
    $\kappa_\bz' = g \cdot \kappa_\bz$, $\kappa_\bo' = g \cdot \kappa_\bo$, and $g'_0 = g^{-2} g_0$. \qed
\end{thm}

\subsection{Projective special linear Lie superalgebras (Series \texorpdfstring{$A$}{A})}\label{subsec:gradings-Lie-A-n-n}

Finally, consider $L \coloneqq \mathfrak{psl}(n+1 \,|\, n+1)$, the simple Lie superalgebra of type $A(n,n)$.
The gradings on $L$ can be divided into five subtypes, each of which corresponds, in view of the bijection between gradings and actions (\cref{subsec:grds-and-actions}), to one of the five subgroups of the Klein group $\Aut(L)/\mc E(n+1, n+1) \iso C_2 \times C_2$ (\cref{eq:barr-Aut-S}) according to the image of $\widehat G$.

\paragraph{Even Type I}
The even Type I gradings on $L$ are similar to the Type I gradings on $A(m,n)$ with $m \neq n$. 
By \cref{cor:iso-M-even}, they are induced by gradings of the form $\Gamma_M(T, \beta, \kappa_\bz, \kappa_\bo)$ on the associative superalgebra $M(n+1, n+1)$ (see \cref{def:Gamma-T-beta-kappa-even}).
Since these gradings come from an associative superalgebra of type $M$, we denote $L$ endowed with such grading by $A_M^{\mathrm{(I)}}(T, \beta, \kappa_\bz, \kappa_\bo)$.

\paragraph{Odd Type I}
The odd Type I gradings on $L$ are similar to the Type I gradings on $Q(n)$.
By \cref{cor:iso-M-odd}, they are induced by gradings of the form $\Gamma_M(T, \tilde\beta, \kappa)$ on the associative superalgebra $M(n+1, n+1)$ (see \cref{def:Gamma-T-beta-kappa-odd}).
We denote $L$ endowed with such grading by $A_Q^{\mathrm{(I)}}(T, \tilde\beta, \kappa)$.

\bigskip
Recall that Type II gradings on $L$ are induced from gradings on the superalgebra with superinvolution $M(n+1, n+1)\times M(n+1, n+1)\sop$ making it isomorphic to $E(\D, \U, B)$ (see \cref{thm:vphi-iff-vphi0-and-B,def:superadjunction}).
Hence, the Type II gradings can be divided into three cases as in the construction in \cref{para:even-D-even-B}.

\paragraph{Even Type II with even $B$}
The case of even $\D$ and even $B$ is similar to orthosymplectic superalgebras.
By \cref{thm:MxM-even}, these gradings on $L$ come from graded superalgebras with superinvolution $M^{\mathrm{ex}}(T, \beta, \kappa_\bz, \kappa_\bo, g_0)$ (see \cref{def:model-grd-MxM-even}) with $g_0 \in G = G\times \{\bz\} \subseteq G^\#$.
We denote $L$ with the induced grading by $A_{\mathfrak{osp}}^{\mathrm{(II)}}(T, \beta, \kappa_\bz, \kappa_\bo, g_0)$.

\paragraph{Even Type II with odd $B$}
The case of even $\D$ and odd $B$ is similar to periplectic superalgebras.
By \cref{thm:MxM-even}, these gradings on $L$ come from graded superalgebras with superinvolution $M^{\mathrm{ex}}(T, \beta, \kappa_\bz, \kappa_\bo, g_0)$ (see \cref{def:model-grd-MxM-even}) with $g_0 = (h_0, \bo) \in G^\#$.
Recall from \cref{para:even-D-odd-B} that, in this case, $\kappa_\bo$ is determined by $\kappa_\bz$, so we denote $L$ with the induced grading by $A_{P}^{\mathrm{(II)}}(T, \beta, \kappa_\bz, h_0)$.

\paragraph{Odd Type II}
The case of odd $\D$ and even $B$ is similar to Type II gradings on queer Lie superalgebras.
By \cref{thm:MxM-odd}, these gradings on $L$ come from graded superalgebras with superinvolution $M^{\mathrm{ex}}(T, \tilde\beta, t_p, \kappa, g_0)$
(see \cref{def:model-grd-MxM-odd-or-QxQ}).
We denote $L$ with the induced grading by $A_{Q}^{\mathrm{(II)}}(T, \tilde\beta, t_p, \kappa, g_0)$.

\medskip
The classification for Type I follows from \cref{thm:MxM-type-I} (not from \cref{cor:iso-M-even,cor:iso-M-odd}), and for Type II from \cref{thm:MxM-even,thm:MxM-odd}.

\begin{thm}\label{thm:final-A(n-n)}
    Let $n > 1$.
    The simple Lie superalgebra $\mathfrak{psl}(n+1 \,|\, n+1)$ with any $G$-grading is isomorphic to one of 
    $A_M^{\mathrm{(I)}}(T, \beta, \kappa_\bz, \kappa_\bo)$, 
    $A_Q^{\mathrm{(I)}}(T, \tilde\beta, \kappa)$, 
    $A_{\mathfrak{osp}}^{\mathrm{(II)}}(T, \beta, \kappa_\bz, \kappa_\bo, g_0)$, 
    $A_{P}^{\mathrm{(II)}}(T, \beta, \kappa_\bz, h_0)$, or 
    $A_{Q}^{\mathrm{(II)}}(T, \tilde\beta, t_p, \kappa, g_0)$, 
    and these cases are mutually exclusive.
    Furthermore:
    
    \smallskip\noindent
    \fbox{\emph{Type I\textsubscript{M}}} $A_M^{\mathrm{(I)}}(T, \beta, \kappa_\bz, \kappa_\bo) 
    \iso 
    A_M^{\mathrm{(I)}}(T', \beta', \kappa_\bz', \kappa_\bo')$ \IFF $T = T'$ and there exists $g \in G$ such that one of the following holds:
    \begin{enumerate}
        \item $\beta' = \beta$ and either
            \begin{itemize}
                \item $\kappa_{\bar 0}' = g \cdot \kappa_{\bar 0}$ and $\kappa_{\bar 1}' = g \cdot \kappa_{\bar 1}$, or
                \item $\kappa_{\bar 0}' = g \cdot \kappa_{\bar 1}$ and $\kappa_{\bar 1}' = g \cdot \kappa_{\bar 0}$;
            \end{itemize}
        \item $\beta' = \beta\inv$ and either
            \begin{itemize}
                \item $\kappa_{\bar 0}' = g \cdot \kappa_{\bar 0}\Star$ and $\kappa_{\bar 1}' = g \cdot \kappa_{\bar 1}\Star$, or
                \item $\kappa_{\bar 0}' = g \cdot \kappa_{\bar 1}\Star$ and $\kappa_{\bar 1}' = g \cdot \kappa_{\bar 0}\Star$.
            \end{itemize}
    \end{enumerate}
    
    \noindent
    \fbox{\emph{Type I\textsubscript{Q}}}
    $A_Q^{\mathrm{(I)}}(T, \tilde\beta, \kappa)
    \iso
    A_Q^{\mathrm{(I)}}(T', \tilde\beta', \kappa')$ \IFF $T' = T$ and there exists $g \in G$ such that one of the following holds:
    \begin{enumerate}
	    \item $\tilde\beta'=\tilde\beta$ and $\kappa' = g \cdot \kappa$;
	    \item $\tilde\beta'=\tilde\beta\inv$ and $\kappa' = g \cdot \kappa\Star$.
	\end{enumerate}
    
    \noindent
    \fbox{\emph{Type II\textsubscript{$\mathfrak{osp}$}}}
    $A_{\mathfrak{osp}}^{\mathrm{(II)}}(T, \beta, \kappa_\bz, \kappa_\bo, g_0) 
    \iso 
    A_{\mathfrak{osp}}^{\mathrm{(II)}}(T', \beta', \kappa_\bz', \kappa_\bo', g'_0)$ \IFF
    $T = T'$, $\beta = \beta'$, and there is $g \in G$ such that either
    \begin{enumerate}
        \item $\kappa_\bz' = g \cdot \kappa_\bz$, $\kappa_\bo' = g \cdot \kappa_\bo$ and $g'_0 = g^{-2} g_0$, or
        \item $\kappa_\bz' = g \cdot \kappa_\bo$, $\kappa_\bo' = g \cdot \kappa_\bz$ and $g'_0 = f g^{-2} g_0$.
    \end{enumerate}
    
    \noindent
    \fbox{\emph{Type II\textsubscript{P}}}
    $A_{P}^{\mathrm{(II)}}(T, \beta, \kappa_\bz, h_0)
    \iso 
    A_{P}^{\mathrm{(II)}}(T', \beta', \kappa_\bz', h'_0)$ \IFF
    $T = T'$, $\beta = \beta'$, and there is $g \in G$ such that either
    \begin{enumerate}
        \item $\kappa_\bz' = g \cdot \kappa_\bz$ and $h'_0 = g^{-2} h_0$, or
        \item $\kappa_\bz' = g h_0^{-1} \cdot \kappa_\bz\Star$ and $h'_0 = f g^{-2} h_0$.
    \end{enumerate}
    
    \noindent
    \fbox{\emph{Type II\textsubscript{Q}}}
    $A_{Q}^{\mathrm{(II)}}(T, \tilde\beta,  t_p, \kappa, g_0) 
    \iso 
    A_{Q}^{\mathrm{(II)}}(T', \tilde\beta',  t_p', \kappa', g_0')$ \IFF
    $T =T'$, $\tilde\beta = \tilde\beta'$, $t_p = t_p'$ and there exists $g \in G$ such that $\kappa' = g\cdot\kappa$ and $g_0' = g^{-2}g_0$. \qed
\end{thm}

\bibliographystyle{amsalpha}
\bibliography{bibliography}

\end{document}